\documentclass{article}
\usepackage{amsmath}
\usepackage{amssymb}
\usepackage{amsthm}
\usepackage{mathrsfs}
\usepackage{enumerate}
\usepackage{color}
\usepackage{geometry}
\usepackage{bbm}
\usepackage{slashed}
\usepackage{hyperref} 
\usepackage{enumerate,cancel,amssymb,centernot}
\usepackage{latexsym}
\usepackage{graphicx}
\usepackage{tikz}
\usepackage{tikz-network}
\usepackage{chngcntr}
\usepackage{amsmath}
\allowdisplaybreaks[4]
\usetikzlibrary{shapes,arrows}

\newcounter{cnstcnt}
\newcommand{\newconstant}{%
\refstepcounter{cnstcnt}%
\ensuremath{c_{\thecnstcnt}}}
\newcommand{\oldconstant}[1]{\ensuremath{c_{\ref*{#1}}}}
\newtheorem{thm}{Theorem}[section]
\newtheorem{lem}[thm]{Lemma}
\newtheorem{cor}[thm]{Corollary}

\newtheorem{rmk}[thm]{Remark}
\newtheorem{prop}[thm]{Proposition}
\newtheorem{defi}[thm]{Definition}
\numberwithin{equation}{section}

\def\mbb#1{\mathbb{#1}}
\def\mcc#1{\mathcal{#1}}
\def\mss#1{\mathscr{#1}}
\def\pari{\intB}

\def\eps{\epsilon}
\def\veps{\varepsilon}
\def\lamn{\Lambda_N}
\def\P{\mathbb{P}}
\def\cA{\mathcal{A}}

\newcommand{\1}{\mathbf{1}}
\newcommand{\cC}{\mathcal{C}}
\newcommand{\cH}{\mathcal{H}}
\newcommand{\cV}{\mathcal{V}}
\newcommand{\cZ}{\mathcal{Z}}
\newcommand{\cB}{\mathcal{B}}
\newcommand{\cE}{\mathcal{E}}
\def\G{G}

\newcommand{\extB}{\partial_{\mathrm{ext}}}
\newcommand{\intB}{\partial}
\newcommand{\edge}{E}
\newcommand{\E}{{\mathbb E}}
\def\0{\textbf{0}}

\newcommand{\R}{{\mathbb R}}

\def \fkh#1{\phi_{p,\lamn,\epsilon h}^{{#1}}}

\def \fk#1{\phi_{p,\lamn,0}^{#1}}

\def \vfk#1{\varphi_{p,\lamn,0}^{#1}}
\def \vfkh#1{\varphi_{p,\lamn,\epsilon h}^{#1}}

\def\sB {\mathscr B}

\def\cD {\mathcal D}
\def\rc {\mathtt{Cross}}
\def\hc {\mathtt{Hcross}}
\def\aro{\mathtt{Around}}
\def\daro{\mathtt{DAround}}

\def\ne {\mathtt{NoEx}}
\def\cV {\mathcal V}

\def\cR {\mathcal R}
\def\cB {\mathcal B}

\def\cL {\mathcal L}
\def\con{\mathtt{Con}}
\def\cont{\mathtt{Con_2}}
\def\cRp {\mathcal R^{\prime}}

\def\fQ {\mathfrak Q}

\renewcommand{\vert}{V}

\def\cT{\mathcal{T}}
\newcommand{\nc}{\newconstant}
\newcommand{\oc}{\oldconstant}

\DeclareSymbolFont{sfoperators}{OT1}{ptm}{m}{n}

\makeatletter
\def\operator@font{\mathgroup\symsfoperators}
\makeatother

\usepackage{graphicx} 

\title{A phase transition and critical phenomenon for the two-dimensional random field Ising model}
\author{Jian Ding\\Peking University\and Fenglin Huang \\Peking University \and Aoteng Xia\\Peking University}
\date{}

\begin{document}
\maketitle

\begin{abstract}
We study the random field Ising model in a two-dimensional box with side length $N$ where the external field is given by independent normal variables with mean $0$ and variance $\epsilon^2$. Our primary result is the following phase transition at $T = T_c$: for $\epsilon \ll N^{-7/8}$ the boundary influence (i.e., the difference between the spin averages at the center of the box with the plus and the minus boundary conditions) decays as $N^{-1/8}$ and thus the disorder essentially has no effect on the boundary influence; for $\epsilon \gg N^{-7/8}$, the boundary influence decays as $N^{-\frac{1}{8}}e^{-\Theta(\epsilon^{8/7} N)}$ (i.e., the disorder contributes a factor of $e^{-\Theta(\epsilon^{8/7} N)}$ to the decay rate). For a natural notion of the correlation length, i.e., the minimal size of the box where the boundary influence shrinks by a factor of $2$ from that with no external field, we also prove the following: as $\epsilon\downarrow 0$ the correlation length transits from $\Theta(\epsilon^{-8/7})$ at $T_c$ to $e^{\Theta(\epsilon^{-4/3})}$ for $T < T_c$. 
\end{abstract}

\section{Introduction}

For $d\geq 2$, let $\mathbb Z^d$ be the $d$-dimensional lattice  where $u, v\in \mathbb Z^d$ are adjacent (which we denote as $u\sim v$) if their $\ell_1$-distance is $1$.
For $N\geq 1$, let $\lamn=[-N,N]^d \cap \mathbb Z^d$ be the box of side length $2N$ centered at the origin $o$. For $v\in \mathbb Z^d$, let $h_v$'s be independent normal variables with mean $0$ and variance $1$ (in what follows we denote by $\mathbb P$ and $\mathbb E$ the measure and the expectation with respect to  $h=\{h_v: v\in \mathbb Z^d\}$, respectively). For $\epsilon \geq 0$, we define the random field Ising model (RFIM) Hamiltonian $H^{\pm}_{ \lamn, \epsilon h}$ with the plus (respectively minus) boundary condition and the external field  $\eps h=\{ \epsilon h_v :v \in {\mbb Z^d} \}$ by
\begin{equation}
\label{def_h}
H^{\pm}_{\lamn, \epsilon h} (\sigma)= - \left(\sum_{\substack{u,v \in \Lambda_{N-1},\\\, u\sim v}} \sigma_u \sigma_v \pm \sum_{\substack{ u \in \Lambda_{N-1},\\ v\in \Lambda_{N-1}^c,\\\, u\sim v}} \sigma_u + \sum_{u \in \Lambda_{N-1}} \epsilon h_u \sigma_u\right) \mbox{ for } \sigma \in \{-1, 1\}^{\Lambda_{N-1}}\,.
\end{equation}
For $T>0$, we define $\mu^{\pm}_{T, \lamn, \epsilon h} $ to be the Gibbs measure on $\{-1,1\}^{\Lambda_{N-1}}$ at temperature $T$ by
\begin{equation}
\label{def_mu}
\mu^{\pm}_{T, \lamn, \epsilon h} (\sigma)= \frac{1}{\mathcal{Z}^{\pm}_{T, \lamn,\eps h}} e^{-\frac{1}{T} H^{\pm}_{ \lamn, \epsilon h} (\sigma)} \mbox{ for } \sigma\in \{-1, 1\}^{\Lambda_{N-1}}\,,
\end{equation}
where  $\mathcal{Z}^{\pm}_{T, \lamn,\eps h}$ is the partition function given by
\begin{equation}
\label{def_z}
\mathcal{Z}^{\pm}_{T, \lamn,\eps h} = \sum_{\sigma \in \{-1,1\}^{\Lambda_{N-1}}} e^{-\frac{1}{T} H^{\pm}_{ \lamn,\epsilon h} (\sigma)}.
\end{equation}
(Note that $\mu^{\pm}_{T, \lamn, \epsilon h}$ and $\mathcal Z^{\pm}_{T, \lamn,\eps h} $ are random variables depending on  the external field $\eps h$.) We will denote by  $\langle \cdot \rangle^{\pm}_{T,\lamn,\eps h}$ the average taken with respect to $\mu^{\pm}_{T, \lamn,\eps h}$.  The key quantity we are interested in is the \textbf{boundary influence}, defined as
\begin{equation}\label{eq-def-boundary-influence}
m(T,N,\eps)=\mbb E(\mathbf m(T, N, \epsilon h)) \mbox{ where } \mathbf m(T, N, \epsilon h)  = \frac{1}{2}\Big(\langle\sigma_o\rangle^{+}_{T,\lamn,\eps h}-\langle\sigma_o\rangle^{-}_{T,\lamn,\eps h}\Big)\,.
\end{equation}
In the case of no disorder (i.e., when $\epsilon = 0$),  there exists $T_c = T_c(d)\in (0, \infty)$ (for $d\geq 2$) with the following phase transition. 
\begin{itemize}
\item The boundary influence stays above a {positive} constant for $T<T_c$ \cite{Peierls1936}.
\item The boundary influence decays exponentially for $T > T_c$ \cite{ABF87} (see also \cite{DT16, DRT19}).
\item For $T=T_c,$ the behavior of the boundary influence remains partially understood:
\begin{itemize}
    \item For $d=2$, the boundary influence decays as $N^{-\frac{1}{8}}$; see $\cite{Onsager44, Yang52, MW73, DHN11, CHI15, KS19, GW20}$.
    \item For $d\geq 3$, using  the so-called reflection positivity property,  it was proved in \cite{AF86, ADS15,FSS76} that the boundary influence decays to $0$ with a polynomial lower bound. 
    \item For more general Ising models without reflection {positivity}, 
    the same result is only proven for sufficiently large $d$ using the lace expansion; see \cite{Sak07, CS15,Sak22}. Furthermore, the one-arm exponent is proved to be at most $1$ in the same setting in \cite{HHS19}.
\end{itemize}
\end{itemize}

In this work, we focus on the disordered case for $d=2$. Our primary result is the following two theorems, which provide a detailed description of the phase transition as $\epsilon$ varies at the critical temperature $T_c = T_c(2)$. We declare that all our statements hold for all $N\geq 1$.
\begin{thm}\label{thm-critical-temperature}
Fix $d=2$. For any constant $\newconstant\label{1}>0$, there exists a constant
$ \newconstant\label{2}=\oc{2}(\oc{1})>0$ such that for {$\eps\geq\oldconstant{1}N^{-\frac{7}{8}}$},
\begin{equation}\label{eq-main-theorem}
   \oldconstant{2}N^{-\frac{1}{8}}  \exp\Big(-\oldconstant{2}^{-1}\epsilon^{\frac{8}{7}} N\Big) \leq  m(T_c,N,\eps) \leq \oldconstant{2}^{-1}N^{-\frac{1}{8}} \exp\Big(-\oldconstant{2} \epsilon^{\frac{8}{7}} N\Big)\,. 
\end{equation}
Furthermore, there exist constants $\nc\label{-1},\nc\label{-2}>0$ such that for {$\eps\geq\oldconstant{-1}N^{-\frac{7}{8}}$} the following holds with $\mathbb P$-probability at least {$1-\oc{-2}^{-1}\exp\Big(-\oc{-2}(\eps^{\frac{8}{7}}N)^{\frac{1}{10}}\Big)$}: 
\begin{align}
    \oldconstant{-2}N^{-\frac{1}{8}}\exp(-\oldconstant{-2}^{-1}\epsilon^{\frac{8}{7}} N)\le \mathbf m(T_c,N,\eps h)&\le \oldconstant{-2}^{-1}N^{-\frac{1}{8}}  \exp(-\oldconstant{-2}\epsilon^{\frac{8}{7}} N).\label{eq-main-theorem-lower-bound}
\end{align}
\end{thm}
\begin{thm}\label{thm-critical-temperature-small-perturbation}
    Fix $d=2$. For any constant $\theta\in(0,1)$, there exists a constant $\newconstant\label{4}>0$ depending only on $\theta$ such that the following holds. If $0<\eps\le \oldconstant{4}N^{-\frac{7}{8}}$, then
     $$(1-\theta)\cdot m(T_c,N,0)\le m(T_c,N,\eps)\leq m(T_c,N,0)\,.$$ 
Furthermore, there exist constants $\nc\label{-3},\nc\label{-4}>0$ such that for $0<\eps\le\oldconstant{-3}N^{-\frac{7}{8}}$ 
the following holds with $\mathbb P$-probability at least $1-\oc{-4}\exp\left(-\oc{-4}^{-1}\sqrt{\eps^{-1} N^{-\frac{7}{8}}}\right)$: 
\begin{equation}\label{eq-small-perturbation}
    (1-\theta)\cdot m(T_c,N,0)\le \mathbf m(T_c,N,\eps h)\le m(T_c,N,0)\,.
\end{equation}
\end{thm}
From Theorems~\ref{thm-critical-temperature} and \ref{thm-critical-temperature-small-perturbation}, we see that at $T_c$ there is a  phase transition around $\epsilon \asymp N^{-7/8}$, above which the disorder contributes a factor of $e^{-\Theta(\epsilon^{8/7} N)}$ to the decay rate of the boundary influence and below which the disorder essentially has no effect on the boundary influence.

In order to describe our result for $T<T_c$, we define (a notion of) the \textbf{correlation length} by
    \begin{equation}
\label{def_psi}
\psi(T, \mathsf m, \epsilon) =\min \{N:  m(T,N,\eps) \leq \mathsf m\}, \mbox{ for }\mathsf m\in (0, 1).
\end{equation}
\begin{thm}\label{thm-supercritical-temperature}
Fix $d=2$. For  $0<T<T_c$ and $0<\mathsf m<m(T,\infty,0)$,
there exist constants $\nc\label{5} = \oc{5}(T,\mathsf m)>0$ and $\eps_0=\eps_0(T,\mathsf m)>0$ such that for { $0 <  \eps \leq \eps_0$, we have }
$$
\exp{(\oc{5} \epsilon^{-4/3})} \leq  \psi(T, \mathsf m, \epsilon) \leq \exp{(\oc{5}^{-1}\epsilon^{-4/3})}.
$$
\end{thm}
We may also consider the following notion of correlation length which is well-defined at all $T$ (and in particular is consistent with \eqref{def_psi} up to changing some constants):
$$ {\psi_\star(T, \epsilon)} = \min\{N: m(T, N, \eps) \leq m(T, N, 0)/2\}\,.$$
Then, as a particular consequence of Theorems~\ref{thm-critical-temperature},  \ref{thm-critical-temperature-small-perturbation} and \ref{thm-supercritical-temperature}, for $d=2$ the following holds for some constants ${c=c(T)}>0$:
\begin{equation}\label{eq-correlation-length-critical}
\begin{split}
c^{-1}\epsilon^{-8/7} &\leq  {\psi_\star(T_c, \epsilon)} \leq c \epsilon^{-8/7} \mbox{ as }\epsilon \to 0\,,\\
c^{-1} e^{c^{-1}\epsilon^{-4/3}}& \leq  {\psi_\star(T, \epsilon)} \leq c e^{c\epsilon^{-4/3}} \mbox{ as }\epsilon \to 0 \mbox{ for } T < T_c\,.
\end{split}
\end{equation}
The bounds in \eqref{eq-correlation-length-critical} illustrate an essentially complete description on how this correlation length transits from $T_c$ to $T<T_c$. For the notation of correlation length which is the inverse of the rate for the exponential decay, we do not establish a phase transition as in \eqref{eq-correlation-length-critical}. This is because while Theorem \ref{thm-critical-temperature} computes it up to constant at $T_c$, currently the method for $T < T_c$ falls short of bounding this correlation length.
Finally, we remark that we did not consider the case for $T>T_c$ since it is known that the boundary influence decays exponentially with an arbitrary external field by \cite{ABF87} and \cite[Theorem 1.1]{DSS22}\footnote{In March 2023, Aizenman informed J.D. during J.D.'s visit at IAS, that he proved \cite[Theorem 1.1]{DSS22} using the random current model in his unpublished notes back to 1980's.}.

\subsection{Previous works}
While our results lie under the umbrella of the general Imry--Ma \cite{IM75} phenomenon on the effect of
disorder on phase transitions in two-dimensional physical systems, it seems fair to say that by providing a rather detailed description at criticality our Theorems~\ref{thm-critical-temperature} and \ref{thm-critical-temperature-small-perturbation} may have even gone beyond what  Imry--Ma predicted (which, to our best knowledge, seemed to mainly concern the low temperature regime). While physicists have made attempts on critical or near-critical RFIM (see, e.g., \cite{FVP19, FPS18, RY99}), we are unable to find in the literature predictions regarding Theorems \ref{thm-critical-temperature} and \ref{thm-critical-temperature-small-perturbation}. In what follows, we describe mathematical progress on the RFIM.

We will only focus our discussions on the case of weak disorder (that is, when $\epsilon$ is small, either a small fixed constant or vanishing at a certain rate), as this is most relevant to our results. For $d\geq 3$, it was predicted by \cite{IM75} that at low temperatures for small (and fixed) $\epsilon$, the boundary influence stays above a certain constant as $N \to \infty$. This was proved by \cite{Imbrie85} (for $T = 0$) and by \cite{BK88} (for small $T>0$) using the renormalization group theoretic approach. Recently, a new and simple proof (without renormalization group) was provided in \cite{DZ21} (with critical inputs from \cite{Chalker83, FFS84}), building on which this result was extended to the entire low temperature regime in \cite{DLX22}. 

For $d=2$, it was predicted by \cite{IM75} that the boundary influence vanishes for any fixed $\epsilon$, and this was established in \cite{AW89, AW90}. Recent progress in this direction mainly concerned quantitative bounds on the decay rate and in particular on the existence of the so-called BKT transition (i.e., a transition from polynomial decay to exponential decay). Progressively improved bounds were obtained in \cite{Cha18, AP19, DX21, AHP20}, and exponential decay  was finally proved in \cite{DX21, AHP20} (which implies no BKT transition). Afterwards, it was established in \cite{DW20} that  the correlation length scales as $e^{\Theta(\epsilon^{-4/3})}$, which seems unexpected from the point of view of predictions in physics. The upper bound on the correlation length was proved for all temperatures in \cite{DW20}, while the lower bound was initially only proved at $T = 0$ in \cite{DW20}. This was later extended to small $T>0$ in \cite{DZ21}, and now to the entire low temperature regime as incorporated in Theorem~\ref{thm-supercritical-temperature}. At $T_c$, a continuum RFIM was constructed in \cite{BS22} via scaling limits of the magnetization field (which in particular is singular to the two-dimensional continuum pure Ising model constructed in \cite{CGN15}), and an exponential decay was proved in \cite{CJN23} for the near-critical Ising model with vanishing but constant external field. In a direction somewhat different from ours, spin glass features for the RFIM were studied in \cite{Cha15, Cha23}.

\subsection{Proof ingredients}

We first briefly make a remark on Theorem~\ref{thm-supercritical-temperature}: this extends the results of \cite{DW20, DZ21} to the entire low temperature regime, and the extension essentially follows the framework developed in \cite{DLX22}. Now we emphasize that the main novelty of this work is on the critical phenomenon for the RFIM, as incorporated in Theorems \ref{thm-critical-temperature} and \ref{thm-critical-temperature-small-perturbation}. Our proofs of these two theorems combine in a novel way many interesting ingredients as we elaborate as follows.
\begin{itemize}
\item The critical exponent of $7/8$ was hinted in \cite{BS22}, from which we also learned the method of chaos expansion employed in \cite{FSZ16, BS22}. The method of chaos expansion allows us to show that when the disorder is weak it `essentially' has no effect (here essentially refers to the fact that the notion of no effect is much weaker than absolute continuity of the measures). This is not only used in the proof for the lower bound on the boundary influence,  but also used in a sophisticated manner in the proof for the upper bound and perhaps most prominently in the proof of the Russo-Seymour-Welsh (RSW) estimate with disorder (as in Theorem~\ref{thm: RSW for dis with external field}). 

\item The proof of the (stretched) exponential decay is inspired by \cite{DX21, AHP20} which proves exponential decay for fixed $\epsilon$ at any temperature. In particular, the crucial application of \cite{AB99} (initiated in the proof for $T = 0$ in \cite{DX21}) remains an important ingredient in our proof. In addition,  proofs in both \cite{DX21} and \cite{AHP20} employ a relation between  the boundary influence and the percolation of disagreements: the proof in \cite{DX21} employs a particular monotone coupling between Ising measures with the plus and the minus boundary conditions, and the proof in \cite{AHP20} works with the product measure for two \emph{extended} Ising measures. It turns out that the disagreement percolation as in \cite{AHP20} (see Section~\ref{sect:disagreement percolation} for a review) is more convenient for the current work, and this is what we will employ. 

\item Another crucial ingredient in our proof is to take advantage of \emph{criticality}. One way to capture the critical temperature is through the RSW estimate (without disorder) for the FK-Ising model  \cite{DHN11} (see also e.g. \cite{BDC12, Tassion16, KT23} for remarkable further development{s} on the RSW theory), and this is exactly how we access the criticality. For instance, this is how we verify the tortuous condition (as formulated in \cite{AB99}) for the disagreement percolation (see Lemmas \ref{lem:crossing} and \ref{lem:criterion}) and this step is very different from that of \cite{DX21, AHP20}. 

\item At times, we need to prove that some bounds hold with arbitrary disorder assuming they hold with no disorder. To this end, we apply a general Ising inequality proved recently in \cite{DSS22}; see, e.g., \eqref{eq-apply-DSS} for an instance of such an application.
\end{itemize} 
Of course, the discussions above only offer a glance at various ingredients employed in our work, and it remains highly non-trivial to assemble them into a working proof. After reviewing various preliminaries in Section~\ref{sec:prelim}, we outline our proofs (in a rather elaborative manner) for Theorems \ref{thm-critical-temperature} and \ref{thm-critical-temperature-small-perturbation} in Section~\ref{sec:outline}. In Section~\ref{sec: RSW}, we prove the important RSW estimates with disorder, and in Section~\ref{sec:postponed} we supply postponed proofs in Section~\ref{sec:outline}. In Section~\ref{sect:correlation length}, we prove Theorem~\ref{thm-supercritical-temperature}, and this is rather separate from Sections~\ref{sec:outline}, \ref{sec: RSW} and \ref{sec:postponed}.

\section{Preliminaries}\label{sec:prelim}
In this section, we review various existing (mostly well-known) results on the Ising model and some closely related models such as the FK-Ising model. We always consider an arbitrary temperature $T$ unless otherwise stated and we might drop it from the subscript without further notice.
\subsection{Notations}\label{sec: notations}
In this subsection, we describe the notations used in the following sections. We use $c,C,C_i$ to represent positive constants whose actual values may vary from line to line. In addition, we use $c_i$ for constants whose values are fixed {throughout the paper}.

For $\Gamma\subset \mathbb{Z}^2$, we use the notation $\intB\Gamma$ to denote its interior boundary. That is, $\intB \Gamma=\{u\in\Gamma\mid u\sim v \text{ for some } v\in \Gamma^c \}$. For any configuration $\omega\in\mathbb{R}^{\Gamma}$ {and $\Gamma'\subset\Gamma$} we use the notation $\omega|_{\Gamma'}$ to denote its restriction on $\Gamma'$, i.e., $\omega|_{\Gamma'}(u)=\omega_u$ for any $u\in \Gamma'$.

We will use the notation $\mu_{G,h}^{A,\xi}$ to denote the Ising measure on a graph $G$ with the boundary condition $\xi$ on $A\subset G$ and with the external field $h$. This is a natural extension of \eqref{def_mu} except that the Hamiltonian is modified to the following: \begin{equation*}
H^{{A},\xi}_{G, h} (\sigma)= - \left(\sum_{\substack{u,v \in G\setminus A,\\\, u\sim v}} \sigma_u \sigma_v {+} \sum_{\substack{ u \in G\setminus A,\\ v\in A,\\\, u\sim v}} \sigma_u\xi_v + \sum_{u \in G\setminus A} h_u \sigma_u\right) \mbox{ for } \sigma \in \{-1, 1\}^{G\setminus A}\,.
\end{equation*}
We may drop $G$ and $A$ in the notation when they are clear from the context. For any subset $A$ and any external field $h$, let $h_A = \sum_{v\in A} h_v$.

Recall that the $\ell^{\infty}$-distance in $\mbb Z^2$ is defined to be $\|u-v\|_{\infty}=\max\{|x_u-x_v|,|y_u-y_v|\}$ for $u=(x_u,y_u)$ and $v=(x_v,y_v)$, and we will write $dist(u,v)=\Vert u-v\Vert_{\infty}$.

We call $\cR\subset\mathbb Z^2$ an $M$-box if $\cR$ is a translation of $\Lambda_M$ in $\mathbb Z^2$.
And we also denote by $\Lambda_M(x)$ the $M$-box centered at $x$. For integers $0<M<N,$ we use $\Lambda_{M,N}$ to denote the annulus $\Lambda_N\setminus\Lambda_M$. We also denote by $\Lambda_{M,N}(x)$ the annulus centered at $x$. For convenience of exposition, we modify the definition of $\partial \Lambda_{M, N}$ by $$\intB\Lambda_{M,N}=\intB\Lambda_N\cup \intB \Lambda_{M+1}.$$ (Note that the above modification adds four corner points to the boundary and this is immaterial.) We will refer to $\partial \Lambda_N$ and $\partial \Lambda_{M+1}$ as the outer and inner boundaries of the annulus $\Lambda_{M, N}$, respectively.

For any region $\Gamma_1\subset\Gamma_2$, we define a circuit $\cC$ in $\Gamma_1\setminus\Gamma_2$ to be a subset of $\Gamma_1\setminus\Gamma_2$ such that for any $x\in\intB \Gamma_1$, $y\in\intB \Gamma_2$ and any path $\gamma$ connecting $x$ and $y$, we have $\gamma$ intersects with $\cC$.

Throughout the paper, for $a,b>0$, we will use the notation $a\ll b$ to denote that $a<b\cdot c$ for some $c>0$ small enough. Similarly, we use the notation $a\gg b$ to denote that $a>b\cdot c$ for some $c>0$ big enough.

\subsection{FK-Ising model}

In this subsection, we briefly review the FK-Ising model with an external field, which was introduced in \cite{ES88}. For a finite graph $G=(V,E)$, let $\mathcal P, \mathcal M\subset V$ denote the subsets with the plus and minus boundary conditions respectively, and let $\{h_x:x\in V\}$ be an external field. We can consider an edge configuration $\omega\in\Omega=\{0,1\}^E$ where $0$ indicates that the edge is \textbf{closed} and $1$ indicates the edge is \textbf{open}. For $\omega \in \Omega$ and $x \neq y\in V$, we say $x$ is \textbf{connected} to $y$ in $\omega$ and denote it by $x\longleftrightarrow y$ if there exists a sequence $x_0=x,x_1,\cdots,x_{n}=y$ such that $\{x_i,x_{i+1}\}\in E$ and $\{x_i,x_{i+1}\}$ is open in $\omega$ for all $0\leq i\leq n-1$; this sequence is called an \textbf{open path} in $\omega$. So given any $\omega\in\Omega$, the graph $G$ is divided into a disjoint union of connected components, which will be referred to as \textbf{open clusters}. 

We use the notation $\cC_+$ and $\cC_-$ to denote the cluster connected to the plus and minus boundaries, respectively. Here we view $\mcc P$ and $\mcc M$ as single points,
and thus $\mathcal C_\pm$ refers to the unique cluster of $\mcc P$ and $\mcc M$ respectively. Note that it is possible to have $\mathcal C_+  = \mcc P$, and we use the convention that $\mathcal C_+ = \emptyset$ if $\mcc P = \emptyset$ (and similarly for $\mathcal C_-$).   
Let $\mathfrak C$ be the collection of other open clusters (so excluding $\mathcal C_\pm$) in $\omega$.

Then the FK-Ising model on $G$ with parameter $p \equiv 1-\exp(-\frac{2}{T})$ is a probability measure on  $\Omega$ given by (recalling our notation convention $h_A = \sum_{v\in A} h_v$)

\begin{equation}\label{eq-def-FK-external-field}
\phi^{\xi}_{p,G,h}(\omega)=\frac{1}{\cZ^{\xi,\phi}_{p,G,h}}  \prod_{e\in E}\hspace{0.5em}p^{\omega(e)}(1-p)^{1-\omega(e)}\prod_{\mathcal C\in \mathfrak C} 2 \cosh(\frac{h_\mathcal C}{T})\times \exp( \frac{h_{\cC_+}-h_{\cC_-}}{T})\times \1_{\{{\mcc P\centernot \longleftrightarrow\mcc M}\}}\,,
\end{equation}
where $\cZ^{\xi,\phi}_{p,G,h}$ is the normalizing constant (partition function) of $\phi^{\xi}_{p,G,h}$ and $\xi$ denotes the boundary condition, i.e., $\xi_u = 1 \text{ for } u\in \mcc P \text{ and } \xi_u = -1 \text{ for } u\in \mcc M$. 
\begin{rmk}
    Somewhat different from the usual occurrence of the FK-Ising model, in this paper, we allow the boundary conditions to have nontrivial plus and minus parts simultaneously. This gives a restriction that  ${\mcc P}$ cannot be connected to ${\mcc M}$ in $\omega$, { which is why we have the term $\1_{\{{\mcc P\centernot \longleftrightarrow\mcc M}\}}$.}
\end{rmk}
\begin{rmk}
    For boundary conditions without restriction to be plus or minus, the FK-Ising measure can be defined similarly, see \cite{D19} (and \cite{DLX22} for the version with an external field) for a detailed definition. In this case, we will use the notation $\mathrm w$ and $\mathrm f$ to denote the wired and free boundary conditions respectively.
\end{rmk}

In \cite[Lemma 26]{DHN11} (see also \cite{CHI15,GW20}), it was shown that for {$d=2$ and} $ p = p_c$ (i.e., $p =1-\exp({-}\frac{2}{T_c})$) there exists a constant $\nc\label{fk exponent}>0$ such that for any $n>0$,\begin{equation}\label{eq: FK one arm event exponent}
\oc{fk exponent}^{-1}n^{-\frac{1}{8}}\le \phi_{\Lambda_n,0}^{\mathrm{w}}(o\longleftrightarrow\intB\Lambda_n)\leq \oc{fk exponent}n^{-\frac{1}{8}}.
\end{equation}
 {Using Edward-Sokal coupling \cite{ES88} we get $m(T_c,n,0)=\phi_{\Lambda_n,0}^\mathrm{w}(o\longleftrightarrow\intB\Lambda_n)$,
and thus} we obtain the following bound on the boundary influence in the $0$-external field case: \begin{equation}\label{eq: origin decay rate without disorder}
   \oc{fk exponent}^{-1}n^{-\frac{1}{8}}\le m(T_c,n,0)\le\oc{fk exponent}n^{-\frac{1}{8}}.
\end{equation}

Furthermore, using some quasi-{multiplicativity} argument (see \cite[Theorem 1.3]{CDH16}) and the RSW theory, one can show that at $p = p_c$ there exists a constant \nc\label{annulus-crossing}$>0$ such that for any integers $m<n$,
\begin{equation}\label{0-field-annulus-crossing}
    \oc{annulus-crossing}^{-1}(\frac{n}{m})^{-\frac{1}{8}}\leq \phi^{\mathrm{w},\mathrm{w}}_{\Lambda_{m,n},0}(\intB \Lambda_{m+1}\longleftrightarrow\intB\Lambda_n)\leq \oc{annulus-crossing}(\frac{n}{m})^{-\frac{1}{8}}.
\end{equation}
Here the $\mathrm{w},\mathrm{w}$ in the superscript indicates that we have wired boundary conditions on both the inner and outer boundaries of $\Lambda_{m,n}$.

\subsection{ A continuous extension of the Ising model}\label{sect:extended-Ising}

As introduced in \cite{AHP20}, one may extend the Ising model to both vertices and edges by assigning a random value from $\{-1, 0, 1\}$ to each edge. We shall use $\bar{\sigma}$ to denote the extended configuration on a graph $G$ (and we use ${V}(G)$ and use ${ E}(G)$ to denote the vertex and edge set of $G$ from now on) which satisfies
\begin{enumerate}[(i)]
\item for any vertex $v\in {V}(G)$, we have $\bar{\sigma}_v \in \{-1,+1\}$;
\item for any edge $e\in {E}(G)$, we have $\bar{\sigma}_e \in \{-1,0,+1\}$;
\item\label{hard constraint} if $v$ is an endpoint of edge $e$ (and we write $v\in e$), then $|\bar{\sigma}_v- \bar{\sigma}_{e}|\le 1$.
\end{enumerate}
We see from~(\ref{hard constraint}) that if the two endpoints of an edge $e$ have value $1$ and $-1$ respectively, then the value of the edge $e$ is restricted to be $0$; if the two endpoints of an edge $e$ have the same value, then the value of the edge $e$ is either $0$ or the same as its endpoints. We use $\bar{\G}={V}(G)\cup {E}(G)$ to denote the set of vertices and edges of $\G$, and it can be viewed as a graph induced by putting a vertex on the midpoint of each edge. For any external field $h$, let $\bar{\mu}_{\G,h}$ be the unique probability measure on the set of extended configurations $\bar{\sigma}$ with
\begin{equation}\label{eq:P_Lambda_tau_def}
  \bar{\mu}_{G,h}(\bar{\sigma}) := \frac{1}{\bar{\cZ}_{G,h}} \prod_{ \substack{e \in {E}(G) \\ v \in e}}  W(\bar{\sigma}_v,\bar{\sigma}_{e}) \cdot  e^{- U(\bar\sigma)}
\end{equation}
where
$U(\bar\sigma) = \frac{1}{T}\sum_{v \in {V}(G)}  h_v \sigma_v$ and $W$ is defined as follows.
For each $a \in\{-1,1\}$ and $ b\in\{-1,0,1\}$,
\begin{equation}\label{eq:Wdef}
    W(a,b):= \lambda \Big( \delta _{a,b}  \ + \ t \delta_{b,0}  \Big)\ = \  \lambda  \cdot \begin{cases}
    1& b=a,\\
    t& b =0,\\
    0&b =-a,
  \end{cases}
\end{equation}
and $(t, \lambda)$ is related to the temperature $T$ by
\begin{equation*} \label{t_beta}
  t=\Big(\exp(\frac{2}{T})-1\Big)^{-\frac{1}{2}} \,,  \qquad \lambda =  \Big(2 \sinh(\frac{1}{T})\Big)^{\frac{1}{2}} \,.
\end{equation*}
Note that notation-wise it may make more sense to write $\bar \mu_{\bar G, h}$; we have chosen to drop the bar on $G$ for simplicity since
{the bar on $\mu$ suffices} to indicate that this refers to an extended Ising measure.  For $A\subset\bar G$, we use the notation $\bar{\mu}_{G,h}^{A,+}$ ($\bar{\mu}_{G,h}^{A,-}$) to denote the extended Ising measure with the plus (minus) boundary condition on $A$. More generally, for any boundary condition $\xi\in \{-1,1\}^{A \cap V(G)}\times \{-1,0,1\}^{A \cap E(G)}$, we denote by $\bar \mu^{A, \xi}_{G,h}$ the extended Ising measure with the boundary condition $\xi$ on $A$. For notational clearness, we might omit $A$ from the superscript, and in this case the boundary condition is posed on $\intB \bar{G}$.
When considering the extended Ising model, we define the boundary of $\Gamma\subset\bar G$ by $$\intB\Gamma=\{e={\{x,y\}}\in \Gamma:x\notin \Gamma\text{ or }y\notin \Gamma\}\cup\{x\in \Gamma:\exists\, {\{x,y\}}\in \bar G\setminus\Gamma\}\,.$$

Basic properties such as the domain Markov property (DMP) have been proved in \cite{AHP20} and it is straightforward to verify that the extended Ising model satisfies the FKG inequality. 

The extended Ising model is indeed an extension of the Ising model, that is, its restriction on vertex spins follows the law of the Ising model (see \cite[Section 2]{AHP20}).

In addition, the extended Ising model is also closely related to the FK-Ising model as incorporated in the next lemma.
\begin{lem}\label{lem:extended to FK}
   For any finite graph $\G$, any external field $h$ and any Ising boundary condition $\xi$, at any positive temperature $T$, the restriction of the extended Ising measure $\bar\mu^{\xi}_{{T},G,h}$ to ${E}(\G)$ coincides with the FK-Ising measure $\phi^{\xi}_{p,G,h}$ with parameter $p=1-\exp(-\frac{2}{T})$ in the sense that there exists a coupling such that an edge is closed in the FK-Ising measure if and only if its spin value is $0$ in the extended Ising measure.
\end{lem}

\begin{proof}
The proof is similar to the Edward-Sokal coupling, and we give a self-contained proof here merely for completeness. 

For an FK-configuration $\omega$, recall
that $\mathfrak C$ is the collection of all the open clusters in $\omega$ excluding $\mathcal C_\pm$, and that $h_A=\sum_{x\in A} h_x$ for $A\subset V(G)$. Let $\mathtt{c} (\omega)$ and $\mathtt{o} (\omega)$ denote the numbers of closed and open edges in $\omega$ respectively.
Then, recalling \eqref{eq-def-FK-external-field} we can write the corresponding FK-Ising measure $\phi^{\xi}_{p,G,h}$ as 
\begin{align*}
\phi^{\xi}_{p,G,h}(\omega)
&=\frac{1}{\cZ^{\xi,\phi}_{p,G,h}}p^{|{E}(G)|}(\frac{1-p}{p})^{\mathtt c(\omega)} \prod_{\mathcal C\in \mathfrak C} 2 \cosh(\frac{h_\mathcal C}{T})\times \exp( \frac{h_{\cC_+}-h_{\cC_-}}{T})\times \1_{\{{\mcc P\centernot \longleftrightarrow\mcc M}\}}\, ,
\end{align*}
where $\cZ^{\xi,\phi}_{p,G,h}$ is the partition function of $\phi^{\xi}_{p,G,h}$, {and $\mcc P, \mcc M$ are the positive and negative parts of the boundary condition $\xi,$ respectively.}

We say an extended Ising configuration $\bar{\sigma}$ is compatible with an FK-configuration $\omega$ if  $|\bar{\sigma}_e|=\omega_e$ for each edge $e$, and we use ${\rm Com}(\omega)$ to denote the set of all the extended Ising configurations compatible with $\omega$.
Then we obtain that for {$\omega$ such that } ${\rm Com}(\omega)\neq\emptyset$,
\begin{equation*}
\begin{aligned}
   \sum_{\bar{\sigma}\in {\rm Com}(\omega)}\bar{\mu}_{T,\G,h}^\xi(\bar{\sigma}) &= \frac{1}{\bar{\cZ}^{\xi}_{T,G,h}} \sum_{\bar{\sigma}\in {\rm Com}(\omega)}\lambda^{2\mathtt o(\omega)}\cdot(\lambda\cdot t)^{2\mathtt c(\omega)}\cdot \exp{\left\{\Big(\sum_{\cC\in \mathfrak {C}}\bar\sigma_{\cC}\cdot h_{\cC}+h_{\cC^+}-h_{\cC^-}\Big)/T\right\}} \\
   &= \frac{1}{\bar{\cZ}^{\xi}_{T,G,h}}\lambda^{2|{E}(G)|}\cdot t^{2\mathtt c(\omega)}\cdot\prod_{\mathcal C\in \mathfrak C} 2 \cosh(\frac{h_\mathcal C}{T})\times \exp( \frac{h_{\cC_+}-h_{\cC_-}}{T}),
\end{aligned}
\end{equation*}
where $\bar \sigma_\mathcal C$ is the (same) spin on the cluster $\mathcal C$ and $\bar{\cZ}^{\xi}_{T,G,h}$ is the partition function of the extended Ising model. Note that for any FK-configuration $\omega$, ${\rm Com}(\omega)\neq\emptyset$ if and only if ${\mcc P\centernot \longleftrightarrow\mcc M}$.  The desired result follows since $t^2 =\frac{1}{\exp(\frac{2}{T})-1} =\frac{\exp(-\frac{2}{T})}{1-\exp(-\frac{2}{T})} = \frac{1-p}{p}$.
\end{proof}
\begin{rmk}\label{rmk:connecting events for extended Ising}
    In the FK-Ising model, recall that we write $u\longleftrightarrow v$ (in $\omega$) if and only if there is a path $u=x_1,x_2,\cdots,x_n=v$ connecting $u$ and $v$ such that $\{x_i,x_{i+1}\}$ is an open edge in $\omega$. In the extended Ising model, we write $u\longleftrightarrow v$ (in $\bar\sigma$) if and only if there is a path $u=x_1,x_2,\cdots,x_n=v$ connecting $u$ and $v$ such that all the spin values along the path (including the edge spins) are either all plus or all minus. By Lemma \ref{lem:extended to FK}, this coincidence of notations is well-justified. 
\end{rmk}

\subsection{Disagreement percolation}\label{sect:disagreement percolation}

For a finite graph $G$, let $\bar{\sigma}^+$ and $\bar{\sigma}^-$ be two extended Ising configurations on $\bar G$. The \textbf{pre-disagreement set} of $(\bar \sigma^+, \bar \sigma^-)$ is defined as
\begin{equation*}
\mathcal{D}:=\mathcal{D}(\bar{\sigma}^+,\bar{\sigma}^-) = \{u \in \bar{\G} : \bar{\sigma}_u^+>\bar{\sigma}_u^- \}.
\end{equation*}
Similarly, the \textbf{anti-disagreement set} of $(\bar \sigma^+, \bar \sigma^-)$ is defined as $\{u \in \bar{\G} : \bar{\sigma}_u^+<\bar{\sigma}_u^- \}$. For $u\in \bar G$, we call it a pre-disagreement (or anti-disagreement) if it is in the pre-disagreement set (or anti-disagreement set). It can be seen from \eqref{hard constraint} (in the definition of $\bar \sigma$) that pre-disagreements can not be adjacent to anti-disagreements in the following sense: if $u\in V(G)$ is a pre-disagreement, then $e=\{u,v\}\in E(G)$ can not be an anti-disagreement and vise versa.

We say that $u,v\in\bar G$ are connected in $\mcc D$, denoted by $u \stackrel{\mathcal{D}}{\longleftrightarrow} v$, if there is a path connecting $u$ and $v$ such that all vertices and edges along the path are contained in  $\mathcal D$. Furthermore, for $A,B\subset \bar G$, we say $A,B$ are connected in $\mcc D$ if there exist $u\in A, v\in B$ such that $u \stackrel{\mathcal{D}}{\longleftrightarrow} v$.
A \textbf{pre-disagreement cluster} is a connected component of the pre-disagreement set. An \textbf{anti-disagreement cluster} is defined similarly.
For a set $S \subset \bar{\G}$, let $\mathcal{D}_{S}$ be the union of the connected components of $\mathcal{D}$ that intersect $S$.

We write $\bar \mu_{G,h}^{A,\xi^+/\xi^-}=\bar \mu_{G,h}^{A,\xi^+}\otimes \bar \mu_{G,h}^{A,\xi^-}$ to be the product of two independent extended Ising measures. By our notation convention of using  plus and minus as superscripts for the boundary condition of the first copy and of the second copy in the product respectively, we indicate  that \emph{usually} the boundary condition for the first copy in the product measure is larger, but we emphasize that this does not have to be the case. Sometimes we will drop $A$ from the superscript, and in this case the boundary condition is posed on $\intB \bar{G}$.
When $G$ is an annulus, say $\Lambda_{M,N}$, we use the notation $\bar \mu_{\Lambda_{M,N},h}^{\xi_1,\xi_2}$
for simplicity to denote the extended Ising measure with boundary conditions $\xi_1,\xi_2$ on $\intB \Lambda_{M+1}$, $\intB\Lambda_{N}$ respectively, and we use the notation $\bar \mu_{\Lambda_{M,N},h}^{\xi^+_1/\xi^-_1,\xi^+_2/\xi^-_2}$ for the product measure where the boundary condition for (e.g.) the first copy is $\xi_1^+, \xi_2^+$ on $\intB \Lambda_{M+1}$, $\intB\Lambda_{N}$ respectively.
We denote $\langle\cdot\rangle^{A,\xi}_{G,h}$ and $\langle\cdot\rangle^{A,\xi^+/\xi^-}_{G,h}$ the expectation operators with respect to the extended Ising measure (or simply the Ising measure) and the product of extended Ising measures respectively. Here we slightly abuse the notation $\langle\cdot\rangle^{A,\xi}_{G,h}$ since the Ising measure is just the restriction of the extended Ising measure to ${V}(G)$.  

\begin{prop}\label{prop:DisagreementRep}
For a finite graph $\G$ and an arbitrary external field $h$, the following hold.
\begin{enumerate}
    \item For any  $A\subset  \bar G $ and $u \in {V}(\G)$
\begin{eqnarray}\label{eq:DisagreementRepresentation}
\frac{1}{2} \cdot \left(
  \langle \sigma_u \rangle_{\G,h}^{A,+} -
  \langle \sigma_u \rangle_{\G,h}^{ A,-} \right)
  =
  \left\langle \1_{\{u \stackrel{\mathcal{D}}{\longleftrightarrow} A\}}\right\rangle^{A,+/-}_{\G,h} .
\end{eqnarray}
\item For any $A\subset \bar G, V_0\subset \vert(G)$ and any boundary condition $\xi \in \{1,-1\}^A$, we denote $\cB$ 
 the intersection of the following events: (i) All the points in $V_0$ are pre-disagreements or anti-disagreements. (ii) Each pre-disagreement cluster contains an even number of points in $V_0$. (iii) Each anti-disagreement cluster contains an even number of points in $V_0$.
 Then we have \begin{equation}\label{eq: disagreement-spin average expansion with agreement boundary}
    \Big\langle\prod_{u\in V_0}(\bar\sigma_{u}-\bar\sigma_{u}')\Big\rangle^{A,\xi/\xi}_{\G,h}=2^{{|V_0|}}\langle\1_{\cB}\rangle^{A,\xi/\xi}_{\G,h}
\end{equation} 
where 
$(\bar\sigma,\bar\sigma')$ is sampled from $\bar{\mu}^{A,\xi}_{\G,h}\otimes \bar{\mu}^{A,\xi}_{\G,h}$. \end{enumerate}
\end{prop}

The first item of Proposition~\ref{prop:DisagreementRep} was proved in \cite{AHP20} using a general principle of symmetry under the following class of \textbf{swap} operation (see \cite{AHP20} for detailed discussions) and the second item can be derived similarly.  For any pairs of sets $A, S\subset \bar G$,  let $R^A_S$ be the mapping
$R^A_S:  (\bar{\sigma},\bar{\sigma}') \mapsto   (\bar{\phi},\bar{\phi}')$ which acts as the identity in case $\mathcal{D}_S \cap A \neq \emptyset$, and otherwise swaps the configurations in $\mathcal{D}_S$, in the sense that
\begin{equation}\label{eq:ClusterSwap}
(\bar\phi_u,\bar\phi'_u) := \begin{cases}  (\bar{\sigma}'_u,\bar{\sigma}_u) & u \in \mathcal{D}_S \\
(\bar{\sigma}_u,\bar{\sigma}'_u) & u \not \in \mathcal{D}_S \\ \end{cases}  \quad \mbox{(provided $\mathcal{D}_S \cap A = \emptyset$)}.
\end{equation}

\begin{prop}\cite[Proposition 3.2]{AHP20}\label{prop:swap}
For any $A,S \subset \bar\G$, any boundary conditions $\xi,\xi'\in \{-1,{0,}+1\}^A$ and any external field $h$, the product measure $\bar{\mu}_{\G,h}^{A, \xi} \otimes\bar{\mu}_{\G,h}^{A, \xi'}$ (interpreted as $\bar{\mu}_{\G,h} \otimes\bar{\mu}_{\G,h}$ if $A=\emptyset$) is invariant under the action of $R_S^A$.
\end{prop}

Applying \eqref{eq:DisagreementRepresentation} to $\bar\G$ with $\G =\Lambda_N \subset \mathbb{Z}^2$ and $A = \intB \lamn$, we obtain the following geometric representation of the boundary influence:
 \begin{equation} \label{eq:boundary influence}
\mathbf m(T_c,N,\eps h) =   \left\langle \1_{\{o \stackrel{\mathcal{D}}{\longleftrightarrow} \intB \Lambda_N\}}\right\rangle^{\intB \Lambda_N,+/-}_{\Lambda_N,\eps h} .
\end{equation}
In correspondence to the connection with $\partial \Lambda_N$ above, we point out that it is crucial to consider $\mcc D_{\intB \Lambda_N}$, the connected component of $\intB \Lambda_N$ in $\mcc D$.  We will call $\mathcal D_{\partial \Lambda_N}$ the \textbf{disagreement set} and denote it as $\mathcal D_\partial$ for short. Thus, the boundary influence is just the probability of the event that the origin $o$ is a disagreement point (i.e., a point in the disagreement set).

\begin{rmk}
    Despite being quite obvious, we nevertheless emphasize that the disagreement set is a subset of the pre-disagreement set, but the reverse is not true. This is also the rationale behind the terminology of pre-disagreement and disagreement. 
\end{rmk}

\subsection{Basic properties of the model(s)}

Here we record some of the standard and well-known properties that will be used repeatedly (see e.g., \cite{FV17,D19} for details and proofs).

\begin{enumerate}[(i)]
    \item \textbf{FKG inequality.} This was introduced in \cite{FKG71}  and named after the three authors.
    
    If $\mathtt A,\mathtt B$ are both increasing (or decreasing) events, then 
$$P(\mathtt A\cap \mathtt B)\geq P(\mathtt A)\times P(\mathtt B).$$
    
    \item \textbf{Comparison of boundary conditions}.
    If two boundary conditions $\xi_1\leq \xi_2$, then for any increasing event $\mathtt A$ we have 
    $$P^{\xi_1}(\mathtt A)\leq P^{\xi_2}(\mathtt A).$$
    Here $\leq$ stands for a partial order in the set of boundary conditions: for spin configurations, we say $\xi_1\leq \xi_2$ if every plus spin in $\xi_1$ is also plus in $\xi_2;$ for bond configurations, we say $\xi_1\leq \xi_2$ if every open edge in $\xi_1$ is also open in $\xi_2$.
    \item \textbf{Domain Markov property}. For two domains $\Gamma_1\subset\Gamma_2$, given the configuration $\xi$ on $\Gamma_1$, the influence on the measure in $\Gamma_2\setminus\Gamma_1$ 
    behaves like a boundary condition:
    $$P_{\Gamma_2}(\cdot\mid \xi)=P_{\Gamma_2\setminus\Gamma_1}^{\xi|_{\partial\Gamma_1}}(\cdot),$$
where $\xi|_{\partial\Gamma_1}$ is the restriction of $\xi$ to $\partial\Gamma_1$ and it denotes for the boundary condition on $\partial\Gamma_1$.


\end{enumerate}
In (i), (ii), and (iii), $P$ may stand for the (extended) Ising measure with or without the external field and with or without the boundary condition and may also stand for the FK-Ising measure without the external field. However, the FK-Ising measure with the external field only satisfies (i) and (ii).
We will write FKG, CBC, and DMP in the following for convenience. 

\subsection{FKG for pre-disagreements}\label{FKG}
In what follows, we often work with the product of two (extended) Ising measures and we are particularly interested in the (pre-)disagreement set according to the corresponding pair of configurations. In light of this, we will refer to this product of {extended} Ising models as the \textbf{disagreement Ising model}.
In our analysis of the disagreement Ising model, we would like to have a type of FKG inequality which roughly states that the presence of pre-disagreements in some region will encourage the appearance of pre-disagreements in some other region. To this end, we need to define a specific partial order on the product of the configuration spaces. To be precise, let us focus on a fixed graph $\bar G$, with two boundary conditions $\xi_1, \xi_2$ on $A\subset\bar G$ and with an arbitrary external field $h=\{h_x:x\in { V}(G)\}$. Recall $\bar\mu_{G,h}^{\xi_1/\xi_2}=\bar\mu_{G,h}^{\xi_1}\otimes \bar\mu_{G,h}^{\xi_2}$ is the product of two extended Ising measures (with boundary conditions $\xi_1$ and $\xi_2$ respectively). In addition, we sample $(\bar \sigma^1,\bar\sigma^2)$ from this product measure.



Recalling \eqref{eq:P_Lambda_tau_def}, we have
\begin{equation}\label{eq:product-def}
\bar\mu_{G,h}^{\xi_1/\xi_2}(\bar \sigma^1,\bar \sigma^2)=\bar\mu_{G,h}^{\xi_1}(\bar\sigma^1)\times \bar\mu_{G,h}^{\xi_2}(\bar\sigma^2)=\frac{1}{\bar{\mcc Z}_{G,h}^{\xi_1}\cdot\bar{\mcc Z}_{G,h}^{\xi_2}}\prod_{i=1,2}\prod_{\substack{e \in {E}(G) \\ v \in e}}  W(\bar{\sigma}_v^i,\bar{\sigma}_{e}^i) \cdot  e^{- U(\bar\sigma^i)}.\end{equation}

For each vertex $v\in{V}(G)$ and each edge $e\in{E}(G)$, we see that $(\bar\sigma_v^1,\bar\sigma_v^2)$ and $(\bar\sigma_e^1,\bar\sigma_e^2)$  are elements of $\{-1,0,1\}^2. $ We write $\Theta=\{-1,0,1\}^2$ for convenience.
This motivates the following definition.

\begin{defi}\label{def:partial-order}
    We define a partial order $\succeq$ \emph{(}and  correspondingly $\preceq$\emph{)} on $\Theta$ such that
\begin{equation*}
    (a,b)\succeq (c,d)\iff (c,d)\preceq(a,b)\iff a\geq c\mbox{ and } b\leq d.
\end{equation*} 
This induces a partial order on $\Theta^{\bar G}$ naturally. That is,  $$(\bar\sigma^1,\bar\sigma^2)\succeq (\bar\nu^1,\bar\nu^2)\iff (\bar\nu^1,\bar\nu^2)\preceq(\bar\sigma^1,\bar\sigma^2)\iff(\bar\sigma_v^1,\bar\sigma_v^2)\succeq (\bar\nu_v^1,\bar\nu_v^2)\mbox{ and }  (\bar\sigma_e^1,\bar\sigma_e^2)\succeq (\bar\nu_e^1,\bar\nu_e^2),$$ 
for every  $v\in V(G)$ and $e\in{ E}(G)$.
Given this partial order on $\Theta^{\bar G }$, we can define the notion of \textbf{increasing event}: an event $\mathtt A\subset \Theta^{\bar G }$ is called increasing if and only if $(\bar\nu^1,\bar\nu^2)\in \mathtt A$ implies that $(\bar\sigma^1,\bar\sigma^2)\in \mathtt A$ for all $(\bar\sigma^1,\bar\sigma^2)\succeq (\bar\nu^1,\bar\nu^2)$.\end{defi}

For $(a,b),(c,d)\in \Theta$, we define
$$(a,b)\vee(c,d)=(\max\{a,c\}, \min\{b,d\}),\quad (a,b)\wedge(c,d)=(\min\{a,c\}, \max\{b,d\}).$$
Furthermore, for $(\bar\sigma^1,\bar\sigma^2), (\bar\nu^1,\bar\nu^2)\in \Theta^{\bar G }$, we define $(\bar\sigma^1,\bar\sigma^2)\vee(\bar\nu^1,\bar\nu^2)$ by taking the $\vee$-operation over each vertex and each edge, and we define $(\bar\sigma^1,\bar\sigma^2)\wedge(\bar\nu^1,\bar\nu^2)$ similarly.







\begin{lem}\label{lem:FKG}
{\emph{(FKG)}}
    For any increasing events $\mathtt A,\mathtt B\subset \Theta^{\bar G }$,
    we have 
    $$\bar\mu_{G,h}^{\xi_1/\xi_2}(\mathtt A\cap \mathtt B)\geq \bar\mu_{G,h}^{\xi_1/\xi_2}(\mathtt A)\times \bar\mu_{G,h}^{\xi_1/\xi_2}(\mathtt B).$$
\end{lem}

\begin{proof}
    By \cite{FKG71}, we only need to check that for all $(\bar \sigma^1, \bar \sigma^2), (\bar \nu^1, \bar \nu^2) \in\Theta^{\bar G}$
$$ \bar\mu_{G,h}^{\xi_1/\xi_2}\big((\bar\sigma^1,\bar\sigma^2)\vee(\bar\nu^1,\bar\nu^2)\big)\times \bar\mu_{G,h}^{\xi_1/\xi_2}\big((\bar\sigma^1,\bar\sigma^2)\wedge(\bar\nu^1,\bar\nu^2)\big)\geq \bar\mu_{G,h}^{\xi_1/\xi_2}(\bar\sigma^1,\bar\sigma^2)\times \bar\mu_{G,h}^{\xi_1/\xi_2}(\bar\nu^1,\bar\nu^2).$$

Recalling \eqref{eq:product-def} and \eqref{eq:P_Lambda_tau_def}, we see that the partition functions and the terms from the external field on both sides cancel with each other.
Thus, it suffices to show that for any edge $e$ and one of its endpoints $v$
\begin{equation} 
\begin{aligned}
    &W(\max\{\bar \sigma_v^1, \bar\nu_v^1\}, \max\{\bar \sigma_e^1, \bar \nu_e^1\})\times W(\min\{\bar \sigma_v^1, \bar\nu_v^1\}, \min\{\bar \sigma_e^1, \bar \nu_e^1\})\geq W(\bar \sigma_v^1, \bar \sigma_e^1)\times W(\bar\nu_v^1,  \bar \nu_e^1), \\
    &W(\min\{\bar \sigma_v^2, \bar\nu_v^2\}, \min\{\bar \sigma_e^2, \bar \nu_e^2\})\times W(\max\{\bar \sigma_v^2, \bar\nu_v^2\}, \max\{\bar \sigma_e^2, \bar \nu_e^2\})\geq W(\bar \sigma_v^2, \bar \sigma_e^2)\times W(\bar\nu_v^2,  \bar \nu_e^2).
\end{aligned}\label{eq:FKG-on-one-edge}
\end{equation}
Of course, we only need to show the case  for $i=1$ by symmetry, and we omit the superscript $1$ in what follows.

If $\bar \sigma_v=\bar\nu_v$, then it holds obviously. Otherwise, without loss of generality, we can assume $\bar\sigma_v=1$ and $\bar\nu_v=-1$. So it is equivalent to check
\begin{equation}\label{eq: ingredient for FKG}
    W(1, \max\{\bar \sigma_e, \bar \nu_e\})\times W(-1, \min\{\bar \sigma_e, \bar \nu_e\})\geq W(1, \bar \sigma_e)\times W(-1,  \bar \nu_e).
\end{equation}
In order that the right-hand side is not $0$, recalling \eqref{eq:Wdef}, we must have that \begin{equation*}
(\bar\sigma_e, \bar\nu_e)\in\{(1,-1),(1,0),(0,-1),(0,0)\}.\end{equation*}
Then \eqref{eq: ingredient for FKG} follows trivially from the fact that
$\bar\sigma_e\geq \bar\nu_e$ in all the 4 choices above.
\end{proof}

At the end of this subsection, we point out that the existence of a pre-disagreement crossing is an increasing event. This is because for each vertex being a pre-disagreement its spin value has to be $(1, -1)$, and for each edge being a pre-disagreement its spin value has to be in $\{(1, -1), (1,0), (0, -1)\}$, and both requirements are increasing with respect to our partial order.

\subsection{CBC for pre-disagreements}

In this subsection we
will prove the CBC property for the disagreement Ising model. 

First, we will introduce some notations. Notations from Section~\ref{FKG} will be used without further explanation. Fix a finite graph $G=(V,E)$. The boundary $A\subset V$ is a subset of the vertex set $V$. The external field $h\in \mbb R^{V}$ is also fixed.  We will consider two disagreement Ising measures 
$\bar{\mu}_{G,h}^{\xi^+/\xi^-}$ and $\bar{\mu}_{G,h}^{\zeta^+/\zeta^-}$ with different boundary conditions $(\xi^+,\xi^-), (\zeta^+,\zeta^-)\in\Theta^A$.

Recalling the partial order in Definition~\ref{def:partial-order}, we have $(\xi^+,\xi^-)\preceq (\zeta^+,\zeta^-)$ if $\xi_v^+\leq \zeta_v^+$ and $\xi_v^-\geq \zeta_v^-$ for all $v\in A$.
Then we have the following comparison of boundary conditions(CBC) lemma.

\begin{lem}\label{lem:CBC}{\emph{(CBC)}}
For any increasing event $\mathtt A\subset \Theta^{\bar G }$ and any boundary conditions $(\xi^+,\xi^-)\preceq (\zeta^+,\zeta^-)$, we have 
    $$\bar\mu_{\G, h}^{\xi^+/\xi^-}(\mathtt A)\leq\bar\mu_{\G, h}^{\zeta^+/\zeta^-}(\mathtt A). $$
\end{lem}

\begin{rmk}
    Usually, when applying Lemma~\ref{lem:CBC}, we choose $\zeta^+/\zeta^-$ to be the plus/minus boundary condition.
\end{rmk}

\begin{proof}
[Proof of Lemma~\ref{lem:CBC}.]
It is well-known that (see, e.g., \cite{D19})  
there exists a probability measure $P(\cdot, \cdot)$ on $\{-1,1\}^{V}\otimes \{-1,1\}^{V}$, such that 
    $P$ has marginal distributions $\mu^{\xi^+}_{G,h}(\cdot)$ and $\mu^{\zeta^+}_{G,h}(\cdot)$, and $P(\{(\sigma,\sigma'):\sigma\leq\sigma'\})=1$ (here for spin configurations $\sigma,\sigma'\in\{-1,1\}^V$, we say $\sigma\geq\sigma'$ if $\sigma_v\geq\sigma'_v$ for all $v\in V$).
Analogously, there exists a probability measure $Q(\cdot,\cdot)$ with marginal distributions $\mu^{\xi^-}_{G,h}(\cdot)$ and $\mu^{\zeta^-}_{G,h}(\cdot)$ such that 
$Q(\{(\sigma,\sigma'):\sigma\geq\sigma'\})=1$.

It is then clear that the
the measure $P\otimes Q$ has marginal distributions $\mu_{G,h}^{\xi^+/\xi^-}(\cdot)$ and $\mu_{G,h}^{\zeta^+/\zeta^-}(\cdot)$, and we next extend $P\otimes Q$ by incorporating edge spins.
By \eqref{eq:P_Lambda_tau_def} and \eqref{eq:Wdef}, we obtain that conditioned on vertex spins, the edge spins are independent. Furthermore, one can easily check that for an edge $e=\{x,y\}$ and two pairs of spins $(\sigma_x,\sigma_y)$ and $(\sigma_x^{{\prime}}, \sigma_y^{{\prime}})$, there is a monotone coupling between
 $\sigma_e$ and $\sigma_e^{{\prime}}$  as long as $\sigma_x\geq \sigma_x^{{\prime}}$ and $ \sigma_y\geq\sigma_y^{{\prime}}$.
As a consequence, $P\otimes Q$ can be extended to ${\bar P\otimes \bar Q}(\cdot,\cdot)$, a measure on $\Theta^{\bar G}\otimes\Theta^{\bar G}$, which has marginal distributions 
$\bar\mu_{G,h}^{\xi^+/\xi^-}(\cdot)$ and $\bar\mu_{G,h}^{\zeta^+/\zeta^-}(\cdot)$ and satisfies
$${\bar P\otimes \bar Q}\Big(\big\{(\sigma^+, \sigma^-{ , } \nu^+,\nu^-):(\sigma^+,\sigma^-)\preceq(\nu^+,\nu^-)\big\}\Big)=1.$$
This completes the proof since $\mathtt A$ is increasing.
\end{proof}
Combining Lemma~\ref{lem:CBC} and the DMP property, we obtain the following useful criterion which we also refer to as the CBC property.
\begin{cor}\label{cor: CBC}
    Let $\Gamma'\subset\Gamma\subset\mathbb Z^2$, for any increasing event $\mathtt A\subset \Theta^{\bar\Gamma'}$, any boundary conditions $\xi^+,\xi^-$ on $\partial\Gamma$ and any external field $h$, we have $$\bar\mu_{\Gamma', h}^{+/-}(\mathtt A)\geq\bar\mu_{\Gamma, h}^{\xi^+/\xi^-}(\mathtt A)\ge \bar\mu_{\Gamma', h}^{-/+}(\mathtt A). $$
\end{cor}
In order to meet further needs, we prove the following variant of Corollary \ref{cor: CBC} and denote the property as CBC'.
\begin{lem}\label{lem:CBC prime}
    Fix two integers $M<N$ and $\Lambda_{2M}(u)\subset\lamn$. Let $\cA$ denote the event that there exists a disagreement circuit in $\Lambda_{M,2M}$. Then for any $v\in\Lambda_{M}(u)$ and any external field $h$, we have
    \begin{equation*}
        \bar\mu_{\lamn,h}^{+/-}(o\in\cD_{\partial}\mid\cA)\ge \bar\mu_{\Lambda_{2M},h}^{+/-}(o\in\cD_{\partial\Lambda_{2M}(u)}).
    \end{equation*}
\end{lem}
\begin{proof}
For some connected region $\Gamma$ with $\Lambda_M(u)\subset\Gamma\subset\Lambda_{2M}(u)$, we say $\partial \bar\Gamma$ is a \textbf{disagreement boundary} if $\partial \bar\Gamma \subset \mathcal D_{\partial}$. In addition, we say a disagreement boundary $\partial \bar\Gamma$ is the \textbf{outmost disagreement boundary} in $\bar\Lambda_{M,2M}(u)$ if $\bar \Gamma$ is simply connected, and if for any disagreement boundary $\intB{\bar\Gamma'}$ and any path in $\bar\Lambda_{2M}(u)$ from $x\in\intB{\bar\Gamma'}$ to $y\in \intB\bar\Lambda_{2M}(u)$, this path must intersect with $\intB\bar\Gamma$. Let $\Psi_h(\bar\Gamma)$ denote the probability that $\intB \bar\Gamma$ is the outmost disagreement boundary in $\bar\Lambda_{M,2M}(u)$ with respect to $\bar\mu^{+/-}_{\lamn, h}$. It is clear that the event that $\partial \bar\Gamma$ is the outmost disagreement boundary is measurable with respect to spin values on $(\bar \Lambda_{2M}(u) \setminus \bar \Gamma) \cup \partial \bar \Gamma$. Note that if $\cA$ happens, there exists a unique outmost disagreement boundary. Thus we obtain by DMP and CBC that \begin{align*}
        \bar\mu_{\lamn,h}^{+/-}(o\in\cD_{\partial}\mid\cA)&\stackrel{\mbox{DMP}}= \frac{\sum_{\bar\Gamma}\Psi_h(\bar\Gamma)\cdot\bar\mu^{+/-}_{\Gamma, h}(o\in\cD_{\partial} )}{\sum_{\bar\Gamma}\Psi_h(\bar\Gamma)}\nonumber\\&\stackrel{\mbox{CBC}}\ge \frac{\sum_{\bar\Gamma}\Psi_h(\bar\Gamma)\cdot\bar\mu^{+/-}_{\Lambda_{2M}(u), h}(o\in\cD_{\partial\Lambda_{2M}(u)})}{\sum_{\bar\Gamma}\Psi_h(\bar\Gamma)}=  \bar\mu^{+/-}_{\Lambda_{2M}(u), h}(o\in\cD_{\partial\Lambda_{2M}(u)}).
    \end{align*}
Therefore, we finish the proof of Lemma~\ref{lem:CBC prime}.
\end{proof}

\section{Outline of the proof}\label{sec:outline}
In this section, we will outline the proof of Theorems \ref{thm-critical-temperature} and \ref{thm-critical-temperature-small-perturbation}, while postponing a number of lemmas for smooth flow of presentation. The proof of Theorem \ref{thm-supercritical-temperature} is of different flavor and will be contained in Section \ref{sect:correlation length}. From now on, we fix the temperature $T=T_c(2)$ until Section \ref{sect:correlation length} where we discuss the low temperature regime.

In Section \ref{sec: Proof of Theorem 1.3}, we outline the proof of Theorem \ref{thm-critical-temperature-small-perturbation}, which then also explains how the critical exponent of $7/8$ arises. The proof of Theorem \ref{thm-critical-temperature} is substantially more difficult, and it will be outlined in Section \ref{sec: Proof of Theorem 1.1}. In addition, we will outline the proof of important ingredients for proving Theorem \ref{thm-critical-temperature} in later subsections: see Section~\ref{sec: outline for RSW} for the RSW estimate and see Section \ref{sec:Proof of Proposition 3.10} for an upper bound on the disagreement crossing probability. 
\subsection{Proof of Theorem~\ref{thm-critical-temperature-small-perturbation}}\label{sec: Proof of Theorem 1.3}

In \cite{BS22}, it was proved that when $  \eps=N^{-7/8}$, the partition function {of} the RFIM converges to some non-trivial limit with a proper scaling. Using similar ideas with more delicate estimates, we will prove that when $\eps\ll N^{-7/8}$, the spin averages at the origin for the Ising models without disorder and with the external field $\epsilon h$ are typically close to each other. To this end, we note a simple fact which will be used in the analysis without further notice: by independence, the partition function of the disagreement Ising model, denoted as
 $\mathcal {\bar Z}_{G, h}^{\xi+/\xi-}$, is just the product of the respective partition functions, i.e., $$\bar\cZ^{\xi^+/\xi^-}_{G,h}=\bar \cZ^{\xi^+}_{G,h}\times \bar\cZ^{\xi^-}_{G,h}.$$ Here, we recall our convention on notations and superscripts as in Section~\ref{sect:extended-Ising} after \eqref{eq:Wdef}.   

{
We first expand the ratio between the partition functions with and without disorder as follows:}
\begin{align}
       \frac{\bar\cZ^{+/-}_{\lamn,\eps h}}{\bar\cZ^{+/-}_{\lamn,0}}
       &= \langle ~\exp(\sum_{v\in \lamn}\veps h_v\bar \sigma_v^++\sum_{v\in \lamn}\veps h_v\bar \sigma_v^-) ~ \rangle_{\lamn,0}^{+/-}\nonumber\\ &= \langle ~\prod_{v\in \lamn}\Big([\cosh(\veps h_v)+\bar \sigma_v^+\sinh(\veps h_v)]\cdot[\cosh(\veps h_v)+\bar \sigma_v^-\sinh(\veps h_v)]\Big)~\rangle_{\lamn,0}^{+/-}\nonumber\\ &= \prod_{v\in \lamn} [\cosh(\veps h_v)]^2 \langle ~\prod_{v\in \lamn}\Big([1+\bar \sigma_v^+\tanh(\veps h_v)]\cdot [1+\bar \sigma_v^-\tanh(\veps h_v)]\Big)~\rangle_{\lamn,0}^{+/-}\nonumber
        \\ &= \prod_{v\in \lamn} [\cosh(\veps h_v)]^2 \langle ~\prod_{v\in \lamn}\Big([1+ \sigma_v^+\tanh(\veps h_v)]\cdot [1+ \sigma_v^-\tanh(\veps h_v)]\Big)~\rangle_{\lamn,0}^{+/-}.\label{eq: partition function expansion}
\end{align}
We remark that in the last equality, we dropped the bars on $\bar\sigma$ because the restriction of the extended Ising measure on vertices is the same as the Ising measure. 
In addition, in the above we have written $\varepsilon=\frac{1}{{T_c}}\eps$.
Although it may appear to be disadvantageous that the notation $\epsilon$ and $\varepsilon$ look somewhat similar, we made such choice on {the} notation to emphasize our point that in our proof it is not at all material to change $\epsilon$ up to a constant factor. Finally, in the above products (as well as similar products below), it may make more sense to multiply over $v\in \Lambda_{N-1}$, and we chose to write $v\in \Lambda_N$ for notation clarity since it is also correct.

Similarly, we can get an expansion of the boundary influence as follows: \newpage
\begin{align}
        &\langle\sigma_o\rangle_{\lamn,\eps h}^{+}-\langle\sigma_o\rangle_{\lamn,\eps h}^{-}\nonumber\\=~&\frac{\langle ~\sigma_o \prod_{v\in\lamn}[1+\sigma_v\tanh(\veps h_v)]~\rangle_{\lamn,0}^{+}}{\langle ~\prod_{v\in\lamn}[1+\sigma_v\tanh(\veps h_v)]~\rangle_{\lamn,0}^{+}}-\frac{\langle ~\sigma_o \prod_{v\in\lamn}[1+\sigma_v\tanh(\veps h_v)]~\rangle_{\lamn,0}^{-}}{\langle ~\prod_{v\in\lamn}[1+\sigma_v\tanh(\veps h_v)]~\rangle_{\lamn,0}^{-}}\nonumber\\=~&\frac{\langle ~\sigma_o \prod_{v\in\lamn}[1+\sigma_v\tanh(\veps h_v)]~\rangle_{\lamn,0}^{+}\cdot\langle ~\prod_{v\in\lamn}[1+\sigma_v\tanh(\veps h_v)]~\rangle_{\lamn,0}^{-}}{\langle ~\prod_{v\in\lamn}[1+\sigma_v\tanh(\veps h_v)]~\rangle_{\lamn,0}^{+}\cdot \langle ~\prod_{v\in\lamn}[1+\sigma_v\tanh(\veps h_v)]~\rangle_{\lamn,0}^{-}}\nonumber\\&-\frac{\langle ~\sigma_o \prod_{v\in\lamn}[1+\sigma_v\tanh(\veps h_v)]~\rangle_{\lamn,0}^{-}\cdot\langle ~\prod_{v\in\lamn}[1+\sigma_v\tanh(\veps h_v)]~\rangle_{\lamn,0}^{+}}{\langle ~\prod_{v\in\lamn}[1+\sigma_v\tanh(\veps h_v)]~\rangle_{\lamn,0}^{+}\cdot \langle ~\prod_{v\in\lamn}[1+\sigma_v\tanh(\veps h_v)]~\rangle_{\lamn,0}^{-}}.\label{eq: boundary influence expansion tmp}
\end{align}
 We expand all the brackets in the numerator of \eqref{eq: boundary influence expansion tmp} and obtain that \begin{align}\label{eq: expansion for boundary influence}
    &\langle\sigma_o\rangle_{\lamn,\eps h}^{+}-\langle\sigma_o\rangle_{\lamn,\eps h}^{-}\nonumber\\=~&\frac{\sum\limits_{I,J\subset\lamn}\Big[\langle ~\sigma_o\sigma^I~\rangle_{\lamn,0}^{+}\langle ~\sigma^J~\rangle_{\lamn,0}^{-}-\langle ~\sigma^I~\rangle_{\lamn,0}^{+}\langle ~\sigma_o\sigma^J~\rangle_{\lamn,0}^{-}\Big]\times\prod\limits_{x\in I}\tanh(\veps h_x)\prod\limits_{y\in J}\tanh(\veps h_y)}{\langle ~\prod\limits_{v\in\lamn}[1+\sigma_v\tanh(\veps h_v)]~\rangle_{\lamn,0}^{+}\cdot \langle ~\prod\limits_{v\in\lamn}[1+\sigma_v\tanh(\veps h_v)]~\rangle_{\lamn,0}^{-}},
\end{align}
where we used the notation convention that $\sigma^I=\prod_{x\in I} \sigma_v$.
Before controlling the denominator of \eqref{eq: expansion for boundary influence}, we need Lemma \ref{lem: upper-bound for the sum of squares of k point function} below to control the $k$-point correlation function. In the proof of Lemma~\ref{lem: small perturbation for partition function} below, we will see that the $k$-point correlation function actually gives an upper bound on how the disorder influences the partition function.
\begin{lem}\label{lem: upper-bound for the sum of squares of k point function}
     Fix a constant $\iota>0$. Then there exists a constant $\nc\label{10}=\oc{10}(\iota)>0$ such that the following holds. For any rectangle $\cR$ with size $\iota M\times M$, there exists a function $F$ such that \begin{equation}\label{eq: k point function bound}
        \lvert \langle\prod_{x\in I}\sigma_{x}\rangle_{\cR,0}^{+}\rvert\le F(I), ~~~~\forall ~I \subset\cR .
    \end{equation}
    Moreover, we have for any integer $k\ge0$ and $y\in\cR$,\begin{align}
        \sum_{\substack{I\subset \cR,\\|I|=k} } F(I)^2&\le \oc{10}^{k}M^{\frac{7k}{4}}\cdot\frac{(k!)^{\frac{1}{4}}}{k!},~~~\text{and}\label{eq: upper-bound for the sum of squares of k point function}\\
        \sum_{\substack{I\subset \cR,\\|I|=k} } F(I\Delta \{y\})^2
        &\le \oc{10}^{k}M^{\frac{7k}{4}}\Big(\langle\sigma_{y}\rangle_{\cR,0}^{+}\Big)^2\cdot\frac{[(k+1)!]^{\frac{1}{4}}}{k!}.\label{eq: upper-bound for the sum of squares of k point function2}
    \end{align} Here we use the notation $A\Delta B$ to denote the symmetric difference between two sets $A$ and $B$.
\end{lem}
\begin{proof}
    For any $I \subset\cR$, let \begin{equation*}
        F(I)=\prod_{x\in I}\frac{1}{\big[dist(x,\partial\cR\cup I\setminus\{x\})\big]^{\frac{1}{{8}}}}.
    \end{equation*} Then \eqref{eq: k point function bound} and \eqref{eq: upper-bound for the sum of squares of k point function} follow from \cite[Lemmas 8.1 and 8.3]{FSZ16} and we emphasize that the denominator $k!$ follows from the difference between summing over $I\subset \lamn$ {with $|I|=k$} and summing over $x_1,\cdots,x_k\in \lamn$. To show \eqref{eq: upper-bound for the sum of squares of k point function2}, we write the summation as \begin{equation}\label{eq: expansion for symmetric difference}
        \sum_{\substack{I\subset \cR,\\|I|=k} } F(I\Delta \{y\})^2=\sum_{\substack{I\subset \cR{\setminus \{y\}},\\|I|=k-1} } F(I)^2+\sum_{\substack{ I\subset \cR{\setminus \{y\}},\\|I|=k} } F(I\cup \{y\})^2.
    \end{equation} The first term in \eqref{eq: expansion for symmetric difference} can be upper-bounded by \eqref{eq: upper-bound for the sum of squares of k point function} and the second term can be bounded using induction and an adaptation of \cite[(8.17)]{FSZ16}. Thus we obtain that \begin{equation}
\sum_{\substack{I\subset \cR,\\|I|=k} } F(I\Delta \{y\})^2\le C_1^{k-1}M^{\frac{7(k-1)}{4}}\cdot\frac{[(k-1)!]^{\frac{1}{4}}}{(k-1)!}+C_1^{k}M^{\frac{7k}{4}}
F(\{y\})^2\cdot\frac{[(k+1)!]^{\frac{1}{4}}}{k!}.\label{eq:10}
    \end{equation} Combined with CBC and \eqref{eq: origin decay rate without disorder}, it yields that $\big(\langle\sigma_{y}\rangle_{\cR,0}^{+}\big)^2\ge C_2F(\{y\})^2\ge C_3 M^{-\frac{1}{4}}$ and thus the second term in \eqref{eq:10} can dominate the first term.
\end{proof}

\begin{defi}\label{def: good external field for partition function}
    Let $\cR$ be a rectangle with size $\iota M\times M$ for some $\iota>0$. Let $\cH^+_{*}\subset\R^{\cR}$ denote the collection of the external field $h$ such that \begin{equation}\label{eq: good external field for expectation1}
        |\langle ~ \prod_{v\in\cR}[1+ \sigma_v\tanh(\veps h_v)]~\rangle^+_{\cR,0 }|\le1+ \sqrt{\eps M^{\frac{7}{8}}}.
    \end{equation} 
\end{defi}
\begin{rmk}
    {Of course, the set $\mcc H_*^+$ relies on $\mcc R$ and $\eps,$ but we omit them from the notation for concision. The same conventions will be used several times without further notice.}
\end{rmk}
\begin{lem}\label{lem: small perturbation for partition function}
    Let $\cR$ be a rectangle with size $\iota M\times M$. There exist constants $\nc\label{11}=\oc{11}(\iota),\nc\label{12}=\oc{12}(\iota)>0$ such that $\P(\cH^+_{*})\ge 1-\oc{11}\exp(-\oc{11}^{-1}\sqrt{\eps^{-1}M^{-7/8}})$ for all disorder strength  {$\epsilon\leq\oc{12}M^{-7/8}$}.
\end{lem}
\begin{rmk}
    In Definition \ref{def: good external field for partition function}, we put $+$ in the superscript in the notation $\cH^+_{*}$ in order to emphasize that we consider the Ising measure with the plus boundary condition. We can also define an analog $\mathcal H_
*^-$ with respect to the minus boundary condition, and it is 
 not hard to see that Lemma~\ref{lem: small perturbation for partition function} also holds for $\cH^-_{*}$.
\end{rmk}

To control the numerator of \eqref{eq: expansion for boundary influence}, we introduce an analog of Definition \ref{def: good external field for partition function} and obtain Lemma~\ref{lem: small perturbation for expectation} as an analog to Lemma~\ref{lem: small perturbation for partition function}.

\begin{defi}\label{def: good external field for o}
    Let $\cH_{o}\subset\R^{\lamn}$ denote the collection of the external field $h$ such that \begin{equation*}\label{eq: good external field for expectation2}
    \begin{aligned}
        &\Big|\sum\limits_{I,J\subset\lamn}\Big([\langle ~\sigma_o\sigma^I~\rangle_{\lamn,0}^{+}\langle ~\sigma^J~\rangle_{\lamn,0}^{-}-\langle ~\sigma^I~\rangle_{\lamn,0}^{+}\langle ~\sigma_o\sigma^J~\rangle_{\lamn,0}^{-}]\times\prod\limits_{x\in I}\tanh(\veps h_x)\prod\limits_{y\in J}\tanh(\veps h_y)\Big)\\&-\Big(\langle ~\sigma_o~\rangle_{\lamn,0}^{+}-\langle ~\sigma_o~\rangle_{\lamn,0}^{-}\Big)\Big|\le \sqrt{\eps N^{\frac{7}{8}}}\langle\sigma_o\rangle_{\lamn,0}^{+}.
    \end{aligned}
    \end{equation*} 
\end{defi}

\begin{lem}\label{lem: small perturbation for expectation}
    There exist constants $\nc\label{expectation1},\nc\label{expectation2}>0$ such that $\P(\cH_{o})\ge 1-\oc{expectation1}\exp(-\oc{expectation1}^{-1}\sqrt{\eps^{-1} N^{-\frac{7}{8}}})$ for all {$\eps\leq\oc{expectation2}N^{-\frac{7}{8}}$}.
\end{lem}
With Lemmas ~\ref{lem: small perturbation for partition function} and \ref{lem: small perturbation for expectation} at hand, we are ready to prove Theorem~\ref{thm-critical-temperature-small-perturbation}.
\begin{proof}[Proof of Theorem~\ref{thm-critical-temperature-small-perturbation}]
The upper bound follows directly from \cite[Theorem 1.1]{DSS22}, so it suffices to prove the lower bound. 
Recall that {$\eps \le\oc{4}N^{-\frac{7}{8}}$}, where we choose $\oc{4}$ such that $0<\oc{4}<\min\{\oc{12}, \oc{expectation2}\}$, {and} $\oc{12}, \oc{expectation2}$ are defined in Lemmas~\ref{lem: small perturbation for partition function} ($\iota=1$) and \ref{lem: small perturbation for expectation}.
    For any external field $h\in \cH=\cH^+_*\cap\cH^-_*\cap \cH_o$, we compute by \eqref{eq: expansion for boundary influence}, Definitions \ref{def: good external field for partition function} and \ref{def: good external field for o} that \begin{equation}\label{eq: boundary influence for good external field}
        \langle\sigma_o\rangle_{\lamn,\eps h}^{+}-\langle\sigma_o\rangle_{\lamn,\eps h}^{-}\ge \frac{(2-\sqrt{\eps N^{\frac{7}{8}}}) \langle\sigma_o\rangle_{\lamn,0}^{+}}{\left(1+\sqrt{\eps N^{\frac{7}{8}}}\right)^2}\ge (2-\theta)m(T_c,N,0),
    \end{equation}where the last inequality follows by recalling {$\eps \le\oc{4}N^{-\frac{7}{8}}$} and taking $\oc{4} >0$ small enough.
    Combining Lemmas~\ref{lem: small perturbation for partition function} and \ref{lem: small perturbation for expectation}, we obtain that \begin{equation}\label{eq: probability for good external field}
        \P(\cH)\ge 1-C_{1}\exp({-}C_{1}^{-1}\sqrt{\eps^{-1} N^{-\frac{7}{8}}})\ge 1-\frac{\theta}{2}.
    \end{equation}
    Thus we complete the proof of \eqref{eq-small-perturbation}. In addition, averaging \eqref{eq: boundary influence for good external field} over $h\in \mathcal H$ and applying \eqref{eq: probability for good external field}, we obtain by CBC that $$\E[\langle\sigma_o\rangle_{\lamn,\eps h}^{+}-\langle\sigma_o\rangle_{\lamn,\eps h}^{-}]\ge \E[\langle\sigma_o\rangle_{\lamn,\eps h}^{+}-\langle\sigma_o\rangle_{\lamn,\eps h}^{-}\mid \cH]\times \mbb P(\mcc H)\ge (2-2\theta)m(T_c,N,0).$$ This completes the proof of Theorem~\ref{thm-critical-temperature-small-perturbation}.
\end{proof}
\subsection{Proof of Theorem~\ref{thm-critical-temperature}}\label{sec: Proof of Theorem 1.1}
We first split Theorem \ref{thm-critical-temperature} into two theorems corresponding to the upper and lower bounds respectively.
\begin{thm}\label{thm:main thm-upper bound}
    Fix $d=2$.  There exist constants
$\nc\label{01}, \newconstant\label{02}>0$ such that for {$\eps\geq\oldconstant{01} N^{-\frac{7}{8}}$}
\begin{equation}\label{eq-main-theorem upper bound}
     m(T_c,N,\eps) \leq \oldconstant{02}^{-1}N^{-\frac{1}{8}} \exp\Big(-\oldconstant{02} \epsilon^{\frac{8}{7}} N\Big)\,. 
\end{equation}
\end{thm}

\begin{thm}\label{thm:main thm-lower bound}
Fix $d=2$. There exist constants
$\nc\label{03}, \newconstant\label{04}>0$ such that for {$\eps\geq\oldconstant{03}N^{-\frac{7}{8}}$}
\begin{equation}\label{eq-main-theorem-lower-bound 2}
    \P\left( \mathbf m(T_c,N,\eps h)\ge \oldconstant{04}^{-1}N^{-\frac{1}{8}}  \exp(-\oldconstant{04}\epsilon^{\frac{8}{7}} N)\right)\ge 1-\oc{04}\exp\left(-\oc{04}^{-1}(\eps^{\frac{8}{7}}N)^{\frac{1}{10}}\right).
\end{equation}
\end{thm}
Assuming Theorems \ref{thm:main thm-upper bound} and \ref{thm:main thm-lower bound}, we now provide the proof of Theorem~\ref{thm-critical-temperature}.
\begin{proof}[Proof of Theorem~\ref{thm-critical-temperature}]
    {We first prove \eqref{eq-main-theorem-lower-bound} and then show \eqref{eq-main-theorem}. Let $0<\oc{-2}< \oc{02}$ and $\oc{-1}\ge \max\{\oc{01},\oc{03}\}$. Thus we have $\eps\ge\max\{\oc{01}N^{-\frac{7}{8}},\oc{03}N^{-\frac{7}{8}}\}$. } The lower bound in \eqref{eq-main-theorem-lower-bound} follows directly from Theorem~\ref{thm:main thm-lower bound}. The upper bound in \eqref{eq-main-theorem-lower-bound} follows from Theorem \ref{thm:main thm-upper bound} and a straightforward application of Markov's inequality as follows: \begin{align*}
        \P\Big(\mathbf m(T_c,N,\eps h)> \oldconstant{-2}^{-1}N^{-\frac{1}{8}}  \exp(-\oldconstant{-2}\epsilon^{\frac{8}{7}} N)\Big)&\le \oldconstant{-2}N^{\frac{1}{8}}  \exp(\oldconstant{-2}\epsilon^{\frac{8}{7}} N)\times \oldconstant{02}^{-1}N^{-\frac{1}{8}} \exp\Big(-\oldconstant{02} \epsilon^{\frac{8}{7}} N\Big)\nonumber\\&\le C_1\exp(-C_1^{-1}\epsilon^{\frac{8}{7}} N)
    \end{align*}where the last inequality comes since {$\oc{-2}< \oc{02}$}. Thus we finish the proof of \eqref{eq-main-theorem-lower-bound}. 
    
    In order to show \eqref{eq-main-theorem}, {fix any} constant $\oc{1}>0$.
    If {$\eps\geq\oc{01}N^{-\frac{7}{8}}$}, then by Theorem~\ref{thm:main thm-upper bound}, we obtain the upper bound of \eqref{eq-main-theorem}. If {$\oc{1}N^{-\frac{7}{8}}\leq\eps<\oc{01}N^{-\frac{7}{8}}$}, applying \cite[Theorem 1.1]{DSS22} and letting $\oc{2}{>0}$ be small enough, we obtain the upper bound of \eqref{eq-main-theorem}. As for the lower bound, if {$\eps\geq\oc{03}N^{-\frac{7}{8}}$}, then by Theorem~\ref{thm:main thm-lower bound}, we obtain the lower bound of \eqref{eq-main-theorem}. If $\oc{1}N^{-\frac{7}{8}}\leq\eps<\oc{03}N^{-\frac{7}{8}}$, let $N'=\lceil(\frac{\oc{03}} {\oc{1}})^{\frac{8}{7}}N\rceil$ be such that $\eps\geq\oc{03}(N')^{-\frac{7}{8}}$. Thus we obtain by CBC and Theorem~\ref{thm:main thm-lower bound} that 
    \begin{align*}
        m(T_c,N,\eps)\stackrel{\mbox{CBC}}\ge m(T_c,N',\eps)\stackrel{\text{Thm}~\ref{thm:main thm-lower bound}}\ge& (1-\oc{04}\exp\left(-\oc{04}^{-1}(\eps^{\frac{8}{7}}N')^{\frac{1}{10}}\right)\cdot \oldconstant{04}^{-1}(N')^{-\frac{1}{8}}\exp(-\oldconstant{04}\epsilon^{\frac{8}{7}} N')\\ \ge~~& C_2N^{-\frac{1}{8}}\exp(-C_2^{-1}\eps^{\frac{8}{7}}N)
    \end{align*} where the last inequality follows from the relation between $N$ and $N'$. Thus we finish the proof of the lower bound of \eqref{eq-main-theorem}.
\end{proof}
To show the stretched exponential decay in Theorem~\ref{thm:main thm-upper bound}, we start with the framework presented in \cite{DX21} and \cite{AHP20}, but substantial new ideas are required since our disorder is very weak. For instance, an important ingredient is the following RSW estimate for the disagreement Ising model with disorder. 
\begin{defi}\label{def: horizontal crossing}
    For any integers $a,b>0$, we use the notation $\cR(a,b)$ to denote the rectangle $[-a,a]\times [-b,b]$. We define $\hc(a,b)$ to be the event that there exists a horizontal pre-disagreement crossing through $\cR(a, b)$, i.e., $\{-a\}\times[-b,b]$ is connected to $\{a\}\times[-b,b]$ by $\mcc D\cap \bar{\mcc R}(a,b)$
    (sometimes we also use the notation $\hc(\cR)$ to denote the event that there exists a horizontal pre-disagreement crossing through a rectangle $\cR$). When $a>b$ we say it is a hard crossing, and when $a<b$ we say it is an easy crossing. 
\end{defi}
\begin{thm}\label{thm: RSW for dis with external field}
    For any $\iota>0$, there exist constants $\nc\label{7},\nc\label{8},\nc\label{9}>0$ depending only on $\iota$ such that for any integer $M>0$ the following holds. For any {$\eps\leq\oc{7}M^{-\frac{7}{8}}$}, with
     $\P$-probability at least $1-\oc{8}\exp(-\oc{8}^{-1}\sqrt{\eps^{-1} M^{-\frac{7}{8}}})$, we have \begin{equation}\label{eq: RSW for dis with external field}
        \bar\mu^{\xi^+/\xi^-}_{\cR(\iota M+M,2M),\eps h}(\hc(\iota M,M))>\oc{9}
    \end{equation}{for any boundary condition{s} $\xi^+,\xi^-$ on $\partial\cR(\iota M+M,2M)$.}
\end{thm}
\begin{rmk}
    By rotation symmetry, an analog of Theorem \ref{thm: RSW for dis with external field} holds for vertical crossings, and we sometimes apply Theorem \ref{thm: RSW for dis with external field} to vertical crossings without further notice.
\end{rmk}
\begin{rmk}\label{rmk: anti reduction}
    By CBC (Lemma~\ref{lem:CBC}), it suffices to prove \eqref{eq: RSW for dis with external field} with the anti-disagreement boundary condition, i.e., $\xi^+_v=-1,\xi^-_v=1$ for all $v\in \partial\cR(\iota M+M,2M)$. 
\end{rmk}
\subsubsection{Upper bound}

As in \cite{DX21} and \cite{AHP20}, the key ingredient in proving  Theorem \ref{thm:main thm-upper bound} is the following Proposition \ref{prop:fast power law}. We start with the {following} definition.
\begin{defi}\label{def:connect-event}
    We define $\con(M_1,M_2)$ ($M_1<M_2$) to be the event that  there exists a pre-disagreement crossing from $\intB \Lambda_{M_2}$ to $\intB \Lambda_{M_1+1}$ (with respect to the disagreement Ising model).
    We remark that the event $\con(M_1,M_2)$ is equivalent to $\intB\Lambda_{M_1+1} \cap \mathcal D_{\intB \Lambda_{M_2}} \neq \emptyset$.
\end{defi}

\begin{prop}\label{prop:fast power law}
  For any $\nc\label{fpl1}>0$, there exists a constant $\nc\label{fpl2}>0$ such that for any integer $N>0$ and any 
  disorder strength {$\eps \geq\oc{fpl2}N^{-\frac{7}{8}}$}, 
  \begin{equation*}\label{eq:fast power law}
    \E\bar\mu_{\Lambda_{N/2,N},\eps h}^{+/-,+/-}(\con(N/2,N))\le \oc{fpl1}.
  \end{equation*}
\end{prop}

\begin{proof}[Proof of Theorem~\ref{thm:main thm-upper bound}]
Provided with Proposition~\ref{prop:fast power law}, our proof is an adaption of arguments in \cite[Section 6]{AHP20}. Let $\oc{fpl1}>0$ be an absolute constant to be decided. Let $\oc{fpl2}$ be the corresponding constant defined in Proposition~\ref{prop:fast power law}. Let {$\ell = \lceil\oc{fpl2}^{\frac{8}{7}}\epsilon^{-\frac{8}{7}}\rceil$} and thus {$\eps\ge \oc{fpl2}\ell^{-\frac{7}{8}}$} satisfies the condition in Proposition~\ref{prop:fast power law}.  We can choose $\oc{01}$ in Theorem~\ref{thm:main thm-upper bound} big enough such that $N \geq 2\ell$.

Now without loss of generality, we can assume that $N$ is divisible by 2. Recall $\Lambda_\ell(v)$ is a translated copy of $\Lambda_\ell$ centered at $v$. Then define $o\in\mathcal{V}\subset\Lambda_{ N/2}$ to be a set of vertices such that (the union of) $\{\Lambda_\ell(v): v\in \mathcal{V}\}$ covers $\Lambda_{N/2}$ and each $\Lambda_{2\ell}(v)$ intersects at most 100 boxes in $\{\Lambda_{2\ell}(v): v\in \mathcal V\}$. We further require (and it is clear that this is feasible) $\Lambda_{\ell}(o)$ to be the only box in $\{\Lambda_{\ell}(v): v\in \mathcal V\}$ that contains the origin $o$.

By \eqref{eq:boundary influence},
we only need to bound the probability that there is a pre-disagreement crossing joining $o$ and $\intB\Lambda_N$, which in turn is bounded by the probability that pre-disagreements percolate through boxes {$\{\Lambda_{\ell}(v): v\in \mathcal V\}$} to $o$. To carry this out, define
\begin{align*}
    \mathfrak P_K=\{(v_1,\cdots,v_K): ~&\Lambda_{\ell}(v_k) \mbox{ is either neighboring or overlapping with } \Lambda_{\ell}(v_{k+1}) \mbox{ for } 1\leq k < K\\
    &\mbox{ and }v_1=o
    \}.
\end{align*}
Let $F_v$ be the event that 
$\intB\Lambda_{\ell{+1}}(v)\cap \mathcal{D}_{\intB}\neq \emptyset$. If $o\in\mathcal{D}_{\intB}$, then there exist $K \geq N/4\ell$ and $(v_1, \ldots, v_K) \in \mathfrak P_K$ such that the event $F_{v_k}$ occurs for $1\leq k\leq K$.
So we have
\begin{align*}
    \E\langle\sigma_o\rangle^{+/-}_{\Lambda_N,\eps h}&\leq \E\left(\left\langle \{\substack{\exists (v_1, \ldots, v_K) \in \mathfrak P_K \text{ with } K \geq N/4\ell \\ \text{ such that } F_{v_k}\text{ holds for }1\leq k\leq K}\}\cap\{o\stackrel{\cD}{\longleftrightarrow}\intB\Lambda_{\ell}(o)\}\right\rangle^{+/-}_{\Lambda_N,\eps h}\right)\nonumber\\
    &\leq \sum_{\substack{\exists (v_1, \ldots, v_K) \in \mathfrak P_K \\ K \geq N/4\ell }}\E\left(\left\langle \1_{\{o\stackrel{\cD}{\longleftrightarrow}\intB\Lambda_{\ell}(o)\}}\cdot\prod_{k={1}}^K\1_{F_{v_k}}\right\rangle^{+/-}_{\Lambda_N,\eps h}\right).
\end{align*}

Recalling $\mcc V\subset\Lambda_{N/2}$, we see that $\Lambda_{2\ell}(v_{k})\subset \Lambda_N$. Thus, on the event $F_{v_k}$ there exists a pre-disagreement path connecting $\intB\Lambda_{2\ell}(v_k)$ and $\intB\Lambda_{\ell+1}(v_k)$, which is a translation of the event $\con(\ell,2\ell)$.
Given $(v_1, \ldots, v_K) \in \mathfrak P_K$, we can
choose a subset $\mcc V_0=\{v_{i_1},\cdots, v_{i_\mathsf K}\}$ with $o\in \mathcal V_0$ such that $\mathsf K \geq K/100$ and $\Lambda_{2\ell}(v_{i_1}), \ldots, \Lambda_{2\ell}(v_{i_{\mathsf K}})$ are disjoint.

Recalling Proposition~\ref{prop:fast power law} 
{and applying} DMP and CBC, we deduce that: 
 \begin{align*}
 &\E\left(\left\langle \1_{\{o\stackrel{\cD}{\longleftrightarrow}\partial\Lambda_{\ell}(o)\}}\cdot\prod_{k={1}}^{{K}} \1_{F_{v_k}}\right\rangle^{+/-}_{\Lambda_N,\eps h}\right)\nonumber\\
      \le& \prod_{v\in\mathcal{V}_0
      } \E\left(\left\langle \1_{\{\intB\Lambda_{\ell{+1}}(v)\cap \mathcal{D}_{\intB \Lambda_{2\ell}(v)}\neq\emptyset\}}\right\rangle_{\Lambda_{{\ell,} 2\ell}(v),\eps h}^{+/-}\right)\times\E\langle\sigma_o\rangle^{+/-}_{\Lambda_{\ell}(o),\eps h}\nonumber\\ \le &~  \oc{fpl1}^{|\mathcal{V}_0|}m(T_c,\ell,\eps).
      \label{eq:preparation for percolation}
 \end{align*}

Finally, notice that $|\mathfrak P_K| \leq 100^K$. Then for $\oc{fpl1}$ sufficiently small, 
we get  
\begin{equation*}\label{eq: percolation argument for boundary influence}
    m(T_c,N, \epsilon) \leq m(T_c,\ell, \epsilon)\sum_{K\geq (N/4\ell)} 100^K (\oc{fpl1} )^{\frac{K}{100}} \leq m(T_c,\ell, \epsilon)e^{- C_1N/\ell}\le C_2\ell^{-\frac{1}{8}}e^{- C_1N/\ell}
\end{equation*}
where the last inequality follows from \cite[Theorem 1.1]{DSS22} and \eqref{eq: origin decay rate without disorder}.
This completes the proof since {$\ell = \lceil\oc{fpl2}^{\frac{8}{7}}\epsilon^{-\frac{8}{7}}\rceil$}.
\end{proof}
We leave the proof of Proposition \ref{prop:fast power law} to Section~\ref{sec:Proof of Proposition 3.10} and the proof will rely on the RSW estimate (Theorem~\ref{thm: RSW for dis with external field})
 we presented at the beginning of this subsection.
\subsubsection{Lower bound}\label{subsubsect:lower-bound}
In this subsection, we outline the proof for Theorem \ref{thm:main thm-lower bound}. The idea is to use Theorem~\ref{thm: RSW for dis with external field} to construct a disagreement path from the boundary to a box around the origin $o$.
However, in Theorem~\ref{thm: RSW for dis with external field}, we can only show for a typical disorder that the probability {for} the existence of a pre-disagreement hard crossing has a lower bound. In order to get a disagreement crossing from the boundary to the origin $o$, an intermediate step is to get a sequence of neighboring $M$-boxes each with such a typical disorder (so that a pre-disagreement crossing occurs with good probability). Thus, it is natural to apply some sort of percolation argument. We next carry out the details.

\begin{defi}\label{def: good external field for RSW}
For any $u\in\mbb Z^2$ and integer $M\geq 1$, let $\aro(u;M)$ denote the event that there exists a pre-disagreement circuit in $\Lambda_{M,2M}(u)$.
We say the external field $h|_{\Lambda_{5M}(u)}\in \mathbb{R}^{\Lambda_{5M}(u)}$ is \textbf{(i)-good} if and only if for any $\xi^+,\xi^-\in\{-1,1\}^{\intB\Lambda_{5M}(u)}$, $$\bar\mu^{\xi^+/\xi^-}_{\Lambda_{5M(u)},\eps h}(\aro(u;M))\ge \oc{9}^4.$$ Recall that $\oc{9}$ is defined in Theorem~\ref{thm: RSW for dis with external field} (with $\iota=4$) to be the lower bound on the probability of a hard pre-disagreement crossing. 

Besides, we say the external field $h|_{\Lambda_{5M}(u)}\in \mathbb{R}^{\Lambda_{5M}(u)}$ is \textbf{(ii)-good} if and only if  there exists $J\subset \Lambda_{M}(u)\mbox{ with } |J|\geq 2M^2$ such that for any $x\in J$ the following holds: 
\begin{equation}\label{eq: def of (ii)-good boxes}
    \langle\1_{{\{}x\stackrel{\mcc D}\longleftrightarrow\intB\Lambda_{4M}(x){\}}}\rangle^{+/-}_{\Lambda_{4M}(x),\eps h}\geq \frac{1}{2}\langle\1_{{\{}x\stackrel{\mcc D}\longleftrightarrow\intB\Lambda_{4M}(x){\}}}\rangle^{+/-}_{\Lambda_{4M}(x),0}\,.
\end{equation}

We say $h|_{\Lambda_{5M}(u)}\in \mathbb{R}^{\Lambda_{5M}(u)}$ is \textbf{good} if it is {both} (i)-good and (ii)-good. In this case, we interchangeably say that 
the external field $h$ is good with respect to $\Lambda_M(u)$ or $\Lambda_M(u)$ is {good} under $h$. We use the word \textbf{bad} when it is not good.
\end{defi}
\begin{rmk}
    It is worth emphasizing that we will consider good boxes with different scales (i.e., different M), but $\oc{9}$ does not depend on $M$ in Theorem \ref{thm: RSW for dis with external field}. 
\end{rmk}

\begin{rmk}
    We point out that in Section~\ref{subsubsect:lower-bound}
    we only use the (i)-good property, and the (ii)-good property is defined for later convenience. 
\end{rmk}


By Lemma~\ref{lem: domination by product measure} below, we can easily obtain that for $\P$-probability at least $0.9^{N/M}$, we have a consecutive horizontal sequence of good $M$-boxes from the boundary to the $M$-box containing the origin. However, in order to get a lower bound on $\bar\mu_{\lamn,\eps h}^{+/-}(o\in \mathcal D_{\partial})$ for a \emph{typical} instance of the disorder, we need a refined estimate on good boxes, which calls for a percolation type of argument. {Let $\{B_i\}_{i\in I}$ be a partition of $\Lambda_{N}$ into $M$-boxes (without loss of generality we assume $M$ divides $N$) and from now on, we will only focus on these $M$-boxes.}

We divide the proof into two cases: Case 1 when $\eps>N^{-\frac{1}{5}}$ and Case 2 when $\eps\le N^{-\frac{1}{5}} $. In Case 1, we will define $N_1$ to be an integer such that $N_1\sim (\eps^{\frac{8}{7}} N)^{\frac{1}{10}}\eps^{-\frac{8}{7}}$. Then the sum of the absolute values of disorder on $\Lambda_{N_1}$ is typically smaller than $\eps^{\frac{8}{7}} N$, and thus the influence of disorder in $\Lambda_{N_1}$ is no more than $\exp(-C\eps^{\frac{8}{7}} N)$. Thus it suffices to consider how disagreements percolate from the boundary to $\Lambda_{N_1}$. To this end, we make the following definition. In what follows, we say a collection of boxes separating $A$ and $B$ if any path from $A$ to $B$ has to intersect at least one of the boxes in this collection. 

\begin{defi}\label{def: good external field 1}
    For any constant $c>0$ and integers $N>N_1>M$, we define $\cH_{\star}(c,
N,N_1,M)$ to be the collection of the external field $h\in \mathbb R^{\mbb Z^2}$ satisfying the following two conditions:
    \begin{enumerate}
        \item There exists a neighboring contour of good $M$-boxes separating the inner and outer boundaries of the annulus $\Lambda_{N_1,2N_1}$.\label{item:11}

        \item There exists a neighboring sequence of good $M$-boxes connecting $\intB\lamn$ to $\intB\Lambda_{N_1}$ where the number of $M$-boxes in this sequence is at most $\frac{cN}{M}$.\label{item:12}
    \end{enumerate}
\end{defi}
\begin{lem}\label{lem: good external field probability 0}
    There exist constants $\nc\label{gefp00},\nc\label{gefp01},\nc\label{gefp02}>0$ such that for any integers $N>N_1>M$ with disorder strength {$\eps\leq\oc{gefp00}M^{-\frac{7}{8}}$}, we have $$\P\Big(\cH_{\star}(\oc{gefp01},N,N_1,M)\Big)\ge 1-\oc{gefp02}^{-1}(\frac{N}{M})^4\exp(-\frac{\oc{gefp02}N_1}{M}).$$
\end{lem}

\begin{defi}
    Let $\cH_0$ denote the collection of external field $h$ such that $|h_v|\le (\eps^{\frac{8}{7}} N)^{\frac{1}{10}}$ for any $v\in \Lambda_{N_1}$.
\end{defi}
It is not hard to compute that \begin{equation}\label{eq: not too big external field probability}
    \P(\cH_0)\ge 1-N_1^2\exp(-c(\eps^{\frac{8}{7}} N)^{\frac{1}{10}}).
\end{equation}

\begin{proof}[Proof of Theorem \ref{thm:main thm-lower bound}:Case 1]
Recall that $\eps^{-1}<N^{\frac{1}{5}}$. Let $M=\lfloor\oc{gefp00}^{8/7}\eps^{-8/7}\rfloor$, let $N_1=\lfloor(\eps^{\frac{8}{7}} N)^{\frac{1}{10}}\cdot\eps^{-\frac{8}{7}}\rfloor$ {and let $\oc{03}>0$ be big enough} such that $N>N_1>M$. Since $\frac{N_1}{M}\ge C(\eps^{\frac{8}{7}} N)^{\frac{1}{10}}$ for some $C>0$, by Lemma~\ref{lem: good external field probability 0} and \eqref{eq: not too big external field probability} it suffices to consider $h\in \cH_{\star}(\oc{gefp01},N,N_1,M)\cap \cH_0$.
Recalling $\aro(u;M)$ is the event that there exists a pre-disagreement circuit in the annulus $\Lambda_{M,2M}(u),$
 we first prove \begin{align}
    &\mu^{+/-}_{\lamn,\eps h}(\aro(o;N_1))\ge C_1\exp(-C_1^{-1}\eps^{\frac{8}{7}}N),\label{eq: disagreement circuit in the annulus 2N_1, N_1 0}\\
        &\mu^{+/-}_{\lamn,\eps h}(\cD_{\partial}\cap\intB\Lambda_{N_1}\neq\emptyset)\ge C_2\exp(-C_2^{-1}\eps^{\frac{8}{7}}N),\label{eq: disagreement percolation from origin to 2N_1 0}
    \end{align}for some constants $C_1,C_2>0$.
    To this end, let $\Lambda_{M}(u_1),\cdots,\Lambda_{M}(u_K)$ be a contour of good $M$-boxes whose existence is ensured by Item \ref{item:11} of Definition \ref{def: good external field 1}; let $\Lambda_{M}(v_1),\cdots,\Lambda_{M}(v_J)$ for $J\le \frac{\oc{gefp01}{N}}{M}$ be a sequence of good $M$-boxes whose existence is ensured by Item \ref{item:12} of Definition \ref{def: good external field 1}.
    Define $$\aro^1_{k}=\aro(u_k;M) \text{ for }1\le k\leq K; \quad \aro^2_j=\aro(v_j;M)\text{ for }1\le j\leq J.$$     
      Then, applying FKG (Lemma ~\ref{lem:FKG}), CBC (Corollary~\ref{cor: CBC}) and recalling Definition \ref{def: good external field for RSW}, we get that
\begin{align}
      &\bar\mu^{+/-}_{\lamn,\eps h}(\aro(o;N_1))\ge \bar\mu^{+/-}_{\lamn,\eps h}(\cap_{k=1}^{K}\mathtt{Around}^1_k)\nonumber\\ \stackrel{\mbox{FKG}}\ge& \prod_{k=1}^{K}\bar\mu^{+/-}_{\lamn,\eps h}(\mathtt{Around}^1_k)\stackrel{\mbox{CBC}}\ge \prod_{k=1}^{K}\bar\mu^{-/+}_{\Lambda_{5M}(u_k),\eps h}(\mathtt{Around}^1_k)\ge \oc{9}^{4K};\label{eq: pre-disagreement path constructed in good boxes through RSW 01}\\&\bar\mu^{+/-}_{\lamn,\eps h}(\{\cD_{\partial}\cap\intB\Lambda_{N_1}\neq \emptyset\})\ge \bar\mu^{+/-}_{\lamn,\eps h}(\cap_{j=1}^{J}\mathtt{Around}^2_j) \nonumber\\ \stackrel{\mbox{FKG}}\ge& \prod_{j=1}^{J}\bar\mu^{+/-}_{\lamn,\eps h}(\mathtt{Around}^2_j)\stackrel{\mbox{CBC}}\ge \prod_{j=1}^{J}\bar\mu^{-/+}_{\Lambda_{5M}(v_j),\eps h}(\mathtt{Around}^2_j)\ge \oc{9}^{4J}.\label{eq: pre-disagreement path constructed in good boxes through RSW 02}
\end{align} 
Combining \eqref{eq: pre-disagreement path constructed in good boxes through RSW 01}, \eqref{eq: pre-disagreement path constructed in good boxes through RSW 02} with the facts that $K\le 4N_1^2\le \eps^{\frac{8}{7}} N$ and $J\le \frac{\oc{gefp01}{N}}{M}$, we finish the proof of \eqref{eq: disagreement circuit in the annulus 2N_1, N_1 0} and \eqref{eq: disagreement percolation from origin to 2N_1 0}. 

Let $\daro(o;N_1)$ denote the event that there exists a disagreement circuit in the annulus $\Lambda_{N_1,2N_1}$. On the event $\aro(o;N_1)$, there exists a pre-disagreement circuit $\gamma$ in the annulus $\Lambda_{N_1,2N_1}$. Furthermore, given $\cD_{\partial}\cap\intB\Lambda_{N_1}\neq\emptyset$ and $\aro(o; N_1)$, we obtain a disagreement path $\eta$ from $\intB\lamn$ to $\intB\Lambda_{N_1}$. Then $\gamma\cap\eta\neq\emptyset,$  and thus $\daro(o;N_1)$ happens. Combining \eqref{eq: disagreement circuit in the annulus 2N_1, N_1 0}, \eqref{eq: disagreement percolation from origin to 2N_1 0} and applying Lemma~\ref{lem:FKG}, we obtain that \begin{equation}\label{eq: disagreement percolate from boundary to M_2 box at o 0}
       \mu^{+/-}_{\lamn, \eps h}(\daro(o;N_1))\ge C_3\exp(-C_3^{-1}\eps^{\frac{8}{7}}N),
   \end{equation}for some $C_3>0$. Applying CBC' (Lemma~\ref{lem:CBC prime}), we have        \begin{align}
           \mu^{+/-}_{\lamn, \eps h}(\{o \in\cD_{\partial}\}\mid \daro(o;N_1))&~\ge \mu^{+/-}_{\Lambda_{2N_1}, \eps h}(\{o \in\cD_{\partial\Lambda_{2N_1}}\}).\label{eq: disagreement probability in small N_1 box}
       \end{align}
Recalling that $\eps>N^{-\frac{1}{5}}$ and our assumption that $h\in \mathcal H_0$, we get that $$\sum_{v\in \Lambda_{2N_1}}|\eps h_v|\le ({4}N_1)^2\cdot\eps\cdot(\eps^{\frac{8}{7}} N)^{\frac{1}{10}}\le {16}\eps^{-\frac{33}{35}}N^{\frac{3}{10}}\le {16}\eps^{\frac{8}{7}} N.$$ Therefore, for any configuration $\sigma \in \{1,-1\}^{\Lambda_{2N_1-1}}$, \begin{equation}\label{eq: Hamiltonian difference}
    -{16}\eps^{\frac{8}{7}} N/T_c\le H^\pm_{\Lambda_{2N_1},0}(\sigma)-H^\pm_{\Lambda_{2N_1},\eps h}(\sigma)=\sum\limits_{x\in\Lambda_{2N_1-1}}\frac{\eps h_x\sigma_x}{T_c}\le {16}\eps^{\frac{8}{7}} N/T_c.
\end{equation} 
Hence we get from \cite[Lemma 2.1]{AHP20}
that \begin{equation}\label{eq: partition function ratio}
    {\exp(-16\eps^{\frac{8}{7}} N/T_c)\le \frac{\bar\cZ^{\pm}_{\Lambda_{2N_1},\eps h}}{\bar\cZ^{\pm}_{\Lambda_{2N_1},0}}=\frac{\cZ^{\pm}_{\Lambda_{2N_1},\eps h}}{\cZ^{\pm}_{\Lambda_{2N_1},0}}
    =\frac{\sum_{\bar\sigma}\exp(-H^\pm_{\Lambda_{2N_1},\eps h}(\sigma)/T_c)}{\sum_{\bar\sigma}\exp(-H^\pm_{\Lambda_{2N_1},0}(\sigma)/T_c)}
    \le\exp(16\eps^{\frac{8}{7}} N/T_c).}
\end{equation}
{Recall the definitions in Section~\ref{sect:extended-Ising}. For any $\bar\sigma$, we define $$H^\pm_{\Lambda_{2N_1},\eps h}(\bar\sigma)=H^\pm_{\Lambda_{2N_1},\eps h}(\sigma)-\sum_{\substack{e\in E(G)\\v\in e}}\log\big(W(\bar\sigma_v,\bar\sigma_e)\big)$$ to be the Hamiltonian of $\bar\sigma$. Then the extended ising measure can be written similarly as \eqref{def_mu}. Note that}
the conditional distribution of edge spins given the configuration for vertex spins does not depend on the external field. Therefore, 
{$H^\pm_{\Lambda_{2N_1},0}(\bar\sigma)-H^\pm_{\Lambda_{2N_1},\eps h}(\bar\sigma)=H^\pm_{\Lambda_{2N_1},0}(\sigma)-H^\pm_{\Lambda_{2N_1},\eps h}(\sigma)$. Combined with \eqref{eq: Hamiltonian difference} and \eqref{eq: partition function ratio}, it yields}
that the Radon-Nikodym derivative between $\bar\mu^{+/-}_{\Lambda_{2N_1},\eps h}$ and $\bar\mu^{+/-}_{\Lambda_{2N_1},0}$ is lower-bounded by $\exp(-{64}\eps^{\frac{8}{7}} N/T_c)$. Hence we obtain that 
        \begin{align}
            \bar\mu^{+/-}_{\Lambda_{2N_1},\eps h}(\{o \in\cD_{\partial\Lambda_{2N_1}}\})&\ge \bar\mu^{+/-}_{\Lambda_{2N_1},0}(\{o \in\cD_{\partial\Lambda_{2N_1}}\})\exp(-{64}\eps^{\frac{8}{7}} N/T_c)\nonumber\\& \ge C_4 N_1^{-\frac{1}{8}} \exp(-{64}\eps^{\frac{8}{7}} N/T_c),\label{eq: spin spin correlation with disorder 0}
        \end{align}
where the last inequality follows from \eqref{eq: origin decay rate without disorder} and \eqref{eq:DisagreementRepresentation}. (Note that in the above, thanks to \eqref{eq:DisagreementRepresentation} we apply the Radon-Nikodym derivative bound to the product measure instead of to each copy separately, since otherwise it is hard to control the error terms.) Combined with \eqref{eq: disagreement percolate from boundary to M_2 box at o 0} and \eqref{eq: disagreement probability in small N_1 box}, it yields that \begin{equation*}
        \begin{aligned}
            \bar\mu^{+/-}_{\lamn,\eps h}(\{o \in\mcc D_{\partial}\})&\stackrel{\eqref{eq: disagreement probability in small N_1 box}}{\ge}\mu^{+/-}_{\lamn, \eps h}(\daro(o;N_1))\cdot\mu^{+/-}_{\Lambda_{2N_1}, \eps h}(\{o \in\cD_{\partial\Lambda_{2N_1}}\})\\ &\stackrel{\eqref{eq: disagreement percolate from boundary to M_2 box at o 0}}{\ge} C_3\exp(-C_3^{-1}\eps^{\frac{8}{7}}N))\cdot \mu^{+/-}_{\Lambda_{2N_1}, \eps h}(\{o \in\cD_{\partial\Lambda_{2N_1}}\})\\
            &\stackrel{\eqref{eq: spin spin correlation with disorder 0}}{\ge}  C_3\exp(-C_3^{-1}\eps^{\frac{8}{7}}N))\cdot C_{4} N_1^{-\frac{1}{8}} \exp(-64\eps^{\frac{8}{7}} N/T_c).
        \end{aligned}
    \end{equation*}
This completes the proof for the lower bound in Case 1.
\end{proof}

In Case $2$ when $\eps\le N^{-\frac{1}{5}}$, it is not possible to control the Radon-Nikodym derivatives since $N_1\sim \eps^{-\frac{8}{7}}(\eps^{\frac{8}{7}}N)^{\frac{1}{10}}$ may be too big in this case (especially if $\epsilon$ has the order of $N^{-7/8}$). We will use an alternative approach to control the probability of ${\{}o \stackrel{\cD}\longleftrightarrow  \partial \Lambda_{N_1}{\}}$. Let $M_*$ be an integer such that $\eps^{-\frac{8}{7}}M_*^{-1}\sim (\eps^{\frac{8}{7}} N)^{\frac{1}{10}}$. ($M_*$ is well-defined since $\eps $ is small.) Since (as we will show that) the probability for an $M_*$-box to be good is very close to $1$, the percolation formed by good $M_*$-boxes is very supercritical. As a result, we can get that with probability close to $1$, the origin $o$ is connected via a sequence of neighboring good $M_*$-boxes to the boundary $\intB \Lambda_{N_1}$. In order to carry out the details, we need the following modification of Definition \ref{def: good external field 1}. 
\begin{defi}\label{def: good external field 2}
    For any constant $c>0$ and integers $N>N_2>M$, we define $\cH_{\diamond}(c,
N,N_2,M)$ to be the collection of the external field $h\in \mathbb R^{{\mbb Z^2}}$ satisfying the following two conditions:
    \begin{enumerate}
        \item There exists a neighboring contour of good $M$-boxes separating the inner and outer boundaries of the annulus $\Lambda_{N_2,2N_2}$, and the number of boxes in this contour is at most $\frac{cN_2}{M}$.\label{item:21}

        \item There exists a neighboring sequence of good $M$-boxes that connects $\intB \Lambda_{2N_2}$ and $\Lambda_M$, and the number of boxes in this sequence is at most $\frac{cN_2}{M}$.\label{item:22}
    \end{enumerate}
\end{defi}
\begin{lem}\label{lem: good external field probability 1}
    There exist constants $\nc\label{gefp10},\nc\label{gefp11},\nc\label{gefp12}>0$ such that for any integers $N>N_2>M$ and disorder strength satisfying {$\eps^{-1}M^{-\frac{7}{8}}\geq\oc{gefp10}(\eps^{\frac{8}{7}}N)^{\frac{1}{5}}$} and $\frac{N_2}{M}\le \eps^{\frac{8}{7}}N$, we have $$\P\Big(\cH_{\diamond}(\oc{gefp11},N,N_2,M)\Big)\ge 1-{\oc{gefp12}^{-1}}\exp(-\oc{gefp12}(\eps^{\frac{8}{7}}N)^{\frac{1}{10}}).$$
\end{lem}

Provided with Lemmas~\ref{lem: good external field probability 0} and \ref{lem: good external field probability 1}, we are ready to provide the proof in Case 2.
\begin{proof}[Proof of Theorem \ref{thm:main thm-lower bound}:Case 2]
Recall that in Case 2 we have $\eps^{-1}\ge N^{\frac{1}{5}}$.
     Let $M_1=\lfloor\oc{gefp00}^{8/7}\eps^{-8/7}\rfloor$ be such that {$\eps\leq\oc{gefp00}M_1^{-\frac{7}{8}}$}. Let $M_2=\lfloor\Big(\oc{gefp10}^{-1}\eps^{-1}\cdot (\eps^{\frac{8}{7}}N)^{-\frac{1}{5}}\Big)^{\frac{8}{7}}\rfloor$. Thanks to the assumption $\eps^{-1}\ge N^{\frac{1}{5}}$, we see that $M_2\ge 1$ is well-defined. Let $N_2=\lfloor\eps^{\frac{8}{7}}N\cdot M_2\rfloor$. We can let $\oc{03}>0$ be big enough such that $N>N_2>M_1>M_2$. By Lemmas \ref{lem: good external field probability 0}, \ref{lem: good external field probability 1} and Theorem~\ref{thm-critical-temperature-small-perturbation}, it suffices to consider the external field $h$ such that \eqref{eq-small-perturbation} holds for {$\Lambda_{2M_2}$}
     (with $\theta=\frac{1}{2}$) and that $$h\in \cH_{\star}(\oc{gefp01},N,N_2,M_1)\cap \cH_{\diamond}(\oc{gefp11},N,N_2,M_2).$$ We first prove \begin{align}
     &\mu^{+/-}_{\lamn,\eps h}(\cD_{\partial}\cap\intB\Lambda_{N_2}\neq\emptyset)\ge C_1\exp(-C_1^{-1}\eps^{\frac{8}{7}}N),\label{eq: disagreement percolate to N_2}\\
    &\mu^{+/-}_{\lamn,\eps h}(\aro(o;N_2))\ge C_2\exp(-C_2^{-2}\eps^{\frac{8}{7}}N),\label{eq: disagreement circuit in the annulus 2N_2, N_2}\\
        &\mu^{+/-}_{\lamn,\eps h}(\{\intB\Lambda_{M_2+1}\stackrel{\cD}{\longleftrightarrow}\intB\Lambda_{2N_2})\ge C_3\exp(-C_3^{-1}\eps^{\frac{8}{7}}N),\label{eq: disagreement percolation from origin to 2N_2}\\
        &\mu^{+/-}_{\lamn,\eps h}(\aro(o;M_2))\ge C_4,\label{eq: disagreement circuit in the annulus 2M_2, M_2}
    \end{align}for some constants $C_1,C_2,C_3,C_4>0$. 
    
    The proof of \eqref{eq: disagreement percolate to N_2} is the same as that of \eqref{eq: pre-disagreement path constructed in good boxes through RSW 02}, and \eqref{eq: disagreement circuit in the annulus 2M_2, M_2} comes from CBC and the definition of a good $M_2$-box. In order to prove \eqref{eq: disagreement circuit in the annulus 2N_2, N_2} and \eqref{eq: disagreement percolation from origin to 2N_2},  let $\Lambda_{M_2}(u_1),\cdots,\Lambda_{M_2}(u_K)$ for $K\le \frac{\oc{gefp11}N_2}{M_2}$ be a contour of good $M_2$-boxes whose existence is ensured by Item \ref{item:21} of Definition \ref{def: good external field 2}; let $\Lambda_{M_2}(v_1),\cdots,\Lambda_{M_2}(v_J)$ for $J\le \frac{\oc{gefp11}N_2}{M_2}$ be a sequence of good $M_2$-boxes whose existence is ensured by Item \ref{item:22} of Definition \ref{def: good external field 2}. Thus it suffices to construct a pre-disagreement circuit in each annulus of the form $\Lambda_{M_2,2M_2}(u_k)$ and $\Lambda_{M_2,2M_2}(v_j)$. Thanks to the restriction that $\frac{N_2}{M_2}\le \eps^{\frac{8}{7}}N$, the verifications are highly similar to that of \eqref{eq: pre-disagreement path constructed in good boxes through RSW 01} and \eqref{eq: pre-disagreement path constructed in good boxes through RSW 02} and thus we omit further details.

   We now prove the theorem provided with \eqref{eq: disagreement percolate to N_2}, \eqref{eq: disagreement circuit in the annulus 2N_2, N_2}, \eqref{eq: disagreement percolation from origin to 2N_2} and \eqref{eq: disagreement circuit in the annulus 2M_2, M_2}. Recall 
   $\daro (o;M_2)$ is the event that there exists a disagreement circuit in the annulus $\Lambda_{M_2,2M_2}$. Given $\aro(o;N_2)$, we obtain a pre-disagreement circuit $\gamma$ in the annulus $\Lambda_{N_2,2N_2}$. Thus, on the event $$\{\cD_{\partial}\cap\intB\Lambda_{N_2}\neq\emptyset\}\cap\aro(o;N_2)\cap\{\intB\Lambda_{M_2+1}\stackrel{\cD}{\longleftrightarrow}\intB\Lambda_{2N_2}\}\cap\aro(o;M_2),$$ there exist a disagreement path from $\intB\lamn$ to $\intB\Lambda_{N_2+1}$ and a pre-disagreement path from $\intB\Lambda_{2N_2}$ to $\intB\Lambda_{M_2+1}$. In addition, these two paths both intersect with $\gamma$, and thus $\daro(o;M_2)$ happens. Combining \eqref{eq: disagreement percolate to N_2}, \eqref{eq: disagreement circuit in the annulus 2N_2, N_2}, \eqref{eq: disagreement percolation from origin to 2N_2}{, \eqref{eq: disagreement circuit in the annulus 2M_2, M_2}} and applying Lemma~\ref{lem:FKG} (FKG), we obtain that \begin{equation}\label{eq: disagreement percolate from boundary to M_2 box at o}
       \mu^{+/-}_{\lamn, \eps h}(\daro(o;M_2))\ge C_5\exp(-C_5^{-1}\eps^{\frac{8}{7}}N),
   \end{equation}for some $C_5>0$. Applying CBC' (Lemma~\ref{lem:CBC prime}) and Theorem~\ref{thm-critical-temperature-small-perturbation}, we have (recall that we have assumed $\theta=\frac{1}{2}$ when posing constraints on $h$ earlier)    
   \begin{align}
           \mu^{+/-}_{\lamn, \eps h}(\{o \in\cD_{\partial}\}\mid \daro(o;M_2))&~\ge \mu^{+/-}_{\Lambda_{2M_2}, \eps h}(\{o \in\cD_{\partial\Lambda_{2M_2}}\})~~~\text{(by CBC')}\nonumber\\&~\ge (1-\theta)m(T_c,2M_2,0)~~~\text{(by Theorem \ref{thm-critical-temperature-small-perturbation})}\nonumber\\&\stackrel{\eqref{eq: origin decay rate without disorder}}{\ge} C_6M_2^{-\frac{1}{8}}.\label{eq: disagreement probability in small M_2 box}
       \end{align}
   Combining \eqref{eq: disagreement percolate from boundary to M_2 box at o} and \eqref{eq: disagreement probability in small M_2 box}, we obtain that \begin{equation*}
       \begin{aligned}
           \mu^{+/-}_{\lamn, \eps h}(\{o \in\cD_{\partial}\})&\ge \mu^{+/-}_{\lamn, \eps h}(\{o \in\cD_{\partial}\}\mid \daro(o;M_2))\times\mu^{+/-}_{\lamn, \eps h}(\daro(o;M_2))\\&\ge C_6M_2^{-\frac{1}{8}}\times C_5\exp(-C_5^{-1}\eps^{\frac{8}{7}}N)\\&\ge C_7N^{-\frac{1}{8}}\times \exp(-C_7^{-1}\eps^{\frac{8}{7}}N).
       \end{aligned}
   \end{equation*}
   Note that in the last inequality when $M_2 \ll N$, we have that the stretched exponential term is much smaller than the polynomial prefactor and this is why the inequality can hold by changing from $C_5$ to $C_7$.
\end{proof}

\subsection{RSW for disagreement Ising model}\label{sec: outline for RSW}
In this section, we outline the proof of Theorem \ref{thm: RSW for dis with external field}. It is clear, also as hinted by the statement of Theorem \ref{thm: RSW for dis with external field},  that the RSW estimate cannot hold with an arbitrary external field. For instance, if it turns out that $h_v \geq{10}$ for all $v$, then we will have exponential decay and as a result will not have an effective RSW estimate. In order to prove Theorem \ref{thm: RSW for dis with external field}, there are two possible routes:
\begin{itemize}
    \item First bound the easy crossing probability with a typical disorder and then extend it to the hard crossing probability. After attempts, we found that one is able to derive the easy crossing probability by using a similar technique as in the proof of Theorem \ref{thm-critical-temperature-small-perturbation} and also using Proposition \ref{prop:DisagreementRep}, but it seems to us rather difficult to extend the easy crossing probability to the hard crossing probability. Despite the fact that we now have a very powerful RSW method from \cite{KT23}, in our setting, the disorder ruins the translation invariance, and as a result it seems very difficult to adapt arguments in \cite{KT23}.
    \item We first bound the hard (and thus easy) crossing probability without disorder and then extend it to the case with disorder by the method of chaos expansion. This is the route we end up employing.
\end{itemize}
As a warm up, we will first show the RSW estimate for the disagreement Ising model without disorder. 
\begin{thm}\label{thm: RSW for dis at 0 external field}
    For any $\iota>0$, there exists a constant $\nc\label{6}=\oc{6}(\iota)>0$, such that for any integer $M>0$ and any boundary conditions $\xi^+,\xi^-$ on $\intB \cR(\iota M+M,2M)$, we have that \begin{equation*}\label{eq: RSW for dis at 0 external field}
        \bar\mu^{\xi^+/\xi^-}_{\cR(\iota M+M,2M),0}(\hc(\iota M,M))>\oc{6}.
    \end{equation*}
\end{thm}
\begin{rmk}\label{rmk: weird +/- notation}
    We do not require $\xi^+\ge \xi^-$ in the previous theorem, but we use this notation to remind us that on a pre-disagreement crossing spin values of the first copy (i.e., under the boundary condition $\xi^+$) are greater than {those of} the second copy.
\end{rmk}
We now explain how we proceed provided with Theorem \ref{thm: RSW for dis at 0 external field}. First, as pointed out in Remark \ref{rmk: anti reduction}, it suffices to consider the anti-disagreement boundary condition. Thus, in the remaining of this subsection, we assume that $\xi^+_v=-1$ and $\xi^-_v=1$ for all $v\in \partial\cR(\iota M+M,2M)$. For notation clarity, in the rest of this subsection, we fix $\iota=4$ (since this is the only case we will need later and other cases can be done similarly) and fix $M$ to be a large enough integer. In addition, we use $\cR$,  $\cR'$ and $\hc$ to denote $\cR(\iota M+M,2M),\cR(\iota M,M)$ and $\hc(\iota M,M)$ respectively. 


Recalling the expansion in \eqref{eq: partition function expansion}, we obtain that\begin{equation}\label{eq: expectation for dis crossing}
\begin{aligned}
        \bar\mu_{\cR,\eps h}^{-/+}(\hc)&= \frac{\langle ~\1_{\hc} \prod_{v\in\cR}\Big([1+\bar \sigma_v^+\tanh(\veps h_v)]\cdot [1+\bar \sigma_v^-\tanh(\veps h_v)]\Big)~\rangle_{\cR,0}^{-/+}}{\langle ~\prod_{v\in\cR}[1+ \sigma_v^+\tanh(\veps h_v)]~\rangle_{\cR,0}^{-}\cdot \langle ~\prod_{v\in\cR} [1+ \sigma_v^-\tanh(\veps h_v)]~\rangle_{\cR,0}^{+}}.
\end{aligned}
\end{equation}
Here $\bar \sigma^+$ and $\bar \sigma^-$ are independent spin configurations sampled from the extended Ising model with the minus and plus boundary conditions, respectively. 
In order to bound \eqref{eq: expectation for dis crossing}, we note that its denominator can be upper-bounded by  Definition \ref{def: good external field for partition function} and Lemma \ref{lem: small perturbation for partition function}, and thus it remains to control its numerator. To this end, we need a similar result as Lemma \ref{lem: upper-bound for the sum of squares of k point function} to control the influence on the crossing event from the external field.

\begin{defi}\label{def:upper-bound function for spin average}
    For $I,J\subset \cR, x\in I\cup J$, define  \begin{align*}
        &d_{\cR,I,J}(x)=\min\{dist(x,\intB\cR),~dist(x,\intB\cR'),~{\lfloor\frac{1}{2}\min_{z\in I\cup J\setminus\{x\}}dist(x,z)\rfloor}\}-1,\\\text{and}~&F_{\cR}(I,J)=\prod_{x\in I}(\max\{d_{\cR, I, J}(x),1\})^{-{\frac{1}{8}}}\cdot \prod_{y\in J}(\max\{d_{\cR, I, J}(y),1\})^{-{\frac{1}{8}}}.
    \end{align*} 
\end{defi}
We remark here that we let $d_{\cR,I,J}(x)<\frac{1}{2}\min_{z\in I\cup J\setminus\{x\}}\{dist(x,z)\}$ in order to ensure that $\Lambda_{d(x)}(x)$ and $\Lambda_{d(y)}(y)$ are disjoint for $x\neq y \in I\cup J$ (where we have dropped the subscripts $R,I,J$ in $d(x)$ and $d(y)$ for notation clarity).

\begin{lem}\label{lem: bound for the spin average conditioned on the crossing}
There exist constants $\nc\label{13},\nc\label{15}>0$ such that the following holds. For any integer {$M$},
we have for any $I,J\subset \cR$   \begin{align}
&|\langle\1_{\hc}\prod_{x\in I} \bar\sigma_{x}^+\prod_{y\in J} \bar\sigma_{y}^- \rangle_{\cR,0}^{-/+}|\le \oc{13}^{|I|+|J|} F_{\cR}(I,J).\label{eq: bound for the spin average conditioned on the crossing}    
    \end{align}
    Furthermore, for any integer $m,k\ge 0$ we have \begin{align}
\sum_{\substack{I,J\subset\cR,\\|I|=k,|J|=m}}F_{\cR}(I,J)^2&\le \oc{15}^{k+m}M^{\frac{7(k+m)}{4}}\times\frac{[(k+m)!]^{1/4}}{k!m!},\label{eq: upper-bound for sum of squares of Fs}
\end{align}
and for any $I, J\subset\mcc R$ we have
\begin{align}
\sum_{\substack{U\subset\cR\setminus(I\cup J),\\|U|=r}} F^2_{\mcc R}(I\cup U, J\cup U)&\leq \oc{15}^{r}M^{\frac{3r}{2}}F^2_{\mcc R}(I,J)\times\frac{\prod_{s=1}^r(|I|+|J|+s)^{\frac{1}{2}}}{r!}.\label{eq: upper-bound for sum of of FUs}
    \end{align}
\end{lem}
The proof of Lemma \ref{lem: bound for the spin average conditioned on the crossing}, as presented in Section \ref{Sec:4.3} below, strongly relies on the relation between the extended Ising model and the FK-Ising model.

With Lemma~\ref{lem: bound for the spin average conditioned on the crossing} replacing Lemma~\ref{lem: upper-bound for the sum of squares of k point function}, we obtain the following lemma similar to Lemma~\ref{lem: small perturbation for partition function}.
\begin{defi}\label{def: good external field for crossing event}
    Let $\cH_{\hc}\subset\R^{\mcc R}$ denote the collection of the external field $h$ such that \begin{equation}\label{eq: good external field for crossing event}
        |\langle ~\1_{\hc} \prod_{v\in\cR}\Big([1+\bar\sigma_v^+\tanh(\veps h_v)]\cdot [1+\bar\sigma_v^-\tanh(\veps h_v)]\Big)~\rangle_{\cR,0}^{-/+}-\langle ~\1_{\hc} ~\rangle_{\cR,0}^{-/+}|\le \sqrt{\eps M^{\frac{7}{8}}}.
    \end{equation} 
\end{defi}
\begin{lem}\label{lem: small perturbation for crossing probability}
    There exist constants $\nc\label{16},\nc\label{17}>0$ such that for any disorder strength {$\eps\leq\oc{17}M^{-\frac{7}{8}}$}, we have $$\P(\cH_{\hc})\ge 1-\oc{16}\exp(-\oc{16}^{-1}\sqrt{\eps^{-1}M^{-7/8}}).$$
\end{lem}

Now we are ready for the proof of Theorem~\ref{thm: RSW for dis with external field}, which is at this point very similar to the proof of Theorem~\ref{thm-critical-temperature-small-perturbation}.
\begin{proof}[Proof of Theorem~\ref{thm: RSW for dis with external field}]
    For $h\in \cH^+_*\cap\cH^-_*\cap\cH_{\hc}$, recalling Definitions~\ref{def: good external field for partition function} and \ref{def: good external field for crossing event}, we obtain by \eqref{eq: expectation for dis crossing} that\begin{equation}\label{eq: influence for disorder on crossing probability}
        \bar\mu_{\cR,\eps h}^{-/+}(\hc)\ge \frac{\bar\mu_{\cR,0}^{-/+}(\hc)-\sqrt{\eps M^{\frac{7}{8}}}}{(1+\sqrt{\eps M^{\frac{7}{8}}})^2}.
    \end{equation}
    Let $\sqrt{\oc{7}}<\min\{\oc{6}/2,1\}$ where $\oc{6}$ is the constant defined in Theorem~\ref{thm: RSW for dis at 0 external field}. Combined with Theorem \ref{thm: RSW for dis at 0 external field} and the fact that $\eps M^{\frac{7}{8}}\le \oc{7}$, it yields that $$\bar \mu^{-/+}_{\cR, \epsilon h}(\hc) \geq \oc{6}/8,$$ thereby completing the proof by combining  Lemmas~\ref{lem: small perturbation for partition function} and \ref{lem: small perturbation for crossing probability}. 
\end{proof}
\subsection{Proof of Proposition~\ref{prop:fast power law}}\label{sec:Proof of Proposition 3.10}
In this subsection, we outline the proof for  Proposition~\ref{prop:fast power law}. We first explain the proof strategy in the overview level. On the one hand, we will show (see Proposition~\ref{prop: strictly faster polynomial decay}) by controlling the surface tension (see Definition~\ref{def:surface tension}) that the decay rate of the boundary influence is faster than $m(T_c,N,0)$ (recall that $m(T_c,N,0)$ has the order of $ N^{-\frac{1}{8}}$ from \eqref{eq: origin decay rate without disorder}). 
On the other hand, we will use the fractality of the disagreement crossing to derive a lower bound on the number of disagreements provided that there exists a disagreement crossing in an annulus, and thus combined with Proposition \ref{prop: strictly faster polynomial decay} this yields Proposition \ref{prop:fast power law}. Following this strategy, in Section \ref{sec: upper-bound for main theorem} we will prove Lemma \ref{lem: dis crossing probability upper bound} (below), which is the key input for the proof of Proposition \ref{prop:fast power law}.

We now implement the proof strategy above. We start with the following estimate on the boundary influence. 
\begin{prop}\label{prop: strictly faster polynomial decay}
    There exists an absolute constant $\nc\label{prior}>0$ such that for any disorder strength $\eps$, we have $$m(T_c,N,\eps) \leq  \frac{\oc{prior}}{\eps N} \mbox{ for all }  N\geq 1\,.$$
\end{prop}

In the case that $\eps\gg N^{-\frac{7}{8}}$, Proposition~\ref{prop: strictly faster polynomial decay}  shows that the decay rate of the boundary influence is strictly faster than $m(T_c,N,0)$. The proof of Proposition~\ref{prop: strictly faster polynomial decay} employs the concept of surface tension (see Definition~\ref{def:surface tension}) which connects the influence of boundary condition with the disagreement crossing probability, and the actual proof is carried out in Section \ref{sec: upper-bound for main theorem}. 

Next, we state the aforementioned fractality result, which is similar to \cite[Proposition 3.1]{DX21} and \cite[Theorem 5.5]{AHP20}. 
The proof of this lemma will rely strongly on \cite{AB99} and Lemma~\ref{lem:crossing} below. To meet further needs, we let $M$ be an integer smaller than $N$, whose exact value will be specified later.

\begin{defi}\label{def: dis crossing intersecting a large number of M-boxes}
    For $\alpha >0$ and integers $N\ge M \geq 1$, recall that $\{B_i\}_{i\in I}$ is a partition of $\Lambda_{N}$ into $M$-boxes (without loss of generality we assume $2M$ divides $N$).
    Let $\cE_{\alpha,N,M}$ be the event that there is a path of disagreements intersecting at most $(\frac{N}{M})^{1 + \alpha}$ $M$-boxes in $\{B_i\}_{i\in I}$ and this path joins the inner and outer boundaries of $\Lambda_{N/2, N}$.
\end{defi}

\begin{thm}\label{thm:Fractality}
There exist absolute constants $\alpha,\nc\label{tur} >0$,  such that for any integers $N>M>0$, 
\begin{equation*}\label{fractality}
   \mathbb{E}\left(\langle 1_{\cE_{\alpha,N,M}}\rangle^{+/-,+/-}_{\Lambda_{N/2,N},\eps h}\right) \leq  \oc{tur}^{-1} \exp(-\oc{tur} \sqrt{N/M}). 
\end{equation*}
\end{thm}

By Theorem \ref{thm:Fractality}, once there is a disagreement crossing (say $\mathcal C$) there is a large number of $M$-boxes ($M$ will be chosen later such that $\epsilon M^{7/8}\ll1$) which intersects $\mathcal C$. In light of this, we wish to apply the RSW estimate to obtain with positive probability a pre-disagreement circuit around each $M$-box intersecting $\mathcal C$. 
This relates to the definition of our good boxes.


In order to discuss whether an $M$-box intersected by $\mcc C$ is good or not, we consider the graph consisting of $M$-boxes as well as the configuration of good or bad for these $M$-boxes.
\begin{defi}\label{def: domination}
    Let $G_M$ be the graph where the vertex set is given by the collection of M-boxes $\{B_i\}_{i\in I}$ and two $M$-boxes are neighboring with each other if their $\ell_1$-distance is $1$. For each $\zeta\in \{0, 1\}^I$, we say $h$ is \textbf{consistent} with $\zeta$ if for each $i\in I$ with $\zeta_i = 1$ (respectively $\zeta_i=0$) we have that $h$ is good (respectively bad) with respect to $B_i$. Let  \begin{equation*}
        \begin{aligned}
            \P_{G_M}(\zeta) = \P(\text{h is consistent with } \zeta).
        \end{aligned}
    \end{equation*}
\end{defi}

Applying a result of \cite{LSS97}, we obtain in the following lemma that $\mbb P_{G_M}$ dominates a supercritical Bernoulli percolation since $\mathbb P_{G_M}$ is a very supercritical percolation with finite-range dependence. 
\begin{lem}\label{lem: domination by product measure}
    For any constant $\rho>0$, there exists a constant $\nc\label{dom}=\oc{dom}(\rho)>0$ such that for any integer $M$ and any disorder strength {$\eps\leq\oc{dom}M^{-\frac{7}{8}}$}, we have $\P_{\G_M}$ stochastically dominates $\P_{\rho}$, where $\P_{\rho}$ denotes the Bernoulli site percolation on $\G_M$ with  density $\rho$. 
\end{lem}

In order to derive a lower bound on the number of disagreements provided with the existence of a pre-disagreement crossing, we need to strengthen Theorem~\ref{thm:Fractality} and show that there exists a large number of good $M$-boxes intersected by a pre-disagreement crossing. 
\begin{defi}\label{def: perfect external field}
    We say an external field $h\in \R^{\Lambda_{N/2,N}}$ is \textbf{perfect} if for each crossing in $\Lambda_{N/2,N}$ intersecting at least $(\frac{N}{M})^{1+\alpha}$ $M$-boxes in $\Lambda_{N/2,N}$, this crossing intersects at least $\frac{1}{2}(\frac{N}{M})^{1+\alpha}$ good $M$-boxes in $\Lambda_{N/2,N}$. We denote by $\cH_{\mathtt {perf}}$ the collection of perfect external fields.
\end{defi}
\begin{lem}\label{lem: perfect external field bound}
    There exist constants $\nc\label{perf1},\nc\label{perf2} >0$ such that for any integer $N>0$ and any disorder strength {$\eps\leq\oc{perf1}M^{-\frac{7}{8}}$}, we have \begin{equation*}\label{eq: perfect external field bound}
        \P(\cH_{\mathtt {perf}})\ge 1-\oc{perf2}^{-1}\exp(-\oc{perf2}(\frac{N}{M})^{1+\alpha}).
    \end{equation*}
\end{lem} 
The following is a key input for the proof of Proposition \ref{prop:fast power law}.
\begin{lem}\label{lem: dis crossing probability upper bound}
    There exists a constant $\nc\label{disu}>0$ such that \begin{equation*}\label{eq: dis crossing probability upper bound}
        \E\bar\mu_{\Lambda_{N/2,N},\eps h}^{+/-,+/-}(\con(N/2,N))\le \mathbb{E}\left(\langle 1_{\cE_{\alpha,N,M}}\rangle^{+/-,+/-}_{\Lambda_{N/2,N},\eps h}\right)+\P(\cH_{\mathtt {perf}}^c)+\oc{disu}\frac{1}{\eps M^{\frac{7}{8}}}\cdot(\frac{N}{M})^{-\alpha}.
    \end{equation*}
\end{lem}
Assuming {Theorem \ref{thm:Fractality},} Lemmas~\ref{lem: perfect external field bound} and \ref{lem: dis crossing probability upper bound}, we are able to prove Proposition~\ref{prop:fast power law}.
\begin{proof}[Proof of Proposition~\ref{prop:fast power law}]
    Let $M=\lfloor\oc{perf1}^{{\frac{8}{7}}}\eps^{-\frac{8}{7}}\rfloor$. Recall that {$\eps\geq\oc{fpl2}N^{-\frac{7}{8}}$} and thus we have $\frac{N}{M}\ge \frac{N}{\oc{perf1}^{{8/7}}\eps^{-8/7}}\ge \frac{\oc{fpl2}^{8/7}}{{\oc{perf1}^{8/7}}}$.
    For any $\oc{fpl1}>0$, let $\oc{fpl2}>0$ be big enough such that the following hold.\begin{itemize}
        \item $\oc{perf2}^{-1}\exp\Big(-\oc{perf2}\cdot(\frac{\oc{fpl2}^{8/7}}{{\oc{perf1}^{8/7}}})^{1+\alpha}\Big)<\frac{\oc{fpl1}}{3}$, so that by Lemma \ref{lem: perfect external field bound}, $$\P(\cH_{\mathtt {perf}}^c)\le\oc{perf2}^{-1}\exp(-\oc{perf2}(\frac{N}{M})^{1+\alpha})\le\frac{\oc{fpl1}}{3}.$$
        \item $\frac{\oc{disu}}{\oc{perf1}}\cdot (\frac{\oc{fpl2}^{8/7}}{{\oc{perf1}^{8/7}}})^{-\alpha}\le\frac{\oc{fpl1}}{3}$, so that $\oc{disu}\frac{1}{\eps M^{\frac{7}{8}}}\cdot(\frac{N}{M})^{-\alpha}\le \frac{\oc{fpl1}}{3}.$
        \item $\oc{tur}^{-1}\exp\big(-\oc{tur}\sqrt{\frac{\oc{fpl2}^{8/7}}{{\oc{perf1}^{8/7}}}}\big)\le \frac{\oc{fpl1}}{3}$, so that by Theorem~\ref{thm:Fractality}, $$\mathbb{E}\left(\langle 1_{\cE_{\alpha,N,M}}\rangle^{+/-,+/-}_{\Lambda_{N/2,N},\eps h}\right)\le \oc{tur}^{-1}\exp(-\oc{tur}\sqrt{N/M})\le\frac{\oc{fpl1}}{3}.$$
    \end{itemize}
    Thus we obtain by Lemma~\ref{lem: dis crossing probability upper bound} that \begin{equation*}
\E\bar\mu_{\Lambda_{N/2,N},\eps h}^{+/-,+/-}(\con(N/2,N))\le \frac{\oc{fpl1}}{3}+\frac{\oc{fpl1}}{3}+\frac{\oc{fpl1}}{3}=\oc{fpl1}.
        \qedhere
    \end{equation*}
\end{proof}

\section{RSW for disagreement Ising model}\label{sec: RSW}
In this section, we provide proofs for missing ingredients of RSW estimates. That is, we prove Theorems~\ref{thm: RSW for dis with external field}, \ref{thm: RSW for dis at 0 external field} and Lemmas \ref{lem: bound for the spin average conditioned on the crossing},~\ref{lem: small perturbation for crossing probability}.

\subsection{RSW and a BK-type inequality without disorder}
This subsection will be devoted to the proof of Theorem~\ref{thm: RSW for dis at 0 external field} and a BK-type inequality which will be useful in the proof of Lemma~\ref{lem: bound for the spin average conditioned on the crossing}.

We first give the following definition, which will be helpful for analyzing
 the RSW estimate without disorder.
\begin{defi}\label{def:agreement boundary}
    For any configuration pair $(\bar\sigma^+,\bar\sigma^-)$, we say a nonempty connected region $\Gamma\subset \bar {\mbb Z}^2$ has an \textbf{agreed boundary condition} if $\bar\sigma_x^+=\bar\sigma_x^-$  for all $x\in \intB\Gamma$,  and in this case we say that $\intB \Gamma$ is an \textbf{agreed boundary}.
\end{defi}

Recall that in Definition~\ref{def:connect-event} we denote by $\con(M_1, M_2)$ the event that $\mathcal D \cap \bar\Lambda_{M_2}$ contains a crossing from $\intB \Lambda_{M_2}$ to $\intB \Lambda_{M_1+1}$. {Furthermore, we define $\con_\diamond(M_1, M_2)$ to be the event that there exists an anti-disagreement crossing from $\intB \Lambda_{M_2}$ to $\intB \Lambda_{M_1+1}$.} Then, by planar duality {and the fact that pre-disagreements and anti-disagreements can not be adjacent}, we see that on the complement of {$\con(M_1, M_2)\cup\con_\diamond(M_1, M_2)$}, there exists a region $\bar\Lambda_{M_1} \subset \Gamma \subset \bar\Lambda_{M_2} $ such that $ \Gamma$ has {an} agreed boundary condition.
This then allows us to work with an agreed boundary condition when analyzing spin configurations restricted on $\bar\Lambda_{M_1}$.  
It then naturally calls for an upper bound on the probability of the crossing event, and indeed we next show that the pre-disagreement crossing probability is bounded away from $1$ even in the best boundary condition.
\begin{lem}\label{lem:crossing}
  For any integers $M_2> M_1$ and any external field $h$ on $\Lambda_{M_1, M_2}$, we have
  \begin{equation}\label{eq:lasso estimate}
    \langle \1_{\con(M_1,M_2)} \rangle_{\Lambda_{M_1,M_2},h}^{+/-,+/-}\le\langle \1_{\con(M_1,M_2)} \rangle_{\Lambda_{M_1,M_2},0}^{+/-,+/-}
    = 1-\Big(\frac{1- \phi_{G}^{\mathtt w/\mathtt w}(\intB \Lambda_{M_1+1} \longleftrightarrow \intB\Lambda_{M_2})}{1+ \phi_{G}^{\mathtt w/\mathtt w}(\intB \Lambda_{M_1+1} \longleftrightarrow \intB\Lambda_{M_2})}\Big)^2.
  \end{equation}
  \end{lem}
\begin{rmk}\label{rmk:RSW-ref}
    If  $M_2\ge \iota M_1$ for some constant $\iota>1$, then the right{-}hand side of \eqref{eq:lasso estimate} is upper-bounded by a positive constant smaller than 1  depending only on $\iota$ thanks to the RSW theory of the two-dimensional critical FK-Ising model \cite[Theorem 1.1]{DHN11} (see also e.g. \cite{BDC12, Tassion16, KT23} for remarkable further progress on RSW theory).
\end{rmk}
\begin{rmk}\label{rmk:anti-crossing}
{By symmetry, an analog of Lemma~\ref{lem:crossing} holds for $\con_\diamond(M_1,M_2)$} with boundary condition $-/+,-/+$.
\end{rmk}
\begin{cor}\label{cor: agreed boundary}
    For any integers $M_2\ge \iota M_1$ {with some constant $\iota>1$,
    there exists a constant $\nc\label{agree}=\oc{agree}(\iota)>0$ such that \begin{equation}\label{eq:agree estimate}
    \bar\mu_{\Gamma,0}^{\xi^+/\xi^-} \Big(\con(M_1,M_2)^c\cap\con_\diamond(M_1,M_2)^c\Big) \ge \oc{agree}
  \end{equation}    for any region $\Gamma$ containing $\Lambda_{M_1,M_2}$ with any boundary conditions $\xi^+,\xi^-$ on $\partial\Gamma$.}
\end{cor}
\begin{proof}
   {Let $M'=\lfloor\frac{M_1+M_2}{2}\rfloor.$ Then
    by DMP and CBC, we get that \begin{align*}
        &\bar\mu_{\Gamma,0}^{\xi^+/\xi^-} \Big(\con(M_1,M_2)^c\cap\con_\diamond(M_1,M_2)^c\Big)\\\ge~& \bar\mu_{\Lambda_{M_1,M'},0}^{+/-,+/-} \Big(\con(M_1,M')^c\Big)\cdot\bar\mu_{\Lambda_{M',M_2},0}^{-/+,-/+} \Big(\con_\diamond\big(M',M_2\big)^c\Big).
    \end{align*} The desired result then follows from Lemma~\ref{lem:crossing}, Remarks~\ref{rmk:RSW-ref} and \ref{rmk:anti-crossing}.
    }
\end{proof}
\begin{proof}[Proof of Lemma~\ref{lem:crossing}]
   Denote by $ G$ the graph $\overline{\Lambda_{M_1,M_2}}$ and let $ G^{\prime}= G/{\intB}\Lambda_{M_1+1}$ which is obtained from $ G$ by contracting ${\intB} \Lambda_{M_1+1}$ into a single point (denoted as $\fQ$) and throwing away all edges between the points in ${\intB} \Lambda_{M_1+1}$. 

We compute that
    \begin{equation*}
        \begin{aligned}
            \langle \1_{\con(M_1,M_2)} \rangle_{\Lambda_{M_1,M_2},h}^{+/-,+/-}&= \langle \1_{\con(M_1,M_2)} \rangle_{G,  h}^{+/-, +/-} \stackrel{(*)}{=} \langle \1_{\con(M_1,M_2)} \rangle_{G^{\prime},h^{\prime}}^{+/-, +/-},
        \end{aligned}
    \end{equation*}
    where $h_v^{\prime}=
        h_v$ for  $v \in G^{\prime}\setminus\{\fQ\}$ and $h_\fQ^\prime = g$
         for some $g\in\mbb R$ to be chosen later. Here (*) comes from the fact that once the boundary condition on  ${\intB} \Lambda_{M_1+1}$ is fixed to be plus or minus, wiring ${\intB}\Lambda_{M_1+1}$ into a single point $\fQ$ and posing an arbitrary external field on $\fQ$ will not change the Ising measure.
         For notation convenience, in what follows we use the superscript $\mathbf 0$ to indicate the free boundary condition for the (extended) Ising model. The $\mathbf 0$-boundary condition will only apply to vertex spins, and it means that the corresponding spins can be either plus or minus (i.e., there is no constraint). Using this notation, we see that
         \begin{equation}\label{eq: expansion for annulus boundary}
            \langle \1_{\con(M_1,M_2)} \rangle_{G^{\prime},h^{\prime}}^{+/-, +/-}=\frac{\langle \1_{\{\bar \sigma_{\fQ}^+=1,\bar \sigma_{\fQ}^-=-1\}\cap \con(M_1,M_2)} \rangle_{G^{\prime},h^{\prime}}^{{\mathbf{0}/\mathbf{0},+/-}}}{\langle \1_{\{\bar \sigma_{\fQ}^+=1,\bar \sigma_{\fQ}^-=-1\}} \rangle_{G^{\prime},h^{\prime}}^{\mathbf{0}/\mathbf{0},+/-}}.
        \end{equation}
    On the event $\con(M_1,M_2)^{c}$, we have that $\cD_\fQ$, i.e., the connected component of $\fQ$ in $\mathcal D$, does not intersect $\intB \Lambda_{M_2}$. Applying Proposition~\ref{prop:swap} with $S=\{\fQ\}$ and $A=\intB \Lambda_{M_2}$, we obtain \begin{equation}\label{eq: swapping for crossing probability}
        \langle \1_{\{\bar \sigma_{\fQ}^+=1,\bar \sigma_{\fQ}^-=-1\}\cap \con(M_1,M_2)^{c}} \rangle_{G^{\prime},h^{\prime}}^{{\mathbf{0}/\mathbf{0},+/-}}=\langle \1_{\{\bar \sigma_{\fQ}^+=-1,\bar \sigma_{\fQ}^-=1\}} \rangle_{G^{\prime},h^{\prime}}^{{\mathbf{0}/\mathbf{0},+/-}}\,,
    \end{equation}
    which follows from swapping the configurations on $\cD_{\fQ}$, as in \eqref{eq:ClusterSwap}.
    Combined with \eqref{eq: expansion for annulus boundary}, it yields that \begin{equation}\label{eq:crossing prime -/+}
        \langle \1_{\con(M_1,M_2)^c} \rangle_{G^{\prime},h^{\prime}}^{+/-, +/-}=\frac{\langle \1_{\{\bar \sigma_{\fQ}^+=-1,\bar \sigma_{\fQ}^-=1\}} \rangle_{G^{\prime},h^{\prime}}^{{\mathbf{0}/\mathbf{0},+/-}}}{\langle \1_{\{\bar \sigma_{\fQ}^+=1,\bar \sigma_{\fQ}^-=-1\}} \rangle_{G^{\prime},h^{\prime}}^{{\mathbf{0}/\mathbf{0},+/-}}} =  \frac{\langle \1_{\{\bar \sigma_{\fQ}=-1\}} \rangle_{G^{\prime},h^{\prime}}^{\mathbf{0},+} \cdot \langle \1_{\{\bar \sigma_{\fQ}=1\}} \rangle_{G^{\prime},h^{\prime}}^{\mathbf{0},-}}{\langle \1_{\{\bar \sigma_{\fQ}=1\}} \rangle_{G^{\prime},h^{\prime}}^{\mathbf{0},+} \cdot \langle \1_{\{\bar \sigma_{\fQ}=-1\}} \rangle_{G^{\prime},h^{\prime}}^{\mathbf{0},-}}\,,
    \end{equation}
   where the second equality holds regardless of the choice of $g$ thanks to independence. 
   
{Now for clarity of notation, we define $$p(\xi,h)=\langle \1_{\{\bar \sigma_{\fQ}=1\}} \rangle_{G^{\prime},h}^{\mathbf{0},\xi},\text{ and }m(\xi,h)=\langle \1_{\{\bar \sigma_{\fQ}=-1\}} \rangle_{G^{\prime},h}^{\mathbf{0},\xi}.$$
Here $p$ and $m$ stand for the probability of plus and minus respectively, $\xi\in\{+,-\}$ stands for the boundary condition on $\partial\Lambda_{M_2},$ and $h$ stands for the external field. As a result we have 
\begin{equation}\label{eq:p+m=1}
    p(\xi,h)+m(\xi, h)=1, 
\end{equation}
and the right-hand side of \eqref{eq:crossing prime -/+} can be rewritten as 
\begin{equation}\label{eq:rewrite-by-pm}
    \frac{m(+,h')\cdot p(-,h')}{p(+,h')\cdot m(-,h')}=\frac{p(-,h')}{p(+,h')}\cdot\frac{m(+,h')}{m(-,h')}.
\end{equation}}
   
   Since ${m(+,h')=}\langle \1_{\{\bar \sigma_{\fQ}=-1\}} \rangle_{G^{\prime},h^{\prime}}^{\mathbf{0},+}$ decreases from $1$ to $0$ continuously and  $\langle \1_{\{\bar \sigma_{\fQ}=1\}} \rangle_{G^{\prime},h^{\prime}}^{\mathbf{0},-}$ increases from  $0$ to $1$ continuously as $g$ increases from $-\infty$ to $\infty$, we can choose $g = g_\star$ (and correspondingly $h' = h'_\star$) such that \begin{equation}
       m(+,h'_\star)=\langle \1_{\{\bar \sigma_{\fQ}=-1\}} \rangle_{G^{\prime},h^{\prime}_\star}^{\mathbf{0},+} = \langle \1_{\{\bar \sigma_{\fQ}=1\}} \rangle_{G^{\prime},h^{\prime}_\star}^{\mathbf{0},-}=p(-,h'_\star).\label{eq:choose-g}
   \end{equation} In addition, applying \cite[Theorem 1.1]{DSS22}, we get that 
    \begin{equation}\label{eq-apply-DSS}
        \langle \bar \sigma_{\fQ} \rangle_{G^{\prime},h^{\prime}}^{\mathbf{0},+} - \langle \bar \sigma_{\fQ} \rangle_{G^{\prime},h^{\prime}}^{\mathbf{0},-} \le \langle \bar \sigma_{\fQ} \rangle_{G^{\prime},0}^{\mathbf{0},+} - \langle \bar \sigma_{\fQ} \rangle_{G^{\prime},0}^{\mathbf{0},-} \quad \mbox{ for all } g\in \mathbb R\,,	
    \end{equation}
    which is equivalent to
    \begin{equation*}
        \langle \1_{\{\bar \sigma_{\fQ}=1\}} \rangle_{G^{\prime},h^{\prime}}^{\mathbf{0},+} - \langle \1_{\{\bar \sigma_{\fQ}=1\}} \rangle_{G^{\prime},h^{\prime}}^{\mathbf{0},-} \le \langle \1_{\{\bar \sigma_{\fQ}=1\}} \rangle_{G^{\prime},0}^{\mathbf{0},+} - \langle \1_{\{\bar \sigma_{\fQ}=1\}} \rangle_{G^{\prime},0}^{\mathbf{0},-}\quad \mbox{ for all } g\in \mathbb R\, .
    \end{equation*}
    {
    Taking $g=g_\star$ and thus $h'=h'_\star,$ we have
 $$1-2m(+,h'_\star)\stackrel{\eqref{eq:p+m=1}}{=}p(+,h'_\star)-m(+,h'_\star)\stackrel{\eqref{eq:choose-g}}{=}p(+,h'_\star)-p(-,h'_\star)\leq p(+,0)-p(-,0)=1-2m(+,0).$$
 The last equality comes from the symmetry in the 0-external field case.
 }

{
 Combined with the fact that $\frac{x}{1-x}\ge\frac{y}{1-y}$ if $1-2x\le 1-2y$ and $x,y\in [0,1]$, it yields that
\begin{equation}
    \frac{p(-,h'_\star)}{p(+,h'_\star)}=\frac{m(+,h'_\star)}{1-m(+,h'_\star)}\geq\frac{m(+,0)}{1-m(+,0)}=\frac{p(-,0)}{p(+,0)}.\label{eq: Con lower bound part 1}
\end{equation}}
Similarly we can have 
 {\begin{equation}\label{eq: Con lower bound part 2}
\frac{m(+,h'_\star)}{m(-,h'_\star)}=\frac{p(-,h'_\star)}{1-p(-,h'_\star)}\geq\frac{p(-,0)}{1-p(-,0)}=\frac{m(+,0)}{m(-,0)}.
\end{equation}}
Combining \eqref{eq: Con lower bound part 1}, \eqref{eq: Con lower bound part 2}, \eqref{eq:crossing prime -/+} (where we apply \eqref{eq:crossing prime -/+} to the case of $h' = 0$ and also to the case $h' = h'_\star$) and \eqref{eq:rewrite-by-pm}, we finish the proof of the {inequality} in \eqref{eq:lasso estimate}.

To show the {equality} in \eqref{eq:lasso estimate} it suffices to show (recalling \eqref{eq:crossing prime -/+} for the case of $h'=0$) $$\frac{\langle \1_{\{\bar \sigma_{\fQ}=1\}} \rangle_{G^{\prime},0}^{\mathbf{0},-}}{\langle \1_{\{\bar \sigma_{\fQ}=1\}} \rangle_{G^{\prime},0}^{\mathbf{0},+}} = \frac{1- \phi_{G}^{\mathtt w/\mathtt w}(\intB \Lambda_{M_1+1} \longleftrightarrow \intB\Lambda_{M_2})}{1+ \phi_{G}^{\mathtt w/\mathtt w}(\intB \Lambda_{M_1+1} \longleftrightarrow \intB\Lambda_{M_2})}.$$ To this end, we relate $G^{\prime}$ to $G$ again and obtain that \begin{align*}
  \langle \1_{\{\bar \sigma_{\fQ}=1\}} \rangle_{G^{\prime},0}^{\mathbf{0},-}
            &= \langle \1_{\{\bar\sigma|_{\intB \Lambda_{M_1+1}}=1\}} \mid \bar\sigma|_{\intB \Lambda_{M_1+1}}=1 \text{ or } \bar \sigma|_{\intB \Lambda_{M_1+1}}=-1 \rangle_{G,0}^{\mathbf{0},-}\nonumber\\
            &=\frac{1}{2}\Big(1- \phi_{G,0}^{\mathtt w/\mathtt w}(\intB \Lambda_{M_1+1} \longleftrightarrow \intB\Lambda_{M_2})\Big) \,.\label{eq: con to FK}
    \end{align*}
Similarly, we also have \begin{equation*}
    \langle \1_{\{\bar \sigma_{\fQ}=1\}} \rangle_{G^{\prime},0}^{\mathbf{0},+}
            =\frac{1}{2}\Big(1+ \phi_{G,0}^{\mathtt w/\mathtt w}(\intB \Lambda_{M_1+1} \longleftrightarrow \intB\Lambda_{M_2})\Big).
\end{equation*} Thus we complete the proof of \eqref{eq:lasso estimate}.
\end{proof}

\begin{proof}[Proof of Theorem \ref{thm: RSW for dis at 0 external field}]
    {In this proof, we will drop all the external field in the subscript since it is always $0$.} Recall that $\cR$ and $ \cR'$ denote the rectangle $\cR(\iota M+M,2M)$ and $\cR(\iota M,M)$ respectively. Let $\cR_1=\cR(\iota M+0.2M,0.9M)$ and $\cR_2=\cR(\iota M+0.1M,0.8M)$. 
    We first show that there exists an agreed boundary around $\cR_2$ (we will make this clear later) with positive probability, and then show the pre-disagreement crossing probability given the agreed boundary has a positive lower bound.

    For some $\bar\cR_2\subset\Gamma\subset\bar\cR_1$, we say $\intB \Gamma$ is the \textbf{outmost agreed boundary} in ${\bar\cR}_1\setminus\bar\cR_2$  if $\Gamma$ is simply connected and $\intB \Gamma$ is an agreed boundary, and in addition for any agreed boundary $\intB\Gamma'\subset{\bar\cR}_1\setminus\bar\cR_2$ with $\bar\cR_2\subset\Gamma'$ and any path in $\bar{\mcc R}_1$ from $x\in\intB\Gamma'$ to $y\in \intB\bar{\mcc R}_1$, this path must intersect with $\intB\Gamma$.
    For $\xi\in\{-1,0,1\}^{\intB \Gamma}$, let $\Psi(\Gamma,\xi)$ denote the probability that $\intB \Gamma$ is the outmost agreed boundary in $\bar \cR_1 \setminus \bar \cR_2$ and $\xi$ is the spin configuration on it.  {By Corollary~\ref{cor: agreed boundary}} we get that with positive probability $C_1$, there is an agreed boundary separating $\intB\cR_1$ and $\intB\cR_2$ {and thus there exists a unique outmost agreed boundary} (see Figure~\ref{fig: agreement boundary for RSW in 0 external field} for an illustration). Thus we conclude that \begin{equation}\label{eq: agreement boundary probaility}
        \sum_{\Gamma,\xi}\Psi(\Gamma,\xi)\ge C_1.
    \end{equation} Let $\cA(\Gamma)$ denote the event that there is a pre-disagreement crossing from $\{-\iota M\}\times[-M,M]$ to $\{\iota M\}\times[-M,M]$ in $\Gamma\cap\bar\cR'$. Applying the DMP and using the fact that $\cA(\Gamma)\subset\hc(\iota M,M)$, we obtain that that\begin{equation}\label{eq: DMP for outmost agreement boundary}
        \bar\mu^{\xi^+/\xi^-}_{\cR}(\hc(\iota M,M))\ge\sum_{\Gamma,\xi}\Psi(\Gamma,\xi)\bar\mu^{\xi/\xi}_{\Gamma}(\cA(\Gamma)).
    \end{equation}
    Thus it suffices to prove that \begin{equation}\label{eq: RSW with an agreement boundary in 0 external field case}
        \bar\mu^{\xi/\xi}_{\Gamma}(\cA(\Gamma))\ge C_2
    \end{equation} for some $C_2>0$ that does not depend on $\xi$ {or $\Gamma$}.

    \begin{figure}[htb]
    \centering
    \begin{tikzpicture}    
    \draw (-5,-3)rectangle (5,3);
    \draw (-3.5,-2.3) rectangle (3.5,2.3);
    \draw [dashed] (-4.2,-1.9) rectangle (4.2,1.9);
    \draw [dashed] (-3.7,-1.5) rectangle (3.7,1.5);
    \node [name=start1] at (5,3){ };
    \node [name=end1] at (6,4){$\cR$ };
    \draw [->] (start1)--(end1) ;
    \node [name=start2] at (2,2.2){ };
    \node [name=end2] at (2,2.8){ $\cR'$};
    \node at (0.5,0.5) {pre-disagreement crossing};
    \draw [color=red,very thick] (-3.5,-1.15)--(-3.5,1.15);
    \node [color=red] at (-3.2,1.3){$\cL_1$ };
    \node [color=red] at (3.2,1.3){$\cL_2$ };
    \draw [color=red,very thick] (3.5,-1.15)--(3.5,1.15);
    \draw [->] (start2)--(end2) ;
    \draw (-3.5,0.1) .. controls (1,2) and (2,-1.5) .. (3.5,-1);
    \node [name=start3] at (-4,1.7){ };
    \node [name=end3] at (-4,3.8){outmost agreed boundary };
    \draw [->] (start3)--(end3) ;
    \node at (4.2,1.5){$\cR_1$ };
    \node at (3.8,0){$\cR_2$ };
    \draw (-4.1,1.8) .. controls (-3.5,1) and (-4.5,0) .. (-3.8,-1.7);
    \draw (-3.8,-1.7) .. controls (-2,-1.3) and (1,-2) .. (3.95,-1.7);
    \draw (3.95,-1.7) .. controls (4.4,-0.5) and (4,1) .. (3.95,1.85);
    \draw (3.95,1.85) .. controls (0,1.6) .. (-4.1,1.8);
    \end{tikzpicture}
    \caption{Outmost agreed boundary around $\cR_2$}
    \label{fig: agreement boundary for RSW in 0 external field}
\end{figure}
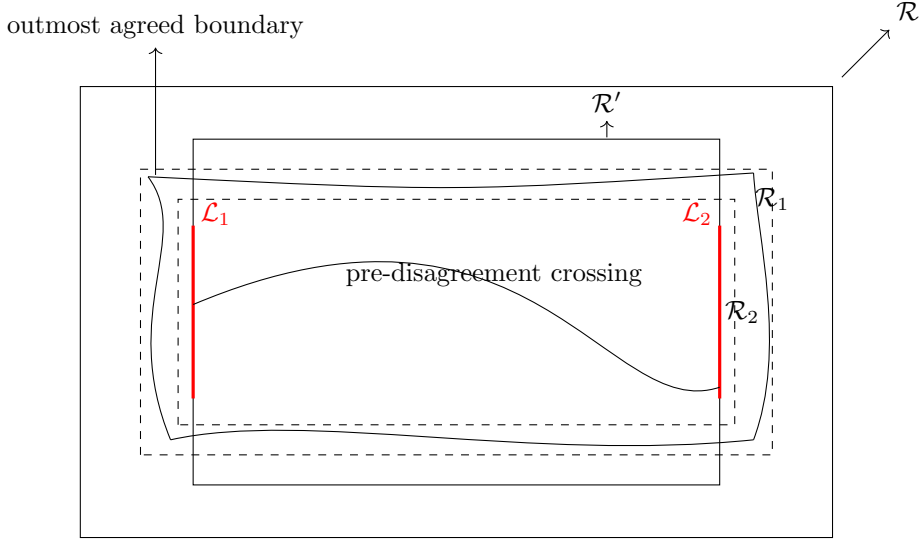

    Let $\cL_1$, $\cL_2$ denote the set of vertices in $\{-\iota M\}\times [-0.5M, 0.5M]$ and $ \{\iota M\}\times [-0.5M, 0.5M]$ respectively. Let $\rc_{u,v}$(for $u\in \mathcal L_1$ and $v\in \mathcal L_2$) denote the event that $u$ and $v$ is connected by a pre-disagreement path in $\Gamma$. Note that we do not restrict the pre-disagreement path to be contained in $\bar\cR'$; this is not a problem since whenever $\rc_{u,v}$ happens, there must be a pre-disagreement path from $\{-\iota M\}\times[-M,M]$ to $\{\iota M\}\times[-M,M]$ in $\Gamma\cap\bar\cR'$. Then we obtain that \begin{equation}\label{eq: crossing event from line to line}
        \bar\mu^{\xi/\xi}_{\Gamma}(\cA(\Gamma))\ge\bar\mu^{\xi/\xi}_{\Gamma}(\sum_{\substack{u\in \cL_1, v\in \cL_2}}\1_{\rc_{u,v}}>0).
    \end{equation}
    We will compute the first and second moments and then apply the Paley-Zygmund inequality to lower-bound the right-hand side of \eqref{eq: crossing event from line to line}. To lower-bound the first moment, we apply \eqref{eq: disagreement-spin average expansion with agreement boundary} to obtain that for $u\in \cL_1,~v\in \cL_2$,\begin{equation}\label{eq: crossing to truncated correlation function}
\langle\1_{\rc_{u,v}}\rangle^{\xi/\xi}_{\Gamma}=\langle\sigma_u\sigma_v\rangle^{\xi}_{\Gamma}-\langle\sigma_u\rangle^{\xi}_{\Gamma}\cdot\langle\sigma_v\rangle^{\xi}_{\Gamma}.
    \end{equation}
Note that in \eqref{eq: crossing to truncated correlation function}, we have used that with an agreed boundary condition,  by symmetry, the probability $u,v$ is connected by pre-disagreements equals to the  probability that $u,v$ is connected by anti-disagreements.
Applying Lemma~\ref{lem: truncated correlation lower-bound uniformly for boundary condition}, we obtain 
    \begin{align}
        \sum_{\substack{u\in \cL_1, v\in \cL_2}}\langle\1_{\rc_{u,v}}\rangle^{\xi/\xi}_{\Gamma}&=
    \sum_{\substack{u\in \cL_1, v\in \cL_2}}\Big(\langle\sigma_u\sigma_v\rangle^{\xi}_{\Gamma}-\langle\sigma_u\rangle^{\xi}_{\Gamma}\cdot\langle\sigma_v\rangle^{\xi}_{\Gamma}\Big)\nonumber\\&\ge\sum_{\substack{u\in \cL_1, v\in \cL_2}}\oc{A4} M^{-\frac{1}{4}}=\oc{A4}(M+1)^2M^{-\frac{1}{4}}\geq \oc{A4}M^{\frac{7}{4}}.\label{eq: first moment lower-bound for pre-disagreement crossing}
    \end{align}

Next, we will upper-bound the second moment. For any $u_1\neq u_2\in \cL_1,~v_1\neq v_2\in \cL_2$, we apply \eqref{eq: disagreement-spin average expansion with agreement boundary} again (note that $\rc_{u_1,v_1}\cap\rc_{u_2,v_2}\subset\cB$ where $\mathcal B$ is defined in Proposition \ref{prop:DisagreementRep} with $V_0 = \{u_1,u_2,v_1,v_2\}$) and obtain that \begin{align}\label{eq: second moment expansion}
\langle\1_{\rc_{u_1,v_1}}\1_{\rc_{u_2,v_2}}\rangle^{\xi/\xi}_{\Gamma}&\le 2^{-4}\langle\prod_{i=1}^2(\sigma_{u_i}^+-\sigma_{u_i}^-)(\sigma_{v_i}^+-\sigma_{v_i}^-)\rangle^{\xi/\xi}_{\Gamma}\nonumber\\
&=2^{-4}\sum_{\substack{I\cup J=\{u_1,u_2,v_1,v_2\}\\I\cap J=\emptyset}}(-1)^{|I|}\langle\sigma^I\rangle^{\xi}_{\Gamma}\langle\sigma^{J}\rangle^{\xi}_{\Gamma}.
    \end{align}
To compute the right-hand side of \eqref{eq: second moment expansion}, we define $$d_u=\min\{0.3M,\lfloor\frac{dist(u_1,u_2)}{2}\rfloor\},\quad d_v=\min\{0.3M,\lfloor\frac{dist(v_1,v_2)}{2}\rfloor\};$$ 
 $$\quad  \cR_{u_1}=\Lambda_{d_u}(u_1), \quad \cR_{u_2}=\Lambda_{d_u}(u_2),\quad \cR_{v_1}=\Lambda_{d_v}(v_1),\quad \cR_{v_2}=\Lambda_{d_v}(v_2).$$
  Recall the coupling in Lemma \ref{lem:extended to FK}, the extended Ising configuration can be understood as an FK-configuration with a random signal assigned to each FK-cluster.
For any $I\subset\{u_1,u_2,v_1,v_2\}$
and FK-configuration $\omega$ on ${E}(\Gamma)$, if there exists $x\in I$ such that $\cC_x\cap \intB\Gamma=\emptyset$ and $\cC_x\cap I=\{x\}$ (here we recall that $\mathcal C_x$ is the FK-cluster containing $x$), then flipping the signal of $\cC_x$ changes the value of $\sigma^I$ and thus the configuration $\omega$ contributes nothing to the expectation of $\sigma^I$.
Then we obtain that for any $I\subset\{u_1,u_2,v_1,v_2\}$
\begin{equation*}
    |\langle\sigma^I\rangle^{\xi}_{\Gamma}|\le \phi^\xi_{\Gamma} \left (\bigcap_{x\in I}\left\{x\longleftrightarrow \intB\Gamma\cup (I\setminus\{x\})\right\}\right )\le \phi^\xi_{\Gamma} \left (\bigcap_{x\in I}\left\{x\longleftrightarrow \intB \mcc R_x\right\}\right ),
\end{equation*}
where the last inequality follows from the choice of $\mcc R_x$. Recalling Lemma~\ref{lem:extended to FK}, 
we see that
the event $\{x\longleftrightarrow \intB\mcc R_x\}$ in the FK Ising model is equivalent to the event $\mathtt {Con}^+(x,\mcc R_x)\cup\mathtt{Con}^-(x,\mcc R_x)$ in the extended Ising model, where $\mathtt {Con}^+(x,\mcc R_x)$ (respectively,  $\mathtt {Con}^-(x,\mcc R_x)$) is the event that there exists a path connecting $x$ and $\intB \mcc R_x$ on which all the vertex spins and edge spins are plus (respectively, minus).
It is obvious that $\con^+(x, \mathcal R_x)$ (respectively, $\con^-(x, \mathcal R_x)$)  is an increasing (respectively, decreasing) event under the extended Ising measure. For $\zeta = (\zeta_x)_{x\in I}$ with $\zeta_x \in \{-1, 0, 1\}^{\partial \mathcal R_x}$, define 
 $\Psi(\zeta)=\bar\mu^{\xi}_{\Gamma}(\bar\sigma|_{\intB\mcc R_x}=\zeta_x\mbox{ for } x\in I)$. Then we have {(recall our notations in Remark~\ref{rmk:connecting events for extended Ising})}
 \begin{align}
\bar\mu_{\Gamma}^{\xi}\left(\bigcap_{x\in I}\left\{x\longleftrightarrow \intB \mcc R_x\right\}\right )\stackrel{\mbox{DMP}}=&\sum_{\zeta}\Psi(\zeta)\prod_{x\in I}\bar\mu_{{\mcc R_x}}^{\zeta_x}\big(\mathtt {Con}^+(x,\mcc R_x)\cup\mathtt{Con}^-(x,\mcc R_x)\big)\nonumber\\
\stackrel{\mbox{CBC}}\leq&\sum_{\zeta}\Psi(\zeta)\prod_{x\in I}\big[\bar\mu_{{\mcc R_x}}^{+}(\mathtt {Con}^+(x,\mcc R_x))+\bar\mu_{{\mcc R_x}}^{-}(\mathtt{Con}^-(x,\mcc R_x))\big]\nonumber\\ 
\stackrel{(*)}=~&\prod_{x\in I}\big[\bar\mu_{{\mcc R_x}}^{+}(\mathtt {Con}^+(x,\mcc R_x))+\bar\mu_{{\mcc R_x}}^{-}(\mathtt{Con}^-(x,\mcc R_x))\big]\nonumber\\
\stackrel{(**)}= \;&2^{{|I|}}\prod_{x\in I}\phi^{\mathrm w}_{\mcc R_x}(x\longleftrightarrow \intB \mcc R_x),\label{eq: multi-point Fk upper-bound}
 \end{align}
where (*) follows from the law of total probability and (**) follows from Lemma~\ref{lem:extended to FK}.
Combined with \eqref{eq: second moment expansion}and  \eqref{eq: FK one arm event exponent}, it yields that
\begin{equation*}\label{eq: second moment upper-bound for pre-disagreement crossing1}
    \begin{aligned}
        &\sum_{\substack{u_1\neq u_2\in \cL_1\\v_1\neq v_2\in \cL_2}}\langle\1_{\rc_{u_1,v_1}}\1_{\rc_{u_2,v_2}}\rangle^{\xi/\xi}_{{\Gamma}}\\\le~& 
     \sum_{\substack{u_1\neq u_2\in \cL_1\\v_1\neq v_2\in \cL_2}}\sum_{\substack{I\cup J=\{u_1,u_2,v_1,v_2\}\\I\cap J=\emptyset}}\prod_{x\in I}\phi^{\mathrm w}_{\cR_x}(x\longleftrightarrow\intB\cR_x)\prod_{x\in J}\phi^{\mathrm w}_{\cR_x}(x\longleftrightarrow\intB\cR_x)\\\stackrel{\eqref{eq: FK one arm event exponent}}{\le}&
     \sum_{\substack{u_1\neq u_2\in \cL_1\\v_1\neq v_2\in \cL_2}}C_4d_u^{-\frac{1}{4}}\cdot d_v^{-\frac{1}{4}}\le \sum_{\substack{u_1\neq u_2\in \cL_1\\v_1\neq v_2\in \cL_2}}C_5[dist(u_1,u_2)]^{-\frac{1}{4}}\cdot [dist(v_1,v_2)]^{-\frac{1}{4}}
     \\\le~& C_6M^{\frac{7}{2}}.
    \end{aligned}
\end{equation*}
Hence we obtain the following upper bound on the second moment: \begin{equation}\label{eq: second moment upper-bound for pre-disagreement crossing2}
    \begin{aligned}
        &\sum_{\substack{u_1, u_2\in \cL_1\\v_1, v_2\in \cL_2}}\langle\1_{\rc_{u_1,v_1}}\1_{\rc_{u_2,v_2}}\rangle^{\xi/\xi}_{{\Gamma}}\\\le~&
     \sum_{\substack{u_1\neq u_2\in \cL_1\\v_1\neq v_2\in \cL_2}}\langle\1_{\rc_{u_1,v_1}}\1_{\rc_{u_2,v_2}}\rangle^{\xi/\xi}_{{\Gamma}}+C_7M^3
     \le C_8M^{\frac{7}{2}}.
    \end{aligned}
\end{equation}
Combining \eqref{eq: crossing event from line to line}, \eqref{eq: first moment lower-bound for pre-disagreement crossing}, \eqref{eq: second moment upper-bound for pre-disagreement crossing2} and applying the Paley-Zygmund inequality, we obtain that \begin{equation*}
    \bar\mu^{\xi/\xi}_{\Gamma}(\cA(\Gamma))\ge \frac{\Big[\sum_{u\in \cL_1, v\in \cL_2}
\langle\1_{\rc_{u,v}}\rangle^{\xi/\xi}_{{\Gamma}}\Big]^2}{\sum_{\substack{u_1, u_2\in \cL_1\\v_1, v_2\in \cL_2}}\langle\1_{\rc_{u_1,v_1}}\1_{\rc_{u_2,v_2}}\rangle^{\xi/\xi}_{{\Gamma}}}\ge C_9.
\end{equation*}
Thus we finish the proof of \eqref{eq: RSW with an agreement boundary in 0 external field case} and the desired result follows from \eqref{eq: agreement boundary probaility} and \eqref{eq: DMP for outmost agreement boundary}.
\end{proof}

We now show the aforementioned BK-type inequality.

\begin{defi}\label{def: BK event}
 Let $\cont(M_1,M_2)$ be the event that there exist two pre-disagreement crossings in the annulus $\overline{\Lambda_{M_1,M_2}}$ such that their pre-disagreement clusters (restricted in $\overline{\Lambda_{M_1,M_2}}$) are disjoint.
 {We say a circuit $\gamma\subset\overline{\Lambda_{M_1,M_2}}$ is a \textbf{connected circuit} if it is connected. Let $\con_{\star}(M_1,M_2)$ be the event that there does not exist a connected circuit $\gamma$ in $\Lambda_{M_1,M_2}$ such that $\gamma\subset\cD\cap\overline{\Lambda_{M_1,M_2}}$.} 
\end{defi}
\begin{lem}\label{lem: BK-type inequality for disagreement Ising model}
There exist constants $\nc\label{14},\delta>0$ such that for any $M_1<M_2$, any
boundary conditions $\xi^+,\xi^-$ on $\intB \Lambda_{M_1+1}$ and any boundary conditions $\zeta^+,\zeta^-$ on $ \intB \Lambda_{M_2}$, \begin{equation*}\label{eq: BK-type inequality for disagreement Ising model}
        \bar\mu^{\xi^+/\xi^-,\zeta^+/\zeta^-}_{\Lambda_{M_1,M_2},0}(\cont(M_1,M_2))
        \le \oc{14}(\frac{M_2}{M_1})^{-\frac{1}{4}-\delta}.
    \end{equation*}
\end{lem}
\begin{proof}
    List $\intB M_2$ as $\{x_1,x_2,\cdots,x_n\}$. 
    For each $x$, if $x$ is not a pre-disagreement, then let $\mathbf C_x = \emptyset$ and $ \mathbf B_x=\{x\}$; if $x$ is a pre-disagreement, then let $\mathbf C_x$ be the cluster of pre-disagreement (restricted in $\overline{\Lambda_{M_1,M_2}}$) that contains $x$, and in addition define
    $$\mathbf B_x=\extB\mathbf C_x\cap\overline{\Lambda_{M_1,M_2}},$$
    where $\extB\mathbf C_x=\{e={\{x,y\}} \notin \mathbf C_x:x\in \mathbf C_x,\text{ or } y\in\mathbf C_x\}\cup\{x\not\in \mathbf C_x:\exists\, {\{x,y\}}\in \mathbf C_x\}$ is the exterior boundary of $\mathbf C_x.$
    


    Now we can check $\mathbf C_{x_1},\mathbf C_{x_2}, \cdots$ in order, and assume that $\mathbf C_{x_\tau}$ is the first crossing cluster we have encountered. (If there is no crossing cluster, we set $\tau=\infty$.) So $\mcc \con(M_1,M_2)$ is equivalent to the event that $\{\tau\leq n\}$.
    Define
    $$\mathbf B=\bigcup_{i=1}^{{\min\{\tau,n\}}} \mathbf B_{x_i}, \quad  \mathbf C=\bigcup_{i=1}^{{\min\{\tau,n\}}} \mathbf C_{x_i}.$$ Note that $\extB \mathbf C\cap\overline{\Lambda_{M_1,M_2}}\subset\mathbf B$ but they may not equal {(this comes from the cases that $\mathbf B_x=\{x\}$). So we know the points in $\mathbf B$ could be agreements or anti-disagreements.}
    By definition of $\con$, we have 
$$\con(M_1,M_2)=\bigcup\limits_{\mathsf B,\mathsf C\subset \overline{\Lambda_{M_1,M_2}}} \{\mathbf B=\mathsf B,  \mathbf C=\mathsf C, \tau\leq n\}.$$ 
Here we used a new font for $ \mathsf B$ and $\mathsf C$ in order to indicate that they are deterministic sets.
We point out that 
    in the union above, some of the events of the form
     $\{\mathbf B=\mathsf B,  \mathbf C=\mathsf C, \tau\leq n\}$ are empty. In light of this, we say a pair of deterministic sets $(\mathsf B,\mathsf C)$ is valid if the event $\{\mathbf B=\mathsf B,  \mathbf C=\mathsf C, \tau\leq n\}$ is not empty, {and there does not exist a connected circuit $\gamma$ in $\mathsf C$ that separates the inner and outer boundary of $\Lambda_{M_1,M_2}$.}
    In addition, recall $\mcc D$ is the set of pre-disagreements, and we see that for each valid pair $(\mathsf B, \mathsf C)$, the event
$\{\mathbf B=\mathsf B,  \mathbf C=\mathsf C, \tau\leq n\}$ is equivalent to $\{\mathsf B\subset\mcc D^c, \mathsf C\subset \mcc D\}$. Moreover, since $\mathsf C$ contains a pre-disagreement crossing between the inner and outer boundary of $\Lambda_{M_1,M_2}$, we have that each connected circuit $\gamma\subset\overline{\Lambda_{M_1,M_2}}$ must intersect with $\mathsf C$ and thus is contained in $\mathsf C$. As a result, we conclude that 
\begin{equation}
  \con(M_1,M_2)\cap\con_{\star}(M_1,M_2)=\bigcup_{\text{valid} (\mathsf B, \mathsf C)} \{\mathsf B\subset\mcc D^c, \mathsf C\subset \mcc D\}.\label{decompose-into-valid-events}  
\end{equation}

By definition, $\cont(M_1,M_2)\subset\con(M_1,M_2)\cap\con_{\star}(M_1,M_2)$, so we have  
    \begin{align}
    &\bar\mu^{\xi^+/\xi^-,{\zeta^+/\zeta^-}}_{\Lambda_{M_1,M_2},0}(\cont(M_1,M_2))\nonumber\\
    =&\sum_{\text{ valid } (\mathsf B,\mathsf C)}\bar\mu^{\xi^+/\xi^-,{\zeta^+/\zeta^-}}_{\Lambda_{M_1,M_2},0}(\cont(M_1,M_2)\mid \mathsf B\subset\mcc D^c, \mathsf C\subset \mcc D)\cdot\bar\mu^{\xi^+/\xi^-,{\zeta^+/\zeta^-}}_{\Lambda_{M_1,M_2},0}(\mathsf B\subset\mcc D^c, \mathsf C\subset \mcc D).
    \label{decomposition-of-exploration}\end{align}
    
For any valid pair $(\mathsf B, \mathsf C)$
and any $\nu^+,\nu^-\in\{-1,0,1\}^\mathsf B$,
define $$\Psi_\mathsf B(\nu^+,\nu^-)=\bar\mu^{\xi^+/\xi^-,\zeta^+/\zeta^-}_{\Lambda_{M_1,M_2},0}(\bar\sigma^+|_\mathsf B=\nu^+,\bar\sigma^-|_\mathsf B=\nu^-~\Big|~ \mathsf B\subset\mcc D^c, \mathsf C\subset \mcc D).$$
    Thus, by the DMP and CBC (as in Lemma~\ref{lem:CBC}), we have 
    \begin{align}
&\bar\mu^{\xi^+/\xi^-,\zeta^+/\zeta^-}_{\Lambda_{M_1,M_2},0}(\cont(M_1,M_2)\mid\mathsf B\subset\mcc D^c, \mathsf C\subset \mcc D)\nonumber\\
    \stackrel{\mbox{DMP}}{=}&\sum_{\nu^+,\nu^-}\Psi_\mathsf B(\nu^+,\nu^-)\bar\mu^{\xi^+/\xi^-,\zeta^+/\zeta^-}_{\Lambda_{M_1,M_2},0}(\con(M_1,M_2,\overline{\Lambda_{M_1,M_2}}\setminus(\mathsf B\cup \mathsf C))~\Big|~ \bar\sigma^+|_\mathsf B=\nu^+, \bar\sigma^-|_\mathsf B=\nu^-)\nonumber\\
\stackrel{\mbox{CBC}}{\le}&\sum_{\nu^+,\nu^-}\Psi_\mathsf B(\nu^+,\nu^-)\bar\mu^{+/-,+/-}_{\Lambda_{M_1,M_2},0}(\con(M_1,M_2)~\Big|~ \bar\sigma^+|_\mathsf B=\bar\sigma^-|_\mathsf B=\nu^+),\label{eq:DMP-of-BK}
    \end{align}
where $\con(M_1,M_2,\overline{\Lambda_{M_1,M_2}}\setminus( \mathsf B\cup \mathsf C))$  is the event that there is a pre-disagreement crossing connecting $\intB\Lambda_{M_1+1}$ and $\intB\Lambda_{M_2}$ in the region $\overline{\Lambda_{M_1,M_2}}\setminus( \mathsf B\cup \mathsf C)$, and the last inequality follows from CBC (note that $(\nu^+, \nu^+) \succeq (\nu^+, \nu^-)$ since $\mathsf B \subset \mathcal D^c$) and the inclusion relation of the two events under consideration.  
Now we introduce an external field $\{\eta_x\}_{x\in\Lambda_{M_1,M_2}}$ with 
$$\eta_x=\left\{
\begin{aligned}
    &+\infty,\quad \text{ if } x\in\mathsf B\text{ and }\nu^{+}_x=1;\\
    &-\infty,\quad \text{ if } x\in\mathsf B\text{ and }\nu_x^{+}=-1;\\
    &0,~~~~~~\quad \text{otherwise}.
\end{aligned}
\right.$$
Then we have 
\begin{align*}
     \bar\mu^{+/-,+/-}_{\Lambda_{M_1,M_2},0}(\con(M_1,M_2)\mid \bar\sigma^+|_\mathsf B=\bar\sigma^-|_\mathsf B=\nu^{+})&=
    \bar\mu^{+/-,+/-}_{\Lambda_{M_1,M_2},\eta}(\con(M_1,M_2))\\&\le \bar\mu^{+/-,+/-}_{\Lambda_{M_1,M_2},0}(\con(M_1,M_2))
\end{align*}from Lemma~\ref{lem:crossing}.
   Plugging it into \eqref{eq:DMP-of-BK} and recalling
    \eqref{decompose-into-valid-events} and \eqref{decomposition-of-exploration},
we obtain that \begin{align}
    &\bar\mu^{\xi^+/\xi^-,{\zeta^+/\zeta^-}}_{\Lambda_{M_1,M_2},0}(\cont(M_1,M_2))\nonumber\\\le~& \bar\mu^{\xi^+/\xi^-,{\zeta^+/\zeta^-}}_{\Lambda_{M_1,M_2},0}(\con(M_1,M_2)\cap\con_{\star}(M_1,M_2))\times\bar\mu^{+/-,+/-}_{\Lambda_{M_1,M_2},0}(\con(M_1,M_2)).\label{eq:BK1}
\end{align} Since $\con(M_1,M_2)$ is an increasing event and $\con_{\star}(M_1,M_2)$ is an decreasing event, we obtain from FKG property(Lemma~\ref{lem:FKG}) that \begin{align}
    &\bar\mu^{\xi^+/\xi^-,{\zeta^+/\zeta^-}}_{\Lambda_{M_1,M_2},0}(\con(M_1,M_2)\cap\con_{\star}(M_1,M_2))\nonumber\\\le~& \bar\mu^{\xi^+/\xi^-,{\zeta^+/\zeta^-}}_{\Lambda_{M_1,M_2},0}(\con(M_1,M_2))\times\bar\mu^{\xi^+/\xi^-,{\zeta^+/\zeta^-}}_{\Lambda_{M_1,M_2},0}(\con_{\star}(M_1,M_2)).\label{eq:BK2}
\end{align} Combining \eqref{eq:BK1} and \eqref{eq:BK2} with Lemma~\ref{lem:crossing}, it suffices to show \begin{equation}\label{eq:BK4}
    \bar\mu^{\xi^+/\xi^-,{\zeta^+/\zeta^-}}_{\Lambda_{M_1,M_2},0}(\con_{\star}(M_1,M_2))\le (\frac{M_2}{M_1})^{-\delta}
\end{equation} for some $\delta>0$. Let $K=\lfloor\log_2(\frac{M_2}{M_1})\rfloor$ and $m_k=M_12^k(0\le k\le K)$. Noting that $\con_{\star}(M_1,M_2)\subset\cap_{k=0}^{K-1}\con_{\star}(m_k,m_{k+1})$ and applying DMP and CBC, we obtain that \begin{equation}\label{eq:BK3}
    \bar\mu^{\xi^+/\xi^-,{\zeta^+/\zeta^-}}_{\Lambda_{M_1,M_2},0}(\con_{\star}(M_1,M_2))\le \prod_{k=0}^{K-1}\bar\mu^{-/+,{-/+}}_{\Lambda_{m_{k},m_{k+1}},0}(\con_{\star}(m_{k},m_{k+1})).
\end{equation} Applying Theorem \ref{thm: RSW for dis at 0 external field} four times for $\iota=4$ (see Figure~\ref{fig:good-box} for an illustration of concatenating four crossings into a contour), we obtain that \begin{equation*}
    \bar\mu^{-/+,{-/+}}_{\Lambda_{m_{k},m_{k+1}},0}(\con_{\star}(m_{k},m_{k+1}))\le C_1
\end{equation*} for some $0<C_1<1$. Combined with \eqref{eq:BK3}, it completes the proof of \eqref{eq:BK4}.
\end{proof}

\subsection{Proof of Lemma~\ref{lem: small perturbation for crossing probability}}
In this subsection, we give the proof of Lemma~\ref{lem: small perturbation for crossing probability} assuming Lemma~\ref{lem: bound for the spin average conditioned on the crossing}.
\begin{proof}[Proof of Lemma~\ref{lem: small perturbation for crossing probability}]
We expand the left-hand side of \eqref{eq: good external field for crossing event} as follows:
\begin{equation*}
    \begin{aligned}
            \langle ~\1_{\hc}\prod_{v\in\cR}[1+\bar\sigma_v^+\tanh(\veps h_v)]\cdot [1+\bar\sigma_v^-\tanh(\veps h_v)] ~\rangle_{\cR,0}^{-/+}= \langle ~\1_{\hc} ~\rangle_{\cR,0}^{-/+}+\sum_{k+m\ge 1} \Phi_{k,m}(h),
    \end{aligned}
\end{equation*} 
where in the sum $k$ and $m$ are non-negative integers and 
$$\Phi_{k,m}(h)= \sum_{\substack{ I,J\subset\cR \\|I|=k, |J|=m}}\langle\1_{\hc}\prod_{x\in I}\bar\sigma_{x}^+\prod_{y\in J}\bar\sigma_{y}^- \rangle_{\cR,0}^{-/+} \prod_{x\in I}\tanh(\veps h_{x})\prod_{y\in J}\tanh(\veps h_{y}).$$ 
Let $a_{I,J}=\langle\1_{\hc}\prod_{x\in I}\bar\sigma_{x}^+\prod_{y\in J}\bar\sigma_{y}^- \rangle_{\cR,0}^{-/+}$ and let $A_r$ be defined as in Lemma \ref{lem: concentration for sum of products of tanh2}. Applying Lemma \ref{lem: bound for the spin average conditioned on the crossing}, we obtain \begin{align}
A_r^2&~=C_1^{r}M^{2r}\sum_{|I|=k,|J|=m,|I\cap J|\ge r}|a_{I,J}|^2\stackrel{\eqref{eq: bound for the spin average conditioned on the crossing}}\le C_2^{m+k}M^{2r}\sum_{|I|=k,|J|=m,|I\cap J|\ge r}F_{\cR}(I,J)^2\nonumber\\&\le C_2^{m+k}M^{2r} \sum_{\substack{I,J\subset\cR,\\|I|=k-r,|J|=m-r}}\sum_{\substack{U\subset\cR\setminus(I\cup J),\\|U|=r}}F^2_{\cR}(I\cup U,J\cup U)\nonumber\\&\stackrel{\eqref{eq: upper-bound for sum of of FUs}}\le C_2^{m+k}M^{2r} \sum_{\substack{I,J\subset\cR,\\|I|=k-r,|J|=m-r}}\oc{15}^rM^{\frac{3r}{2}}F^2_{\cR}(I,J)\times\frac{[(m+k-2r+1)\cdots(m+k-r)]^{\frac{1}{2}}}{r!}\nonumber\\&\stackrel{\eqref{eq: upper-bound for sum of squares of Fs}}\le C_3^{m+k} M^{\frac{7(m+k)}{4}}\times\frac{[(m+k-2r)!]^{\frac{1}{4}}}{(m-r)!(k-r)!}\times\frac{[(m+k-2r+1)\cdots(m+k-r)]^{\frac{1}{2}}}{r!}.\label{eq:4}
\end{align}
In order to upper-bound the term on the right-hand side above, we calculate that 
\begin{align}
&\frac{[(m+k-2r)!]^{\frac{1}{4}}}{(m-r)!(k-r)!}\times\frac{[(m+k-2r+1)\cdots(m+k-r)]^{\frac{1}{2}}}{r!}\le \Big(\frac{\binom{m+k-r}{m-r,k-r,r}}{(m-r)!(k-r)!r!}\Big)^{\frac{1}{2}}\nonumber\\&\le \frac{3^{\frac{m+k-r}{2}}}{[(\frac{m+k}{4})!]^{\frac{1}{2}}}\le \frac{3^{\frac{m+k}{2}}}{(\frac{m+k}{8})^{\frac{m+k}{16}}}=\Big[\frac{3^{\frac{1}{4}}}{(\frac{m+k}{8})^{\frac{1}{32}}}\Big]^{2m+2k}.\label{eq:bound-of-km-factorial1}\end{align}
Combining \eqref{eq:4} and \eqref{eq:bound-of-km-factorial1}, we obtain that $$\max_{0\le r\le \min\{m,k\}} A_r\le C_3^{\frac{m+k}{2}} M^{\frac{7(m+k)}{8}}\times\Big[\frac{3^{\frac{1}{4}}}{(\frac{m+k}{8})^{\frac{1}{32}}}\Big]^{m+k}.$$
Combined with Lemma \ref{lem: concentration for sum of products of tanh2},
it yields that \begin{equation}
  \P(|\Phi_{k,m}(h)|>(\eps M^{\frac{7}{8}})^{(m+k)/2}/3)\le C_4^{-1}\exp\left(-C_4\sqrt{\eps^{-1} M^{-\frac{7}{8}}}\times\frac{(\frac{m+k}{8})^{\frac{1}{32}}}{3^{\frac{1}{4}}}\right ).  \label{eq:bound-of-Phi_mk}
\end{equation} 
Note that $\sum_{m+k\ge 1}^{\infty}(\eps M^{\frac{7}{8}})^{(m+k)/2}\le (2+C_5)\sqrt{\eps M^{\frac{7}{8}}}$ for some constant $C_5>0$ relying on $\eps M^{\frac{7}{8}}$, and we can assume $C_5<1$ by choosing $\oc{17}>0$ small enough.
Then we conclude that 
\begin{align*}
    &\P(|\sum_{m+k\ge 1}^{\infty}\Phi_{k,m}{(h)}| > \sqrt{\eps M^{\frac{7}{8}}})
    \le \sum_{m+k\ge 1}^{\infty}\P(|\Phi_{k,m}(h)|>(\eps M^{\frac{7}{8}})^{(m+k)/2}/3)\\ \le& \sum_{m+k\ge 1}^{\infty}C_4^{-1}\exp\left(-C_4\sqrt{\eps^{-1} M^{-\frac{7}{8}}}\times
    \frac{(\frac{m+k}{8})^{\frac{1}{32}}}{3^{\frac{1}{4}}}\right)\le C_6^{-1}\exp(-C_6\sqrt{\eps^{-1} M^{-\frac{7}{8}}}).\qedhere
\end{align*} 
\end{proof}

\subsection{Proof of Lemma~\ref{lem: bound for the spin average conditioned on the crossing}}\label{sec: bound for the spin average conditioned on the crossing}\label{Sec:4.3}
{In this subsection, we will drop all the external field in the subscript since it is always $0$.}
We first sketch our proof ideas. By the Edward-Sokal coupling, we can first consider the FK-configurations $\omega^{\pm}$ under $\phi^{\mp}_{\cR}$ and then 
let $s^\pm$ be $\{1,-1\}$-valued random vectors which record the signal on each of the FK-clusters in $\omega^\pm$. 
In light of Lemma~\ref{lem:extended to FK}, we can then define the configuration for the disagreement Ising model as follows: for edges within an FK-cluster, we let their edge spins in the (corresponding copy of the) extended Ising model have the same signal as the vertex spins in this cluster; for edges that are between different FK-clusters, we let their edge spins in the extended Ising model be 0. 
Thus, the configuration pair $(\bar \sigma^+, \bar \sigma^-)$ of the disagreement Ising model is in bijection with $(\omega^+, \omega^-, s^+, s^-)$.

Our main intuition is to map a configuration (for the disagreement Ising model) to another  such that their contributions to the targeted sum cancel each other. Naturally, we wish to construct the aforementioned mapping whenever it is possible, and then the targeted sum is bounded by the contribution from the collection of configurations where such mapping is unavailable (our natural hope is that these configurations are rare). In order to implement this strategy, a natural attempt is to flip the signal in some FK-cluster $\mathcal C_z^\xi$ for $z\in I\cup J$ and $\xi\in\{+,-\}$. Here $\mathcal C_z^\pm$ is the FK-cluster of $z$ in the copy with the {$\mp$-boundary condition (and we remark that we use the superscripts $\pm$ due to the same reason as explained in Remark~\ref{rmk: weird +/- notation}).}   There are three potential issues with this strategy:
\begin{enumerate}[(i)]
\item\label{item: RSW1} If $\mathcal C_z^\xi \cap \partial\mathcal R \neq \emptyset$, then the signal of $\mathcal C_z^\xi$ will be restricted by the boundary condition.
\item\label{item: RSW2} If $\mathcal C_z^\xi$ has an even number of intersections with {$I$ (for $\xi=+$) or $J$ (for $\xi=-$)}, then flipping the signal of $\mathcal C_z^\xi$ may not change the value of the product, and as a result, leads to no cancellation. 
\item\label{item: RSW3} If flipping the signal in $\mathcal C_z^\xi$ changes the {value of $\1_{\hc}$}, then it also ruins the desired cancellation. 
\end{enumerate}

We next try to address these three issues. It is clear that if $\mathcal C_z^\xi $ is local (i.e., $\mathcal C_z^\xi$ does not intersect either $\partial\mathcal R$ or $(I\cup J) \setminus \{z\}$), then Issue \eqref{item: RSW1} and \eqref{item: RSW2} disappear. With this motivation in mind, we make the following definition. (Note that we will slightly change the definition of `local' for later convenience.)
Recalling Definition~\ref{def:upper-bound function for spin average}, we fix $I,J\subset \cR$, and drop the subscripts $\cR,I,J$ in $d_{\cR,I,J}(x)$.

\begin{defi}
     For $z\in I\cup J,$ 
     let {$\Lambda_{d(z)}(z)=\{z\}$} if $d(z)<0$.
     We say $\mcc C_z^\xi$ is \textbf{local} if $\mcc C_z^\xi\cap\intB\Lambda_{d(z)}(z)=\emptyset.$
    Otherwise, we say it is \textbf{non-local}. 
\end{defi}
By \eqref{eq: FK one arm event exponent}, we know the probability of $\mcc C_z^{\xi}$ being non-local decays as ${(\max\{d(z),1\})}^{-\frac{1}{8}}$.


In order to address Issue \eqref{item: RSW3}, our main intuition is that if flipping a `local' FK-cluster $\mathcal C$ changes the event $\hc$, then there should exist two disjoint pre-disagreement clusters that connect $\mathcal C$ to the boundary of a larger box containing $\mathcal C$ (see Figure~\ref{fig: two disjoint disagreement crossing} for an illustration). Therefore, an application of Lemma~\ref{lem: BK-type inequality for disagreement Ising model} gives a good upper bound on the probability for this to occur. 
\begin{defi}

For $z\in I\cup J$, we say $\mathcal C_z^\xi$ is \textbf{excellent} if it is local and flipping the signal on this cluster does not change the value of $1_{\hc}$. In addition, let $\ne$ be the
event that there exists no excellent cluster. We say a cluster is \textbf{pivotal} if it is local {but not excellent, i.e.,}
flipping the signal on this cluster changes the value of $1_{\hc}$.
\end{defi}

The contribution from configurations with at least one excellent cluster is 0, as incorporated in the next lemma.
\begin{lem}
\label{lem:cancellation-of-excellent-clusters}
We have that
   \begin{equation}
    \langle\1_{\hc}\prod_{x\in I} \bar\sigma_{x}^+\prod_{y\in J} \bar\sigma_{y}^- \rangle_{\cR}^{-/+}=\langle\1_\ne\1_{\hc}\prod_{x\in I} \bar\sigma_{x}^+\prod_{y\in J} \bar\sigma_{y}^- \rangle_{\cR}^{-/+}.\label{eq:cancellation-of-excellent-cluster}
    \end{equation}
\end{lem}

\begin{proof}

On the event $(\ne)^c \cap \hc$, let $\mathcal C^\xi_z$ be the first (with respect to some pre-fixed order) excellent cluster in $\{\mathcal C^+_x: x\in I\} \cup \{\mathcal C^-_y: y\in J\}$.
Now we can define a map $\tau:(\ne)^c\cap\hc\to (\ne)^c\cap \hc$ by letting
$$\tau(\omega^+,\omega^-,s^+,s^-)=(\omega^+,\omega^-,\hat s^+,\hat s^-),$$
where $(\hat s^+,\hat s^-)$ only differs from $(s^+,s^-)$ on the cluster $\mcc C_z^\xi$.
In light of the aforementioned bijection between extended Ising spins and FK-clusters with signals, we see that
$(\omega^+,\omega^-,s^+,s^-)$ and $(\omega^+,\omega^-,\hat s^+,\hat s^-)$ correspond to some extended Ising configurations $(\bar\sigma^+,\bar\sigma^-)$ and 
$(\hat\sigma^+,\hat\sigma^-)$ respectively.
The map $\tau$ is well-defined since by the definition of excellent, we know the event $\hc$ also holds on the
configuration $(\omega^+,\omega^-,\hat s^+,\hat s^-)$.
Furthermore, we see that $\tau$ is a bijection since it does not change any excellent cluster (and thus does not change the `first' excellent cluster) and the signal can only change between plus and minus. 
Therefore,
\begin{align} 
&2\langle\1_{(\ne)^c}\1_{\hc}\prod_{x\in I} \bar\sigma_{x}^+\prod_{y\in J} \bar\sigma_{y}^- \rangle_{\cR}^{-/+}\nonumber\\=&\sum
\bar\mu_{\cR}^{-/+}(\bar\sigma^+,\bar\sigma^-)\prod_{x\in I} \bar\sigma_{x}^+\prod_{y\in J} \bar\sigma_{y}^- +\bar\mu_{\cR}^{-/+}(\hat\sigma^+,\hat\sigma^-)\prod_{x\in I} \hat\sigma_{x}^+\prod_{y\in J} \hat\sigma_{y}^-, \label{eq:cancellation-of-configurations-of-excellent}
\end{align}
where the summation is taken over all $(\bar\sigma^+,\bar\sigma^-)\in(\ne)^c\cap \hc$
and the spin configurations $(\hat\sigma^+, \hat\sigma^-)=\tau(\bar\sigma^+, \bar\sigma^-)$. By definition, the signals only differ on one local cluster, so we have 
$$\prod_{x\in I} \bar\sigma_{x}^+\prod_{y\in J} \bar\sigma_{y}^-=-\prod_{x\in I} \hat\sigma_{x}^+\prod_{y\in J} \hat\sigma_{y}^-.$$
Using the local property again, we see that the cluster $\mcc C_z^\xi$ is contained in {$\mcc R\setminus\intB\cR$ and thus }
$$\bar\mu_{\cR}^{-/+}(\bar\sigma^+,\bar\sigma^-)=\bar\mu_{\cR}^{-/+}(\hat\sigma^+,\hat\sigma^-)\,.$$
As a result, the right-hand side of \eqref{eq:cancellation-of-configurations-of-excellent} gets cancelled term by term. This completes the proof.
\end{proof}

Now in order to bound the probability of the right-hand side of \eqref{eq:cancellation-of-excellent-cluster}, we will give an upper bound on the probability of the event $\ne$. This leads to the following definition.
\begin{defi}\label{def:noex}For any $z\in I\cup J$, we define the following events:
\begin{itemize}
    \item $\mcc E_z^1:=\{z\in I\setminus J,  \mcc C_z^+ \text{ is non-local}\}\cup\{z\in J\setminus I, \mcc C_z^- \text{ is non-local}\}; $
    \item $\mcc E_z^2:=\{z\in I\setminus J,  \mcc C_z^+ \text{ is pivotal}\}\cup\{z\in J\setminus I, \mcc C_z^- \text{ is pivotal}\};$
    \item $\mcc E_z^3:=\{z\in I\cap J; \mcc C_z^+, \mcc C_z^-\text{ are non-local}\};$
    \item $\mcc E_z^4:=\{z\in I\cap J; \text{one of }\mcc C_z^+, \mcc C_z^- \text{ is pivotal, and the other is non-local}\};$
    \item $\mcc E_z^5:=\{z\in I\cap J; \mcc C_z^+, \mcc C_z^-\text{ are pivotal}\}$.
\end{itemize}

\end{defi}

It is obvious that for any $z\in I\cup J$, the event $z$ is not contained in an excellent cluster is the union of $\mathcal E_z^1, \ldots, \mathcal E_z^5$.
As aforementioned, the probability of one cluster being non-local has a straightforward upper bound, so the main challenge is to bound the probability of  being pivotal. To this end, we consider the following two scenarios:

\begin{enumerate}[(i)]
    \item There are two subcases:
    $z\in I\Delta J$ and $\mcc C_z^\xi$ is pivotal {\bf or} $z\in I\cap J$ but only one of $\mcc C_z^+, \mcc C_z^-$ is pivotal.
    Without
    loss of generality we can assume $z=x\in I$ and $\mcc C_x^+$ is pivotal. Note that $\hc$ is an increasing event with respect to the partial order as in Definition~\ref{def:partial-order}. If flipping the signal of $\mcc C_x^+$ changes ${\1_\hc}$, the only possibility is that $\hc$ holds with $s_x^+=+1$  and fails with $s_x^+=-1$ (here $s_x^+$ is the signal of $\mcc C_x^+)$. For convenience, denote by $\mcc D^+$  and $\mcc D^-$ the sets of pre-disagreements corresponding to $s_x^+=+1$ and $s_x^+=-1$ respectively. 
    Assume {$\cC_x^+\cap\intB\Lambda_t(x)=\emptyset$ but $\cC_x^+\cap\intB\Lambda_{t-1}(x)\neq\emptyset$} 
    for some integer $t\leq d(x)$. 
    Then $\mcc D^+\Delta\mcc D^-\subset\bar\Lambda_t(x)$. 
    
    Now if $x\in \mcc R\setminus \mcc R'$, then $t\leq d(x)$ means that $\bar\Lambda_t(x)\cap \bar\cR'=\emptyset$. As a result, $\cD^+\cap{\bar\cR'}=\mcc D^-\cap{\bar\cR'}$, which means $\1_\hc$ can not change. In conclusion, $\mcc C_x^+$ can not be pivotal for $x\not\in\mcc R'$. This allows us to assume $x\in \mcc R'$ when we consider pivotal clusters. 
    
    By definition, $\mcc C_x^+$ being pivotal means that there exists a path $\mcc P$ crossing $\mcc R'$ in $\mcc D^+$ but not in $\mcc D^-$.
    {Without loss of generality, assume $\mcc P$ is (one of) the  shortest pre-disagreement crossing so we can write $\mcc P=\{x_1,x_2,\cdots,x_L\}$ where each $x_i$ is a pre-disagreement vertex in $\mcc D^+$ and every edge $ {\{x_i,x_{i+1}\}}\in\mcc D^{{+}} $.
     Since $\mcc P$ is not contained in $\mcc D^-,$ we know $\mcc P\cap \bar\Lambda_t(x)\neq \emptyset.$ This motivates us to define
    $$l_2=\min\{1\leq l\leq L: x_l\in\intB\Lambda_{t+1}(x)\},\mbox{ and }  l_1=\max\{1\leq l<l_2: x_l\in\intB\Lambda_{d(x)}(x)\};$$
    $$l_3=\max\{1\leq l\leq L: x_l\in\intB\Lambda_{t+1}(x)\},\mbox{ and }  l_4=\min\{l_3\leq l\leq L: x_l\in\intB\Lambda_{d(x)}(x)\}.$$
    From the above discussions, we see that $1\leq l_1 < l_2 < l_3 < l_4 \leq L$. We further define $$\mcc P_1=\{x_{l_1},x_{l_1+1},\cdots,x_{l_2}\}, \quad \mcc P_2=\{x_{l_3}, x_{l_3+1},\cdots,x_{l_4}\},$$
    which correspond to the `first' and `last' pre-disagreement crossings of $\mcc P$ in $\overline{\Lambda_{t,d(x)}}{(x)}.$
    Now we claim that $\mcc P_1$ and $\mcc P_2$ are not connected in $\mcc D^{{+}}\cap\,  \overline{ \Lambda_{t,d(x)}}(x).$ Otherwise, this forms a disagreement crossing of $\cR'$ without intersecting with $\overline{\Lambda_{t}}(x)$ which contradicts the event that $\mcc C_x^+$ is pivotal. See Figure~\ref{fig: two disjoint disagreement crossing} for an illustration.
    }


    As a result, we have (recalling Definition~\ref{def: BK event}) 
    \begin{equation}\{\mcc C_x^+\text{ is pivotal}\}\subset\bigcup_{t=1}^{d(x)}{\{\cC_x^+\cap\intB\Lambda_t(x)=\emptyset,\cC_x^+\cap\intB\Lambda_{t-1}(x)\neq\emptyset\}}
    \cap \cont(t, d(x)).\label{eq:decomposition-of-pivotal}\end{equation}


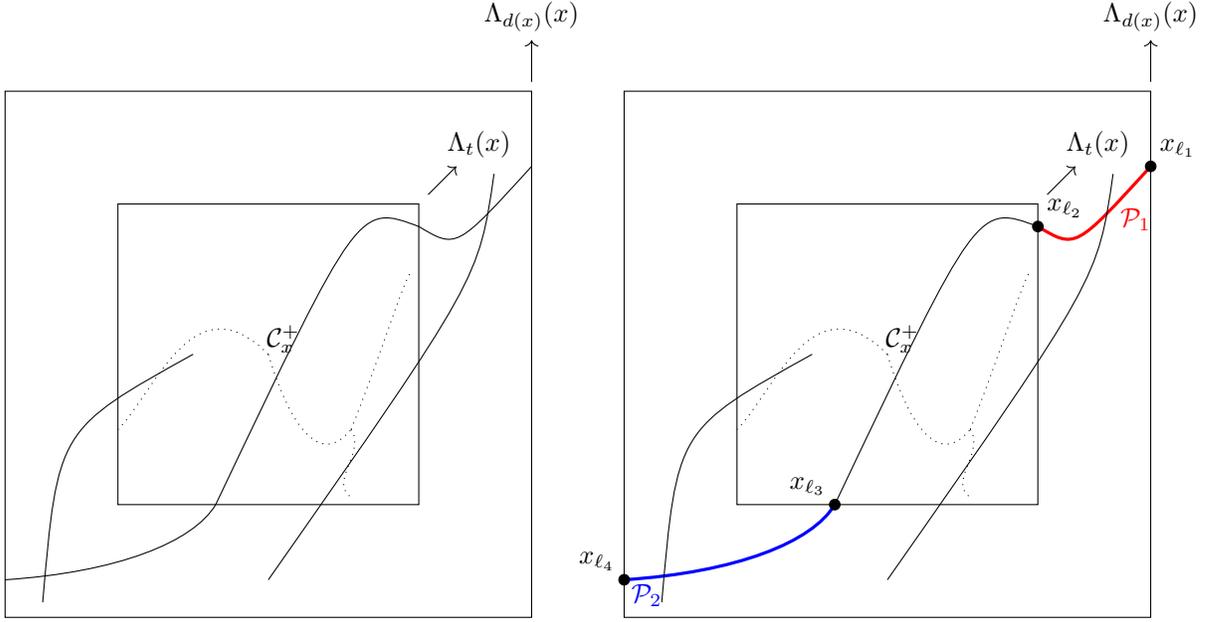
\begin{figure}[htb]
    \centering
    \begin{minipage}[t]{0.49\linewidth}
        \begin{tikzpicture}    
    \draw (-3.5,-3.5)rectangle (3.5,3.5);
    \draw (-2,-2) rectangle (2,2);
    \draw[dotted] (0,0) .. controls (-1,1) and (-1.5,-0.5) .. (-2,-1);
    \draw[dotted] (0,0) .. controls (0.3,-1) and (0.7,-1.5) .. (1.1,-1);
    \draw[dotted] (1.1,-1) .. controls (1.5,0) and (1.8,1) .. (1.9,1.1);
    \draw[dotted] (1.1,-1) .. controls (1.3,-1.3) and (0.8,-1.7) .. (1.1,-1.9);
    \draw[ ] (-3.5,-3) .. controls (-2,-2.9) and (-1,-2.5) .. (-0.7,-2);
    \draw[ ] (-0.7,-2) .. controls (1.2,2) .. (2,1.7);
    \draw[ ] (2,1.7) .. controls (2.5,1.4) .. (3.5,2.5);
    \draw[ ] (3,2.4) .. controls (2.8,1) .. (0,-3);
    \draw[ ] (-3,-3.3) .. controls (-2.8,-1) .. (-1,0);
    \node [name=start1] at (3.5,3.5){ };
    \node [name=end1] at (3.5,4.5){$\Lambda_{d(x)}(x)$ };
    \draw [->] (start1)--(end1) ;
    \node [name=start2] at (2,2){ };
    \node [name=end2] at (2.8,2.8){ $\Lambda_t(x)$};
    \node at (0.2,0.2) {$\cC_{x}^+$};
    \draw [->] (start2)--(end2) ;
    \end{tikzpicture}
    \end{minipage}
    \begin{minipage}[t]{0.49\linewidth}
        \begin{tikzpicture}    
    \draw (-3.5,-3.5)rectangle (3.5,3.5);
    \draw (-2,-2) rectangle (2,2);
    \draw[dotted] (0,0) .. controls (-1,1) and (-1.5,-0.5) .. (-2,-1);
    \draw[dotted] (0,0) .. controls (0.3,-1) and (0.7,-1.5) .. (1.1,-1);
    \draw[dotted] (1.1,-1) .. controls (1.5,0) and (1.8,1) .. (1.9,1.1);
    \draw[dotted] (1.1,-1) .. controls (1.3,-1.3) and (0.8,-1.7) .. (1.1,-1.9);
    \draw[color= blue,very thick] (-3.5,-3) .. controls (-2,-2.9) and (-1,-2.5) .. (-0.7,-2);
    \draw[fill=black] (-3.5,-3) circle(0.2em);
    \node[above left] at (-3.5,-3) {$x_{\ell_4}$};
    \draw[fill=black] (-0.7,-2) circle(0.2em);
    \node[above left] at (-0.7,-2) {$x_{\ell_3}$};
    \draw[ ] (-0.7,-2) .. controls (1.2,2) .. (2,1.7);
    \draw[color=red, very thick] (2,1.7) .. controls (2.5,1.4) .. (3.5,2.5);
    \draw[fill=black] (2,1.7) circle(0.2em);
    \node[above right] at (2,1.7) {$x_{\ell_2}$};
    \draw[fill=black] (3.5,2.5) circle(0.2em);
    \node[above right] at (3.5,2.5) {$x_{\ell_1}$};
    
    \node [color=red] at (3.3,1.8) {$\mcc P_1$};
    \node [color=blue] at (-3.2,-3.2) {$\mcc P_2$};
    \draw (3,2.4) .. controls (2.8,1) .. (0,-3);
    \draw (-3,-3.3) .. controls (-2.8,-1) .. (-1,0);
    \node [name=start1] at (3.5,3.5){ };
    \node [name=end1] at (3.5,4.5){$\Lambda_{d(x)}(x)$ };
    \draw [->] (start1)--(end1) ;
    \node [name=start2] at (2,2){ };
    \node [name=end2] at (2.8,2.8){ $\Lambda_t(x)$};
    \node at (0.2,0.2) {$\cC_{x}^+$};
    \draw [->] (start2)--(end2) ;
    \end{tikzpicture}
    \end{minipage}
    \caption{Illustration of two disjoint pre-disagreement crossings. Dotted lines: the FK-cluster $\cC_{x}^+$. Solid lines: the pre-disagreements in $\Lambda_{d(x)}(x)$. Colored lines: 
    the first and last pre-disagreement crossings of $\mcc P$ in $\Lambda_{t,d(x)}(x)$.}
    \label{fig: two disjoint disagreement crossing}
\end{figure}

\item $z\in I\cap J$, and  both $\mcc C_z^+$ and $\mcc C_z^-$ are pivotal. Similarly, we can
assume for integers $t^+$ and $t^-$, {$$\cC_z^+\cap\intB\Lambda_{t^+}(z)=\emptyset\,,\cC_z^+\cap\intB\Lambda_{t^+-1}(z)\neq\emptyset;\quad\cC_z^-\cap\intB\Lambda_{t^-}(z)=\emptyset\,,\cC_z^-\cap\intB\Lambda_{t^--1}(z)\neq\emptyset.$$}
Using the same analysis, we get that the events $\cont(t^+,d(z))$ and $ \cont(t^-,d(z))$ both happen.
{(But these two events are highly dependent so in what follows we try to use the ``more restricted'' one.)}
Therefore,  we see that there exists $t\leq d(z)$ such that the following event (or the version with $+$ and $-$ switched) holds:{
\begin{equation}
    \begin{aligned}
        \mathcal \cont( t,d(z)) \mbox{ occurs } &\mbox{ and } \mathcal C_z^{-} \cap\intB \Lambda_{t-1}(z)\neq \emptyset \\&\mbox { and }  \mathcal C_z^{{+}} \cap\intB \Lambda_{t-1}(z)\neq \emptyset, \mathcal C_z^{{+}} \cap\intB \Lambda_{t}(z)=\emptyset  \,.
    \end{aligned}
    \label{ev:remove-thinner}
\end{equation}}


\end{enumerate}



Provided with discussions on these two scenarios, we now prove the following lemma.
\begin{lem}\label{lem:bound-of-noex}
There exists a constant $\oc{13}>0$ relying on $\oc{fk exponent}$ and $\oc{14}$
such that
    \begin{equation}
        \bar\mu_{\cR}^{-/+}(\ne)\leq \oc{13}^{|I|+|J|}\prod_{z\in I\cup J}{(\max\{d(z),1\})}^{-\frac{\1_{z\in I}+\1_{z\in J}}{8}}.\label{eq:bound-of-noex}
    \end{equation}
\end{lem}

\begin{proof}
By Definition~\ref{def:noex}, we have 
    $\ne\subset\bigcap_{z\in I\cup J}(\cup_{i=1}^5 \mcc E_z^i).$
    For $\xi^{\pm}=(\xi_z^{\pm})_{z\in I\cup J}$ with $\xi_z^\pm\in\{-1,1\}^{\intB\Lambda_{d(z)}(z)},$ let 
    $$\Psi(\xi^+,\xi^-)=\bar\mu_{\cR}^{-/+}(\sigma^\pm|_{\intB\Lambda_{d(z)}(z)}=\xi_z^\pm
    \mbox{ for all }z\in I\cup J ).$$
    Then by the DMP, we have 
    $$\bar\mu_{\cR}^{-/+}(\ne)\leq\sum_{\xi^+,\xi^-}\Psi(\xi^+,\xi^-)\prod_{z\in I\cup J}\bar\mu^{{\xi_z^+,\xi_z^-}}_{\Lambda_{d(z)}(z)}(\cup_{i=1}^5 \mcc E_z^i).$$
    It is possible that $d(z)\leq 0$, in which case we take $\bar\mu^{\xi_z^+/\xi_z^-}_{\Lambda_{d(z)}(z)}(\cup_{i=1}^5 \mcc E_z^i)=1$ {as a trivial upper-bound}. So in order to prove \eqref{eq:bound-of-noex}, taking the law of total probability into account, we only need to prove the following three estimations:
    \begin{equation}
        \bar\mu^{\xi^+_z/\xi^-_z}_{\Lambda_{d(z)}(z)}(\mcc E_z^1\cup\mcc E_z^3)\leq \left[\oc{13}{(\max\{d(z),1\})}^{-\frac{1}{8}}\right]^{\1_{z\in I}+\1_{z\in J}},\label{eq:noex-case13}
    \end{equation}
    \begin{equation}
        \bar\mu^{\xi_z^+/\xi_z^-}_{\Lambda_{d(z)}(z)}(\mcc E_z^2)\leq\oc{13}{(\max\{d(z),1\})}^{-\frac{1}{8}},\label{eq:noex-case2}
    \end{equation}
    \begin{equation}
        \bar\mu^{\xi^+_z/\xi^-_z}_{\Lambda_{d(z)}(z)}(\mcc E_z^4\cup \mcc E_z^5)\leq\left[\oc{13}{(\max\{d(z),1\})}^{-\frac{1}{8}}\right]^2 .\label{eq:noex-case45}
    \end{equation}

We see that \eqref{eq:noex-case13} follows from \eqref{eq: FK one arm event exponent} as long as $\oc{13} \geq \oc{fk exponent}$. The other two estimates, \eqref{eq:noex-case2} and \eqref{eq:noex-case45} require more effort and in particular require an application of Lemma~\ref{lem: BK-type inequality for disagreement Ising model}; we postpone their proofs until the end of this subsection. 
\end{proof}

With Lemma{s}~\ref{lem:cancellation-of-excellent-clusters} and \ref{lem:bound-of-noex} at hand, we can prove Lemma~\ref{lem: bound for the spin average conditioned on the crossing}.

\begin{proof}[Proof of Lemma~\ref{lem: bound for the spin average conditioned on the crossing}.]
Recalling Definition~\ref{def:upper-bound function for spin average}, we know \eqref{eq: bound for the spin average conditioned on the crossing} is a direct consequence of \eqref{eq:cancellation-of-excellent-cluster}
and \eqref{eq:bound-of-noex}. As for the proof of \eqref{eq: upper-bound for sum of squares of Fs}, it is very similar to that in \cite[Lemma 8.3]{FSZ16}. The only difference is that we have restricted $d(x)< dist(x,\intB\cR')$. In other words, we let the boundary be $\intB\mcc R\cup \intB\mcc R'$. Note that in their proof, the only assumption on the boundary is that $$\sum_{x\in \cR}\Big[\max\{dist(x,\intB\mcc R\cup\intB\mcc R'),1\}\Big]^{-\frac{1}{4}}\le C_1M^{\frac{7}{4}}$$ for some constant $C_1>0$. So our modified boundary satisfies the only property for the boundary that was used in \cite{FSZ16}, and thus their proof extends to our case. The proof of \eqref{eq: upper-bound for sum of of FUs} follows from an adaptation of the induction in \cite[Lemma 8.3]{FSZ16}. We only provide a sketch emphasizing the additional subtleties. To carry out the induction, it suffices to prove that \begin{align}
&\sum_{x_1,\cdots,x_r\in\cR\setminus(I\cup J)} F^2_{\mcc R}(I\cup \{x_1,\cdots,x_r\}, J\cup \{x_1,\cdots,x_r\})\nonumber\\\le ~&\sum_{x_1,\cdots,x_{r-1}\in\cR\setminus(I\cup J)} C(r+|I|+|J|)^{\frac{1}{2}}F^2_{\mcc R}(I\cup \{x_1,\cdots,x_{r-1}\}, J\cup \{x_1,\cdots,x_{r-1}\}).\label{eq:3}
\end{align} The proof of \eqref{eq:3} is highly similar to that of \cite[(8.17)]{FSZ16}. We emphasize that the only difference between \eqref{eq:3} and \cite[(8.17)]{FSZ16} comes from the term $(r+|I|+|J|)^{\frac{1}{2}}$. The reason is that the power of $d(x_i)$ in $F_{\mcc R}(I\cup \{x_1,\cdots,x_{r-1}\}, J\cup \{x_1,\cdots,x_{r-1}\})$ is $-\frac{1}{4}$ rather than $-\frac{1}{8}$ as in \cite[Lemma 8.3]{FSZ16} and thus leads to the difference.
\end{proof}

The rest of this subsection is devoted to the proofs of \eqref{eq:noex-case2} and \eqref{eq:noex-case45}.

\begin{proof}[Proof of \eqref{eq:noex-case2}.]
Without loss of generality, we can assume $z\in I\setminus J,  d(z)>1$ and $\mcc C_z^+$ is pivotal.
{Denote by $\mathtt L_t$  the event that $\{\mcc C_z^+\cap\intB\Lambda_{t}(z)=\emptyset\}\cap \{\mcc C_z^+\cap\intB\Lambda_{t-1}(z)\neq\emptyset\}$. }
Recalling \eqref{eq:decomposition-of-pivotal}, we have 
\begin{equation}
    \bar\mu^{\xi_z^+/\xi_z^-}_{\Lambda_{d(z)}(z)}(\mcc E_z^2)\leq \sum_{t=1}^{d(z)}\bar\mu^{\xi_z^+/\xi_z^-}_{\Lambda_{d(z)}(z)}\Big(\mathtt L_t\cap \cont(t,d(z))\Big)\,.\label{eq:decompose-of-noex-case2}
\end{equation}

Let $\sB_t$ denote the collection of boundary conditions on $\intB\Lambda_{t}(z)$, i.e., $\sB_t=\{1,-1\}^{\intB\Lambda_{t}(z)}$.
Here notice that $\mathtt L_t$ is measurable with respect to the configurations on {$\Lambda_{t}$}. 
 For $\nu^+,\nu^-\in \sB_t$, let $$\Psi_t(\nu^+,\nu^-):=\bar\mu^{\xi^+_z/\xi_z^-}_{\Lambda_{d(z)}(z)}\Big(\bar\sigma^\pm|_{\intB\Lambda_{t}(x)}=\nu^\pm \,
 \Big).$$ Applying the DMP and Lemma ~\ref{lem: BK-type inequality for disagreement Ising model}, we obtain that 
    \begin{align}
        &\sum_{t=1}^{d(z)}\bar\mu^{\xi^+_z/\xi_z^-}_{\Lambda_{d(z)}(z)}(\mathtt L_t\cap \cont(t,d(z)))\nonumber\\ {=}& \sum_{t=1}^{d(z)}\sum_{\nu^+,\nu^-\in \sB_t}\Psi_t(\nu^+,\nu^-)\bar\mu_{\Lambda_{t}(z)}^{\nu^+}(\mathtt L_t)\times\bar\mu_{\Lambda_{t,d(z)}(z)}^{\nu^+/\nu^-,\xi^+_z/\xi_z^-}\big(\cont(t,d(z))\big)\nonumber~~~\text{(by DMP)}\\ {\le} &\sum_{t=1}^{d(z)}\left[\sum_{\nu^+,\nu^-\in \sB_t}\Psi_t(\nu^+,\nu^-)\times\bar\mu_{\Lambda_{t}(z)}^{\nu^+}(\mathtt L_t)\right]\times \oc{14}(\frac{d(z)}{t})^{-\frac{1}{4}-\delta}\nonumber~~~(\text{by Lemma}~ \ref{lem: BK-type inequality for disagreement Ising model}),
    \end{align}
where $\bar\mu_{\Lambda_{t,d(z)}(z)}^{\nu^+/\nu^-,\xi^+_z/\xi_z^-}(\cont(t,d(z))):=1$ if $t=d(z)$. 
Applying the total law of probability and the change from $\bar \mu^{\xi_z^+}_{\Lambda_{d(z)}(z)}$ to $\phi^{\xi_z^+}_{\Lambda_{d(z)}(z)}$ (which follows from Lemma~\ref{lem:extended to FK}), we obtain that \begin{align}  &\sum_{t=1}^{d(z)}\left[\sum_{\nu^+,\nu^-\in \sB_t}\Psi_t(\nu^+,\nu^-)\times\bar\mu_{\Lambda_{t}(z)}^{\nu^+}(\mathtt L_t)\right]\times \oc{14}(\frac{d(z)}{t})^{-\frac{1}{4}{-\delta}}\nonumber\\ =~&\sum_{t=1}^{d(z)}\phi_{\Lambda_{d(z)}(z)}^{\xi_z^+}(\mathtt L_t)\times \oc{14}(\frac{d(z)}{t})^{-\frac{1}{4}{-\delta}}.\label{eq: tough event bound for 4, 5: part 1}
\end{align}
To bound the right-hand side in \eqref{eq: tough event bound for 4, 5: part 1}, we use the following Abel transformation. For arbitrary numbers $a_i,b_i (i=1,\cdots,k)$ and $b_0=0$, we have \begin{equation}\label{eq: Abel transform}
\sum_{i=1}^k a_ib_i=\sum_{i=1}^k(\sum_{j=i}^ka_j)\cdot(b_i-b_{i-1}).
\end{equation}
Applying \eqref{eq: Abel transform}, we get 
    \begin{align}
        &\sum_{t=1}^{d(z)}\phi_{\Lambda_{d(z)}(z)}^{\xi_z^+}(\mathtt L_t)\times \oc{14}(\frac{d(z)}{t})^{-\frac{1}{4}{-\delta}}\nonumber\\
        \stackrel{\eqref{eq: Abel transform}}{=}&\oc{14}\sum_{t=1}^{d(z)}\phi_{\Lambda_{d(z)}(z)}^{\xi_z^+}(\mathtt L_t\cup\cdots\cup \mathtt L_{d(z)})\times [(\frac{d(z)}{t})^{-\frac{1}{4}-\delta}-(\frac{d(z)}{t-1})^{-\frac{1}{4}{-\delta}}]\nonumber\\
        \stackrel{(*)}{\le}~& \oc{14}d(z)^{-\frac{1}{4}{-\delta}}+\oc{14}\sum_{t=2}^{d(z)}\oc{fk exponent}(t-1)^{-\frac{1}{8}}\times [(\frac{d(z)}{t})^{-\frac{1}{4}{-\delta}}-(\frac{d(z)}{t-1})^{-\frac{1}{4}{-\delta}}]\nonumber
        \\\le~& C(\oc{fk exponent},\oc{14},\delta) d(z)^{-\frac{1}{8}},\label{eq: tough event bound for 4, 5: part 2}
    \end{align}
where $(*)$ follows from \eqref{eq: FK one arm event exponent} and the fact that $(\frac{d(z)}{t})^{-\frac{1}{4}{-\delta}}$ increases in $t$ and the last inequality comes from some basic analysis. Choosing $\oc{13}>C$ and combining \eqref{eq: tough event bound for 4, 5: part 1} and \eqref{eq: tough event bound for 4, 5: part 2}, we complete the proof of \eqref{eq:noex-case2}.
\end{proof}

\begin{proof}[Proof of \eqref{eq:noex-case45}.]

The bound for $\mcc E_z^4$ and $\mcc E_z^5$ can be derived together, since on either of the events we have that \eqref{ev:remove-thinner} (or the version with
$+$ and $-$ switched)  holds (the case for $\mathcal E_z^5$ was derived around 
\eqref{ev:remove-thinner}, and the case for $\mathcal E_z^4$ is obvious).
Let us assume without loss of generality that the original version of \eqref{ev:remove-thinner} holds (this only loses a factor of $2$ in the estimate by symmetry), and we denote by
{$$\mathtt L_t = \{\mcc C_z^+\cap\intB\Lambda_{t}(z)=\emptyset\}\cap \{\mcc C_z^+\cap\intB\Lambda_{t-1}(z)\neq\emptyset\}, \quad \mathtt K_t  = \{\mathcal C_z^- \cap\intB\Lambda_{t-1}(z)\neq\emptyset \}.$$}
Using a similar derivation for \eqref{eq:decompose-of-noex-case2} and \eqref{eq:noex-case2} 
we have 
    \begin{align}
        &\frac{1}{2}\bar\mu^{\xi_z^+/\xi_z^-}_{\Lambda_{d(z)}(z)}(\cE_z^4\cup \cE_z^5)\nonumber\\ \leq~& \sum_{t=1}^{d(z)}\sum_{\nu^+,\nu^-\in \sB_t}\Psi_t(\nu^+,\nu^-)\bar\mu_{\Lambda_{t}(z)}^{\nu^+}(\mathtt L_t)\times\bar\mu_{\Lambda_{t}(z)}^{\nu^-}(\mathtt K_t)\times\bar\mu_{\Lambda_{t,d(z)}(z)}^{\nu^+/\nu^-,\xi_z^+/\xi_z^-}(\cont(t,d(z)))\nonumber\\ \le ~&\sum_{t=1}^{d(z)}\sum_{\nu^+,\nu^-\in \sB_t}\Psi_t(\nu^+,\nu^-)\bar\mu_{\Lambda_{t}(z)}^{\nu^+}(\mathtt L_t)\times\bar\mu_{\Lambda_{t}(z)}^{\nu^-}(\mathtt K_t)\times \oc{14}(\frac{d(z)}{t})^{-\frac{1}{4}{-\delta}}\nonumber~~~(\text{by Lemma}~\ref{lem: BK-type inequality for disagreement Ising model}),
    \end{align}
where in the last inequality we recalled that $\bar\mu_{\Lambda_{t,d(z)}(z)}^{\nu^+/\nu^-,\xi^+_z/\xi_z^-} (\cont(t,d(z)))=1$ if $t=d(z)$.
    Combined with the law of total probability and the change from $\bar \mu^{\xi_z^+}_{\Lambda_{d(z)}(z)}$ to $\phi^{\xi_z^+}_{\Lambda_{d(z)}(z)}$ (by Lemma~\ref{lem:extended to FK}), it yields that \begin{align}
        \frac{1}{2}\bar\mu^{\xi_z^+/\xi_z^-}_{\Lambda_{d(z)}(z)}(\cE_x^4\cup \cE_x^5)\le~&\sum_{t=1}^{d(z)}\phi_{\Lambda_{d(z)}(z)}^{\xi_z^+}(\mathtt L_t)\phi_{\Lambda_{d(z)}(z)}^{\xi_z^-}(\mathtt K_t)\times \oc{14}(\frac{d(z)}{t})^{-\frac{1}{4}{-\delta}}\nonumber\\
        \stackrel{\eqref{eq: FK one arm event exponent}}\leq&\sum_{t=1}^{d(z)}\phi_{\Lambda_{d(z)}(z)}^{\xi_z^+}(\mathtt L_t)\times \oc{fk exponent}t^{-\frac{1}{8}}\oc{14}(\frac{d(z)}{t})^{-\frac{1}{4}{-\delta}}.\label{eq: tough event bound for 6, 7: part 1}
    \end{align}
Applying \eqref{eq: Abel transform}, we get that
    \begin{align}
        &\sum_{t=1}^{d(z)}\phi_{\Lambda_{d(z)}(z)}^{\xi_z^+}(\mathtt L_t)\times \oc{fk exponent}t^{-\frac{1}{8}}\oc{14}(\frac{d(z)}{t})^{-\frac{1}{4}{-\delta}}\nonumber\\
        \stackrel{\eqref{eq: Abel transform}}{=}&\oc{fk exponent}\oc{14}\sum_{t=1}^{d(z)}\phi_{\Lambda_{d(z)}(z)}^{\xi_z^+}(\mathtt L_t\cup\cdots\cup \mathtt L_{d(z)})\times [t^{-\frac{1}{8}}(\frac{d(z)}{t})^{-\frac{1}{4}{-\delta}}-(t-1)^{-\frac{1}{8}}(\frac{d(z)}{t-1})^{-\frac{1}{4}{-\delta}}]\nonumber\\
        \stackrel{(*)}{\le} ~&\oc{fk exponent}\oc{14}d(z)^{-\frac{1}{4}{-\delta}}+\oc{fk exponent}\oc{14}\sum_{t=2}^{d(z)}\oc{fk exponent}(t-1)^{-\frac{1}{8}}\times [t^{-\frac{1}{8}}(\frac{d(z)}{t})^{-\frac{1}{4}{-\delta}}-(t-1)^{-\frac{1}{8}}(\frac{d(z)}{t-1})^{-\frac{1}{4}{-\delta}}]\nonumber\\
        \le ~&C(\oc{fk exponent},\oc{14},\delta) d(z)^{-\frac{1}{4}}
        ,\label{eq: tough event bound for 6, 7: part 2}
    \end{align}
    where $(*)$ follows from \eqref{eq: FK one arm event exponent} and the fact that $t^{-\frac{1}{8}}(\frac{d(x)}{t})^{-\frac{1}{4}{-\delta}}$ increases in $t$ and the last inequality comes from some basic analysis.

Choosing $\oc{13}^2>{4C}$ and
combining \eqref{eq: tough event bound for 6, 7: part 1} and \eqref{eq: tough event bound for 6, 7: part 2}, we complete the proof of \eqref{eq:noex-case45}.
\end{proof}

\begin{rmk}\label{rmk-ron}
     We point out that the $-\delta$ term in the exponent is of vital importance. The reason is that without this term, the final line of \eqref{eq: tough event bound for 6, 7: part 2} will include an extra $\ln d(z)$ term, which is insufficient for our applications. Meanwhile, \eqref{eq: tough event bound for 4, 5: part 2} still holds without this term and we just include it for {consistency}.
\end{rmk}

\section{Proof ingredients of Theorems~\ref{thm-critical-temperature} and~\ref{thm-critical-temperature-small-perturbation}}\label{sec:postponed}
In this section, we provide the postponed proofs for ingredients that were employed in the proofs for Theorems \ref{thm-critical-temperature} and \ref{thm-critical-temperature-small-perturbation}, that is, Lemmas \ref{lem: small perturbation for partition function}, \ref{lem: small perturbation for expectation}, \ref{lem: good external field probability 0}, \ref{lem: good external field probability 1}, \ref{lem: domination by product measure},~\ref{lem: perfect external field bound}, \ref{lem: dis crossing probability upper bound}, Theorem~\ref{thm:Fractality} and Proposition \ref{prop: strictly faster polynomial decay}.
\subsection{Proof ingredients of Theorem~\ref{thm-critical-temperature-small-perturbation}}\label{sec: small perturbation}


In this subsection, we prove Lemmas~\ref{lem: small perturbation for partition function} and \ref{lem: small perturbation for expectation}.

\begin{proof}[Proof of Lemma~\ref{lem: small perturbation for partition function}]
    The proof of this lemma is very similar to that of Lemma~\ref{lem: small perturbation for crossing probability}. To show \eqref{eq: good external field for expectation1}, consider the following expansion:
    {$$ \langle ~\prod_{v\in\cR}[1+\sigma_v\tanh(\veps h_v)]~\rangle_{\mcc R,0}^{+}= 1+ \sum_{k=1}^{\infty}\Phi_{k}(h),$$
    where $$\Phi_k(h)= \sum_{\substack{I\subset\cR\\|I|=k}} \langle\sigma^{I}\rangle_{\mcc R,0}^{+} \prod_{v\in I}\tanh(\veps h_v).$$ Here we  recall that $\sigma^I=\prod_{x\in I}\sigma_x.$}
Applying Lemmas~\ref{lem: upper-bound for the sum of squares of k point function} and \ref{lem: concentration for sum of products of tanh}, we obtain that \begin{equation}\label{eq: upper bound on psi_k to be big}
    \P\Big(|\Phi_k(h)|>(\eps M^{\frac{7}{8}})^{k/2}/2\Big)\le C_1^{-1}\exp\Big(-C_1\sqrt{\eps^{-1} M^{-\frac{7}{8}}}\times(k!)^{\frac{3}{8k}}\Big).
\end{equation} Note that $\sum_{k=1}^{\infty}(\eps M^{\frac{7}{8}})^{k/2}\le C_2\sqrt{\eps M^{\frac{7}{8}}}$ for some constant $C_2>0$ relying on $\eps M^{\frac{7}{8}}$, and we can assume $C_2<2$ by choosing $\oc{12}>0$ small enough. Thus we conclude by \eqref{eq: upper bound on psi_k to be big} that \begin{equation*}
\begin{aligned}
    &\P(|\sum_{k=1}^{\infty}\Phi_k(h)| > \sqrt{\eps M^{\frac{7}{8}}})\le \sum_{k=1}^{\infty}\P(|\Phi_k(h)|>(\eps M^{\frac{7}{8}})^{k/2}/2)\\ \stackrel{\eqref{eq: upper bound on psi_k to be big}}{\le}& \sum_{k=1}^{\infty}C_1^{-1}\exp\Big(-C_1\sqrt{\eps^{-1} M^{-\frac{7}{8}}}\times(k!)^{\frac{3}{8k}}\Big)\le C_3^{-1}\exp(-C_3\sqrt{\eps^{-1} M^{-\frac{7}{8}}}).
\end{aligned} 
\end{equation*}
This completes the proof.
\end{proof}
 
\begin{proof}[Proof of Lemma~\ref{lem: small perturbation for expectation}]
    The proof of this lemma is similar to that of Lemma~\ref{lem: small perturbation for partition function} with an additional complication of using a product version of Lemma~\ref{lem: concentration for sum of products of tanh}. 
    Define 
$$\Phi^+_{k,m}(h)=\sum_{|I|=k,|J|=m}\langle ~\sigma_o\sigma^I~\rangle_{\lamn,0}^{+}\langle ~\sigma^J~\rangle_{\lamn,0}^{-}\prod_{x\in I}\tanh(\veps h)\prod_{y\in J}\tanh(\veps h)=\Phi^{+,o}_k(h)\Phi^-_m(h),$$
where 
\begin{align}
\Phi^{+,o}_k(h)=\sum_{|I|=k}\langle ~\sigma_o\sigma^I~\rangle_{\lamn,0}^{+}\prod_{x\in I}\tanh(\veps h),\mbox{  and  }
\Phi^-_m(h)=\sum_{|J|=m}\langle ~\sigma^J~\rangle_{\lamn,0}^{-}\prod_{y\in J}\tanh(\veps h).\nonumber
\end{align}
For any $A_{k,m} > 0$, we have that \begin{align*}
    &\mbb P\left(|\Phi^{+}_{k,m}(h)|>(\eps M^{\frac{7}{8}})^{\frac{m+k}{2}}\langle\sigma_o\rangle^+_{\lamn,0}/6\right)\\ \leq & \mbb P\left(|\Phi^{+,o}_{k}(h)|>A_{k,m}(\eps M^{\frac{7}{8}})^{\frac{k}{2}}\langle\sigma_o\rangle^+_{\lamn,0}\right)  +  \mbb P\left(|\Phi^{-}_{m}(h)|>A_{k,m}^{-1}(\eps M^{\frac{7}{8}})^{\frac{m}{2}}/6\right)\\
    \stackrel{(*)}\leq & C_1^{-1}\exp(-C_1 A_{k,m}^{\frac{1}{k}}\sqrt{\eps^{-1}M^{-\frac{7}{8}}}\times \frac{(k!)^{\frac{1}{2k}}}{[(k+1)!]^{\frac{1}{8k}}})+C_1^{-1}\exp(-C_1 A_{k,m}^{-\frac{1}{m}}\sqrt{\eps^{-1}M^{-\frac{7}{8}}}\times (m!)^{\frac{3}{8m}})\\
    \leq & C_1^{-1}\exp(-C_1 A_{k,m}^{\frac{1}{k}}\sqrt{\eps^{-1}M^{-\frac{7}{8}}}\times (k!)^{\frac{1}{4k}})+C_1^{-1}\exp(-C_1 A_{k,m}^{-\frac{1}{m}}\sqrt{\eps^{-1}M^{-\frac{7}{8}}}\times (m!)^{\frac{1}{4m}}),
\end{align*}
where (*) used Lemmas~\ref{lem: upper-bound for the sum of squares of k point function} and~\ref{lem: concentration for sum of products of tanh} twice.  In order to optimize the right-hand side above, we balance the two exponents by choosing $A_{k,m} = (m!)^{\frac{k}{4(k+m)}}(k!)^{-\frac{m}{4(k+m)}}$.
As a result, we have 
\begin{equation}
    \mbb P\left(|\Phi^{+}_{k,m}(h)|>(\eps M^{\frac{7}{8}})^{\frac{k+m}{2}}\langle\sigma_o\rangle^+_{\lamn,0}/6\right)\leq 2C_1^{-1}\exp(-C_1 \sqrt{\eps^{-1}M^{-\frac{7}{8}}}\times (k!m!)^{\frac{1}{4(k+m)}}). \label{eq:6}
\end{equation}
The same bound can be derived for 
\begin{equation*}
    \Phi^-_{k,m}(h)=\sum_{|I|=k,|J|=m}\langle ~\sigma^I~\rangle_{\lamn,0}^{+}\langle ~\sigma_o\sigma^J~\rangle_{\lamn,0}^{-}\prod_{x\in I}\tanh(\veps h)\prod_{y\in J}\tanh(\veps h).\label{eq:7}
\end{equation*}
With \eqref{eq:6} and its analog for $\Phi^-_{k,m}(h)$ in place of \eqref{eq: upper bound on psi_k to be big}, the rest is just a repetition of the proof of Lemma~\ref{lem: small perturbation for partition function}.
\end{proof}

\subsection{Fractality of disagreement percolation}\label{sec: fractality}

In this subsection, we prove Theorem \ref{thm:Fractality} using the following result of \cite{AB99}: In a percolation system, any crossing has dimension strictly larger than $1$ provided that the percolation system satisfies some tortuous condition (as formulated in \cite{AB99}). Naturally, our main task is to verify this tortuous condition for the disagreement percolation at criticality with the presence of any external field.

Recall that $\{B_i\}_{i\in I}$ is a partition of $\lamn$ into $M$-boxes and recall $\cE_{\alpha,N,M}$ from Definition \ref{def: dis crossing intersecting a large number of M-boxes}.
For a collection of rectangles $\mathfrak{R}$, we say it is {\em well-separated} if the distance from each rectangle $\cR\in \mathfrak R$ to any other rectangle in $\mathfrak R$ is at least $60$ times the diameter of $\cR$.
The set $\mathcal{D}\subset \bar{\mbb Z}^2$ can be naturally embedded into $\mbb R^2$ with line segments connecting a vertex $v$ and a midpoint of an edge $e$ as long as $v$ is an endpoint of $e$ and  both of them are contained in $\mathcal{D}$. With such an embedding, we are then under the same setting as in \cite{AB99} which considers continuous curves.

For a rectangle $\cR$, 
we denote $\rc(\cR)$ the event that there exists a path of pre-disagreements in $\cR$ joining the two short sides of $\cR$. We point out that there is a slight difference between $\rc(\cR)$ and $\mathtt{Hcross}(\cR)$ in Definition \ref{def: horizontal crossing} which focuses on horizontal crossings. Then we only need to check the following condition.

\begin{lem}  \label{lem:criterion}
There exists an absolute constant $b >0$ such that the following holds. For any $N\geq 1$, for any collection of well-separated rectangles $\mathfrak{R}$ contained in $\Lambda_{N/2,N}$ and for any external field $h$,
  \begin{equation*} \label{eq:criterion}
    \left\langle \prod_{\mcc R\in\mathfrak{R}}\1_{\rc(\mcc R)}\right\rangle^{ +/-{,+/-}}_{\Lambda_{N/2,N},h}\le \left(1-b \right)^{|\mathfrak{R}|}.
  \end{equation*}
\end{lem}
\begin{proof}
    Provided with Lemma~\ref{lem:crossing},
    the proof only uses CBC and the fact that crossing a rectangle is harder than crossing an annulus with comparable size. We omit further details here.
\end{proof}
\begin{proof}[Proof of Theorem~\ref{thm:Fractality}]
    By Lemma~\ref{lem:criterion} and \cite[Theorem 1.3]{AB99}, we obtain the following uniform convergence over all instances of external field:
\begin{equation}\label{weak fractality}
    \lim_{N\to\infty}\langle \1_{\cE_{\alpha,N,M}}\rangle^{+/-{,+/-}}_{\Lambda_{N/2,N},\eps h} \to 0 \,.
\end{equation}
Then by a standard percolation argument, we can enhance the probability decay in \eqref{weak fractality} and prove Theorem~\ref{thm:Fractality}. This argument for enhancement was written in \cite[Theorem 5.5]{AHP20} and \cite[Proposition 3.1]{DX21}, and thus we omit further details.
\end{proof}

\subsection{Proof ingredients of Theorem \ref{thm:main thm-upper bound}}\label{sec: upper-bound for main theorem}

In this subsection, we prove Proposition~\ref{prop: strictly faster polynomial decay} and Lemmas~\ref{lem: domination by product measure}, \ref{lem: perfect external field bound}, \ref{lem: dis crossing probability upper bound}.

We define the following surface tension to analyze the influence of boundary conditions in an annulus, as in \cite{AP19}.  Note that the partition function of the Ising model equals to that of the extended Ising model (see  \cite[Lemma 2.1]{AHP20}).
\begin{defi}\label{def:surface tension} For a finite graph $G = (V, E)$, the surface tension between a pair of disjoint
subsets $A_1, A_2 \subset V$ with the external field $\eps h$ at temperature $T$ is defined as
\begin{equation}\label{eq:def_T}
  \mathcal{T}_{A_1,A_2}( \eps h)  =   T  \cdot \log\left(
    \frac{\cZ^{+,+} \cdot \cZ^{-,-} } {\cZ^{+,-} \cdot \cZ^{-,+} }\right)\,,
\end{equation} 
where $ \cZ^{s_1,s_2} \equiv \cZ_{T,\G,\eps h}^{A_1, A_2; s_1,s_2}$ (for  $s_1,s_2 \in\{+,-\}$)  is the partition function 
at temperature $T$ with the boundary condition such that the spin values on $A_1$ are $s_1$ and the spin values on $A_2$ are $s_2$.
\end{defi}
In the rest of this subsection, we  continue to fix $T = T_c$. For $\Gamma_1 \subset \Gamma_2\subset\mathbb Z^2$,  define  
\begin{equation}\label{eq:D_l_def}
  D_{\Gamma_1,\Gamma_2}(\eps  h)  = \frac{1}{2} \sum_{v\in \Gamma_1} \left[\langle \sigma_v \rangle_{\Gamma_2,\eps h}^{+} - \langle \sigma_v \rangle_{\Gamma_2,\eps h}^{-} \right]\,.
\end{equation}
Recall that $\mathcal{D}_{\intB \Gamma_2} = \mathcal{D}_{\intB \Gamma_2} (\bar{\sigma}^+,\bar{\sigma}^-)$
is the connected component of $\intB \Gamma_2$ in the pre-disagreement set $\mathcal{D}$. 
By Proposition~\ref{prop:DisagreementRep}, we get that
 \begin{equation*}\label{eq:disagreement representation for D}
D_{\Gamma_1,\Gamma_2}(\eps  h) = \langle |\Gamma_1 \cap \mathcal{D}_{\intB\Gamma_2}|  \rangle_{\Gamma_{2},\eps h}^{\intB \Gamma_2,+/-}\,.
 \end{equation*}

The following lemma connects the aforementioned surface tension and the crossing probability in an annulus. We follow the notation in Lemma~\ref{lem:crossing}.

\begin{lem}\label{lem: upper-bound for the log partition function}
There exists a constant $\nc\label{surface}>0$ such that for any $\ell \geq 1$ the following holds for $G=\Lambda_{\ell, 2\ell}$ and for any external field $\eps h$ on $G$:\begin{equation}\label{eq: connection between free energy and crossing probability}
    \mathcal{T}_{{\intB\Lambda_{\ell+1},\intB\Lambda_{2\ell}}}(\eps h)=-T_c\cdot\log(1-\langle \1_{\con(\ell,2\ell)} \rangle_{G, \eps h}^{+/-, +/-})\le \oc{surface}.
\end{equation}
\end{lem}
\begin{proof}
    By \eqref{eq:def_T}, in order to prove the equality in \eqref{eq: connection between free energy and crossing probability} it suffices to prove that $$\frac{\cZ^{+,+} \cdot \cZ^{-,-} } {\cZ^{+,-} \cdot \cZ^{-,+} }=\frac{1}{1-\langle \1_{\con(\ell,2\ell)} \rangle_{G, \eps h}^{+/-, +/-}}.$$
    Recall that $G^{\prime}=G/\intB\Lambda_{\ell+1}$ where we contract $\intB\Lambda_{\ell+1}$ into a single point denoted as $\fQ$ and recall that the partition function of the Ising model equals to that of the extended Ising model. Then we have $$\frac{\cZ^{+,+}}{\cZ^{{-,+}}}=\frac{\langle \1_{\{\bar \sigma_{\fQ}=1\}} \rangle_{G^{\prime},\eps h}^{{\mathbf{0},+}}}{\langle \1_{\{\bar \sigma_{\fQ}=-1\}} \rangle_{G^{\prime},\eps h}^{{\mathbf{0},+}}},~~~~~\frac{\cZ^{-,-}}{\cZ^{{+,-}}}=\frac{\langle \1_{\{\bar \sigma_{\fQ}=-1\}} \rangle_{G^{\prime},\eps h}^{{\mathbf{0},-}}}{\langle \1_{\{\bar \sigma_{\fQ}=1\}} \rangle_{G^{\prime},\eps h}^{{\mathbf{0},-}}}.$$
    Here we slightly abused the notation $h$ so that it serves as the external field on both $G$ and $G'$. Thanks to the independence, we can write \begin{equation*}\label{eq: expansion for free energy to disagreement percolation}
        \frac{\cZ^{+,+} \cdot \cZ^{-,-} } {\cZ^{+,-} \cdot \cZ^{-,+} }=\frac{\langle \1_{\{\bar \sigma_{\fQ}=1\}} \rangle_{G^{\prime},\eps h}^{{\mathbf{0},+}}}{\langle \1_{\{\bar \sigma_{\fQ}=-1\}} \rangle_{G^{\prime},\eps h}^{{\mathbf{0},+}}}\cdot \frac{\langle \1_{\{\bar \sigma_{\fQ}=-1\}} \rangle_{G^{\prime},\eps h}^{{\mathbf{0},-}}}{\langle \1_{\{\bar \sigma_{\fQ}=1\}} \rangle_{G^{\prime},\eps h}^{{\mathbf{0},-}}}=\frac{\langle \1_{\{\bar \sigma_{\fQ}^+=1,\bar \sigma_{\fQ}^-=-1\}} \rangle_{G^{\prime},\eps h}^{{\mathbf{0}/\mathbf{0},+/-}}}{\langle \1_{\{\bar \sigma_{\fQ}^+=-1,\bar \sigma_{\fQ}^-=1\}} \rangle_{G^{\prime},\eps h}^{{\mathbf{0}/\mathbf{0},+/-}}}.
    \end{equation*}
    Combined with \eqref{eq:crossing prime -/+}, it yields the desired equality in \eqref{eq: connection between free energy and crossing probability}.

    For the inequality in \eqref{eq: connection between free energy and crossing probability}, Lemma~\ref{lem:crossing} {and Remark~\ref{rmk:RSW-ref} show} directly that $\langle \1_{\con(\ell,2\ell)} \rangle_{G^{\prime}, \eps h}^{+/-, +/-}$ is bounded away from $1$ uniformly for all external field, yielding the desired bound.
\end{proof}
Provided with Lemma~\ref{lem: upper-bound for the log partition function}, we can prove Proposition \ref{prop: strictly faster polynomial decay}.
\begin{proof}[Proof of Proposition \ref{prop: strictly faster polynomial decay}]
    We start from the relation between the surface tension $\cT$ and the average number of disagreements. Applying \cite[Theorem 3.2] {AP19}, we get that \begin{equation}\label{eq: derivative of the log partition function}        \E\cT_{\intB\Lambda_{N+1},\intB\Lambda_{2N}}(\eps  h) = \frac{2 \eps  }{N}\, \E \left(  \frac{D_{\Lambda_N,\Lambda_{2N}}(\eps  h) }{ \varphi(\widehat h) } \right) \ge  \frac{C_1\eps}{N}\E D_{\Lambda_N,\Lambda_{2N}}(\eps  h)\,, 
    \end{equation} where $\varphi$ is the Gaussian density function $\varphi(s) = \frac{1}{\sqrt{2\pi}}e^{-s^2/2}$ and $\widehat h =\frac{\sum_{v\in \lamn }h_v}{{2}N}$. The inequality holds since $\varphi$ is bounded. We remark here that the expectation $\E$ is taken over the external field on $\Lambda_{2N}$ although $\cT_{\intB\Lambda_{N+1},\intB\Lambda_{2N}}$ is only a function of the external field in the annulus $\Lambda_{N,2N}$. By definition of $D$ in \eqref{eq:D_l_def} and CBC, we obtain that $\E D_{\Lambda_N,\Lambda_{2N}}(\eps  h)\ge 
    {4}N^2m(T_c,3N,\eps)$. Combined with \eqref{eq: derivative of the log partition function} and Lemma~\ref{lem: upper-bound for the log partition function}, it yields that $$m(T_c,3N,\eps)\le \frac{1}{{4}N^2}\E D_{\Lambda_N,\Lambda_{2N}}(\eps  h)\le \frac{\E\cT_{\intB\Lambda_{N+1},\intB\Lambda_{2N}}(\eps  h)}{{4}C_1\eps N}\le \frac{\oc{surface}}{{4}C_1\eps N}.$$ Thus we complete the proof.
\end{proof}

\begin{proof}[Proof of Lemma~\ref{lem: dis crossing probability upper bound}]
   In the following proof, we assume that $h\in\cH_{\mathtt {perf}}$. Let $\cA$ denote the event that there exists a disagreement crossing from $\intB\Lambda_{N}$ to $\intB\Lambda_{N/2+1}$ that intersects with at least $(\frac{N}{M})^{1+\alpha}$ $M$-boxes. Recalling Definition~\ref{def: dis crossing intersecting a large number of M-boxes}, we obtain that \begin{equation}\label{eq: crossing that intersects with lot of M-boxes}
   \mbb E\bar\mu_{\Lambda_{N/2,N},\eps h}^{+/-,+/-}(\con(N/2,N))\le \mbb E\bar\mu^{+/-,+/-}_{\Lambda_{N/2,N},\eps h}(\cA)+\mathbb{E}\left(\langle \1_{\cE_{\alpha,N,M}}\rangle^{+/-{,+/-}}_{\Lambda_{N/2,N},\eps h}\right).
    \end{equation} Let $\cA'$ denote the event that there exists a disagreement crossing from $\intB\Lambda_{N}$ to $\intB\Lambda_{N/2+1}$ that intersects with at least $\frac{1}{2}(\frac{N}{M})^{1+\alpha}$ good $M$-boxes. By Definition \ref{def: perfect external field} and our assumption that $h\in H_{\mathtt {perf}}$, we have $\cA\subset\cA'$. Furthermore, we define $\cA_i$ to be the event that there exists a disagreement crossing from $\intB\Lambda_{N}$ to $\intB\Lambda_{N/2+1}$ that intersects with $B_i$. Let $\{B_i\}_{i\in I_0}$ denote the collections of $M$-boxes in $\Lambda_{N/2,N}$. Then, on the one hand, we obtain that \begin{equation}\label{eq: lower bound on the number of good boxes traversed by dis crossing}
        \sum_{{i}\in I_0}\bar\mu^{+/-,+/-}_{\Lambda_{N/2,N},\eps h}(\cA_i){\1_{\{B_i\text{ is good}\}}}\ge \frac{1}{2}(\frac{N}{M})^{1+\alpha}\bar\mu^{+/-,+/-}_{\Lambda_{N/2,N},\eps h}(\cA')\ge\frac{1}{2}(\frac{N}{M})^{1+\alpha}\bar\mu^{+/-,+/-}_{\Lambda_{N/2,N},\eps h}(\cA).
    \end{equation}
    On the other hand, we can give a lower bound on the number of disagreements in each $M$-box $B_i$ in terms of the probability of $\cA_i$ using the FKG property as in Lemma \ref{lem:FKG} and the RSW estimate as in Theorem \ref{thm: RSW for dis with external field}. Recall that $\aro(u;M)$ is defined to be the event that there exists a pre-disagreement circuit in $\Lambda_{M,2M}(u)$. Let $\aro_i=\aro(u_i;M)$ where $u_i$ is the center of $B_i$. Then for a good $M$-box $B_i$, we have that $\bar\mu^{-/+}_{\Lambda_{5M}(u_i),\eps h}(\aro_i)\ge \oc{9}^4$. By CBC' (Lemma~\ref{lem:CBC prime}) and CBC (Corollary~\ref{cor: CBC}), we obtain that for a good $M$-box $B_i$\begin{align}
        \langle|\cD_{\partial}\cap B_i| ~\Big|~{\aro_i\cap\cA_i}\rangle^{+/-,+/-}_{\Lambda_{N/2,N},\eps h}&\stackrel{\mbox{CBC'}}\ge \langle|\cD_{\partial\Lambda_{2M}(u_i)}\cap B_i| \rangle^{+/-}_{\Lambda_{2M}(u_i),\eps h}\nonumber\\&\stackrel{\mbox{CBC}}\ge \sum_{u\in B_i}\langle\1_{u\in \cD_{\partial\Lambda_{4M}(u)}}\rangle^{+/-}_{\Lambda_{4M}(u),\eps h}.\label{eq: outmost disagreement boundary}
    \end{align}
    Recall the definition of (ii)-good boxes, and let $J_i$ denote the collection of vertices in $B_i$ satisfying \eqref{eq: def of (ii)-good boxes}. Thus, we have $|J_i|\ge 2M^2$. Applying \eqref{eq: def of (ii)-good boxes} and Proposition \ref{prop:DisagreementRep}, we obtain that \begin{align}
      \sum_{u\in B_i}\langle\1_{u\in \cD_{\partial\Lambda_{4M}(u)}}\rangle^{+/-}_{\Lambda_{4M}(u),\eps h}\ge~&\sum_{u\in J_i}\langle\1_{u\in \cD_{\partial\Lambda_{4M}(u)}}\rangle^{+/-}_{\Lambda_{4M}(u),\eps h}\nonumber\\\stackrel{\eqref{eq: def of (ii)-good boxes}}\ge& \sum_{u\in J_i}\frac{1}{2}\langle\1_{u\in \cD_{\partial\Lambda_{4M}(u)}}\rangle^{+/-}_{\Lambda_{4M}(u),0}\nonumber\\=~& \sum_{u\in J_i}\frac{1}{2}\Big(\langle\sigma_u\rangle^{+}_{\Lambda_{4M}(u),0}-\langle\sigma_u\rangle^{-}_{\Lambda_{4M}(u),0}\Big)~~~\text{(by Proposition \ref{prop:DisagreementRep})}\nonumber\\\ge~& M^2\Big(\langle\sigma_{o}\rangle^{+}_{\Lambda_{4M}(o),0}-\langle\sigma_{o}\rangle^{-}_{\Lambda_{4M}(o),0}\Big)~~~(\text{since }|J_i|\ge 2M^2)\nonumber\\ \stackrel{\eqref{eq: origin decay rate without disorder}}\ge& C_1 M^{\frac{15}{8}}.\label{eq: quenched thm1.2 application}
    \end{align}
    Thanks to FKG (Lemma~\ref{lem:FKG}), we obtain that for a good $M$-box $B_i$ 
        \begin{align*}
            \langle|\cD_{\partial}\cap B_i|\rangle^{+/-,+/-}_{\Lambda_{N/2,N},\eps h}\ge~& \langle|\cD_{\partial}\cap B_i|\cdot \1_{\cA_i\cap\aro_i}\rangle^{+/-,+/-}_{\Lambda_{N/2,N},\eps h}\nonumber\\ \ge~&C_1M^{\frac{15}{8}}\bar\mu^{+/-,+/-}_{\Lambda_{N/2,N},\eps h}(\aro_i\cap\cA_i)~~~(\text{by \eqref{eq: outmost disagreement boundary} and \eqref{eq: quenched thm1.2 application}})\nonumber\\\ge~&C_1M^{\frac{15}{8}}\bar\mu^{+/-,+/-}_{\Lambda_{N/2,N},\eps h}(\aro_i)\cdot\bar\mu^{+/-,+/-}_{\Lambda_{N/2,N},\eps h}(\cA_i)~~~(\text{by Lemma~\ref{lem:FKG}})\nonumber\\ \stackrel{\mbox{CBC}}\ge&C_1M^{\frac{15}{8}}\bar\mu^{-/+}_{\Lambda_{5M}(u_i),\eps h}(\aro_i)\cdot\bar\mu^{+/-,+/-}_{\Lambda_{N/2,N},\eps h}(\cA_i)\nonumber\\\ge~&C_2M^{\frac{15}{8}}\cdot\bar\mu^{+/-,+/-}_{\Lambda_{N/2,N},\eps h}(\cA_i)~~~(\text{since }B_i\text{ is good}).\label{eq: disagreement lower bound in a good box}
        \end{align*}
    Combined with \eqref{eq: lower bound on the number of good boxes traversed by dis crossing}, it yields that \begin{equation*}\label{eq: lower bound on the disagreement number in an annulus}
        \langle|\cD_{\partial}\cap \Lambda_{N/2,N}|\rangle^{+/-,+/-}_{\Lambda_{N/2,N},\eps h}\ge \frac{C_2}{2}M^{\frac{15}{8}}\cdot (\frac{N}{M})^{1+\alpha}\bar\mu^{+/-,+/-}_{\Lambda_{N/2,N},\eps h}(\cA).
    \end{equation*}
    Now averaging over $h\in \mathcal H_{\mathtt {perf}}$, we obtain that 
        \begin{align}
            \E\langle|\cD_{\partial}\cap \Lambda_{N/2,N}|\rangle^{+/-,+/-}_{\Lambda_{N/2,N},\eps h}&\ge \frac{C_2}{2}M^{\frac{15}{8}}\cdot (\frac{N}{M})^{1+\alpha}\E[\1_{h\in\cH_{\mathtt {perf}}}\bar\mu^{+/-,+/-}_{\Lambda_{N/2,N},\eps h}(\cA)]\nonumber\\&\ge \frac{C_2}{2}M^{\frac{15}{8}}\cdot (\frac{N}{M})^{1+\alpha}\Big(\E\bar\mu^{+/-,+/-}_{\Lambda_{N/2,N},\eps h}(\cA)-\P(\cH_{\mathtt {perf}}^c)\Big).\label{eq: lower bound in average on the disagreement number in an annulus}
        \end{align}
    By Proposition \ref{prop: strictly faster polynomial decay}, we have the following upper bound on the left-hand side above:\begin{equation}\label{eq: upper bound on the disagreement number in an annulus}
        \E\langle|\cD_{\partial}\cap \Lambda_{N/2,N}|\rangle^{+/-,+/-}_{\Lambda_{N/2,N},\eps h}\le \oc{prior}\frac{N}{\eps}.
    \end{equation}
    Combined with \eqref{eq: lower bound in average on the disagreement number in an annulus}, it yields that \begin{equation}\label{eq: Event A probability}
        \E\bar\mu^{+/-,+/-}_{\Lambda_{N/2,N},\eps h}(\cA)\le C_3\frac{N}{\eps M^{\frac{15}{8}}}\cdot(\frac{N}{M})^{-1-\alpha}+\P(\cH_{\mathtt {perf}}^c).
    \end{equation}
    Combining \eqref{eq: crossing that intersects with lot of M-boxes} and \eqref{eq: Event A probability}, we complete the proof of the lemma.
\end{proof}

\begin{proof}[Proof of Lemma~\ref{lem: domination by product measure}.]

Recall Definition~\ref{def: good external field for RSW}
and Theorem~\ref{thm: RSW for dis with external field}. Taking a union bound and using the FKG property for the disagreement Ising model, we obtain that for {$\eps\leq\oc{7}M^{-\frac{7}{8}}$},
\begin{equation}\label{eq:1-good-estimation}
    \mbb P(h|_{\Lambda_{5M}(u)} \mbox{ is (i)-good})\geq 1-4\oc{8}\exp(-\oc{8}^{-1}\sqrt{\eps^{-1} M^{-\frac{7}{8}}}).
\end{equation}

Now in order to estimate the probability of being (ii)-good, fix an arbitrary box $\Lambda_{5M}(u)$. Define
$$J=\Big\{x\in\Lambda_M(u):\langle\1_{x\stackrel{\mcc D}\longleftrightarrow\intB\Lambda_{4M}(x)}\rangle^{+/-}_{\Lambda_{4M}(x),\eps h}\geq \frac{1}{2}\langle\1_{x\stackrel{\mcc D}\longleftrightarrow\intB\Lambda_{4M}(x)}\rangle^{+/-}_{\Lambda_{4M}(x),0}\Big\}.$$
Then $J$ is a random set measurable with respect to $h|_{\Lambda_{5M}(u)}.$

By taking $\theta=\frac{1}{2}$ and $N = {4}M$ in Theorems~\ref{thm-critical-temperature-small-perturbation}, for $\eps\le\oldconstant{-3}{(4M)}^{-\frac{7}{8}}$ we have
$$\mbb P(x\in J)\geq 1-\oc{-4}\exp \left(-\oc{-4}^{-1}\sqrt{\eps^{-1} {M}^{-\frac{7}{8}}}\right ).$$
Combined with Markov's inequality, it yields that
\begin{equation}\label{eq:2-good-estimation}
   \mbb P(h|_{\Lambda_{5M}(u)} \mbox{ is (ii)-good})= \mbb P(|J|\geq 2M^2)\geq 1-2\oc{-4}\exp \left(-\oc{-4}^{-1}\sqrt{\eps^{-1} {M}^{-\frac{7}{8}}}\right ).
\end{equation}
Combining \eqref{eq:1-good-estimation} and \eqref{eq:2-good-estimation}, we have that for $\eps \leq \min\{\oc{7},{\frac{\oc{-3}}{4}}\}M^{-\frac{7}{8}}$
\begin{equation}\label{eq:good-estimation}
    \mbb P(h|_{\Lambda_{5M}(u)} \mbox{ is good})\geq 1-C\exp(-C^{-1}\sqrt{\eps^{-1} {M}^{-\frac{7}{8}}})
\end{equation}
for some constant $C>0$ depending on $\oc{8}$ and $\oc{-4}$.

Recalling Definition~\ref{def: domination}, we have
$$\mbb P_{G_M}(\zeta_u=1)\geq1-C\exp(-C^{-1}\sqrt{\eps^{-1} {M}^{-\frac{7}{8}}}).$$
Furthermore, the events $ \{h|_{\Lambda_{5M}(u)} \mbox{ is good}\}$ and $ \{h|_{\Lambda_{5M}(v)} \mbox{ is good}\}$ are independent once {we have} $dist(u,v)>10M.$ That is, $$\mbb P_{G_M}(\zeta_i=1\mid \{\zeta_j: |i-j|>10\})=\mbb P_{G_M}(\zeta_i=1).$$
By \cite[Theorem 0.0]{LSS97}, we complete the proof of Lemma~\ref{lem: domination by product measure} by choosing $\oc{dom}>0$ small enough.
\end{proof}

\begin{proof}[Proof of Lemma~\ref{lem: perfect external field bound}]
Recalling Definition \ref{def: domination}, we first specify the following notion of coarse graining: for each collection of $M$-boxes $\mathcal B$, we say $\mathcal B$ is a crossing if there is a disagreement crossing $\mathcal C$ connecting the outer and inner boundaries of $\Lambda_{N/2,N}$ such that an $M$-box $B_i\in \mathcal B$ if and only if $B_i \cap \mathcal C \neq \emptyset$. As a result, we see that any $M$-box crossing $\mathcal B$ is *-connected in the graph $G_M$, where $B_i$ is *-neighboring $B_j$ if their $\ell_\infty$-distance is at most $1$.  Denote by $\mathfrak B_k$ the collection of all $M$-box crossings with $k$ boxes.  

We now give an upper bound on $\mathfrak B_k$. Note that any crossing in $\mathfrak B_k$ is a 
*-connected subgraph on $G_M$ with at least one box intersecting $\partial \Lambda_{N}$ which we denote as the root box. Therefore, by considering a depth-first search procedure, each such crossing can be encoded by a *-connected contour starting from and ending at the root box, and the length of the contour is $2k-2$ (this is because the depth-first search procedure results in a subtree of $G_M$ and each edge in the tree is visited twice). Since each $M$-box has at most $8$ *-neighbors, we obtain that

    
    \begin{equation}\label{eq: number bound for crossing with length k}
           |\mathfrak B_k|\leq  \frac{2N}{M}\times 8^{2k-2}. 
    \end{equation}
    Now, for each $\mathcal B\in \mathfrak B_k$, we let $\mathcal E_\mathcal B$ be the event that at least half of the boxes in $\mathcal B$ are good. At this point, we choose $\oc{perf1}>0$ small enough such that 
    $\P_{G_M}$ is dominated below by $\P_{{\rho}}$ {for some $\rho\in(0,1)$ close to 1  such that} for any $\mathcal B\in \mathfrak B_k$ we have
\begin{equation}\label{eq: single bad crossing probability}
    \P(\mathcal E_\mathcal B^c) \leq \exp(-C_1k),
\end{equation} for some constant $C_1>10$.
Therefore, combining \eqref{eq: number bound for crossing with length k} and \eqref{eq: single bad crossing probability}, we obtain that
$$\sum_{k\geq (\frac{N}{M})^{1+\alpha}} \sum_{\mathcal B \in \mathfrak B_k} \P(\mathcal E_\mathcal B^c) \leq \sum_{k\geq (\frac{N}{M})^{1+\alpha}}\frac{2N}{M}\times 8^{2k-2}\cdot \exp(-C_1k)\le C_2^{-1}\exp(-C_2(\frac{N}{M})^{1+\alpha}),$$
completing the proof of Lemma~\ref{lem: perfect external field bound}.
\end{proof}


\subsection{Proof ingredients of Theorem \ref{thm:main thm-lower bound}}\label{sec: lower-bound for main theorem}

In this subsection, we prove Lemmas~\ref{lem: good external field probability 0} and \ref{lem: good external field probability 1}.

\begin{proof}[Proof of  Lemma~\ref{lem: good external field probability 0}]
Let $\cH_{\star}^1$ and $\cH_{\star}^2$ be the collections of the external field $h$ satisfying the first and second items in Definition \ref{def: good external field 1}, respectively.
{By Lemma~\ref{lem: domination by product measure}, we let $\oc{gefp00}>0$ be small enough such that $\P_{G_M}$ is dominated below by $\P_{\rho}$} (where $\rho \geq 0.99$ is close to $1$ as specified below).
    Applying \cite[Lemma 4.4]{RY23}, we have that \begin{equation}\label{eq: bernoulli site percolation criterion 1}
        \P(\cH^2_{\star}(\oc{gefp01},N,N_1,M))\ge 1-C_1(\frac{N}{M})^4\exp(-\frac{C_2N}{M})-C_1(\frac{N}{M})^2\exp(-\frac{C_2N_1}{M}).
    \end{equation}
    The desired result follows since $N>N_1>M$. For $\cH^1_{\star}$, we consider a bond percolation model induced by $\P_{\rho}$, i.e., for each edge $e=(x,y)$, let $\zeta_e=1$ if and only if $ \zeta_x=\zeta_y=1$. Then applying \cite[Theorem 0.0]{LSS97} again and letting $\rho$ be sufficiently close to $1$, we obtain that $\P_{\rho}$ dominates $\Tilde\P_{0.99}$, where we denote by $ \tilde \P_{0.99}$ the Bernoulli bond percolation model with density $0.99$. Let $\cA$ denote the event that there exists an open path separating $\intB\Lambda_{2N_1}$ and $\intB\Lambda_{N_1{+1}}$. Then we have \begin{equation*}
        \P(\cH^1_{\star}(\oc{gefp01},N,N_1,M))\ge \Tilde\P_{0.9{9}}(\cA).
    \end{equation*} Combined with duality and a standard percolation estimate as in e.g. \cite[Theorem 3.3]{D18}, it yields that \begin{equation}\label{eq: bernoulli site percolation criterion 2}
         \P(\cH^1_{\star}(\oc{gefp01},N,N_1,M))\ge1-\Tilde\P_{0.{0}1}(\intB\Lambda_{2N_1}\longleftrightarrow\intB\Lambda_{N_1{+1}})\ge 1-C_3\exp(-C_3^{-1}\frac{N_1}{M}).
    \end{equation}
    Combining \eqref{eq: bernoulli site percolation criterion 1} and \eqref{eq: bernoulli site percolation criterion 2}, we finish the proof of this lemma.
\end{proof}

\begin{proof}[Proof of  Lemma~\ref{lem: good external field probability 1}]
Since $\eps^{-1}M^{-\frac{7}{8}}>(\eps^{\frac{8}{7}}N)^{\frac{1}{5}}$, by \eqref{eq:good-estimation} we have\begin{equation}\label{eq: good external field for RSW 2}
    \begin{aligned}
        &\P(\Lambda_M(u) \text{ is good})\ge 1-C\exp(-C^{-1}{\sqrt{\eps^{-1} M^{-\frac{7}{8}}}})\ge1-C\exp(-C^{-1}(\eps^{\frac{8}{7}}N)^{\frac{1}{10}}).
    \end{aligned}
\end{equation}

It is clear that $\Lambda_{2N_2}$ can be covered by a union of $(2N_2/M)^2$ many $M$-boxes, which we denote as $\mathfrak M$. Applying \eqref{eq: good external field for RSW 2}, we see that 
$$\P(\text{all boxes in } \mathfrak M \text{ are good}) \geq 1-\frac{4N_2^2}{M^2}\exp(-C_1^{-1}(\eps^{\frac{8}{7}}N)^{\frac{1}{10}})\ge 1-C_2\exp(-C_2^{-1}(\eps^{\frac{8}{7}}N)^{\frac{1}{10}})$$
where the last inequality follows from $N_2/M \leq \epsilon^{{\frac{8}{7}}} N$. Since both items in Definition \ref{def: good external field 2} hold if all boxes in $\mathfrak M$ are good, this completes the proof of the lemma.
\end{proof}

\section{Proof of Theorem~\ref{thm-supercritical-temperature}}\label{sect:correlation length}

In this section, we prove Theorem~\ref{thm-supercritical-temperature}, i.e., to compute the correlation length for $d = 2$ in the entire low temperature regime. The upper bound was shown in \cite{DW20}, and we included it in our theorem statement only for completeness.
Naturally, extending the lower bound on the correlation length for $d=2$ from the very low temperature regime (as proved in \cite{DW20,DZ21}) to the entire low temperature regime is highly similar to extending the long-range order for $d=3$ from the very low temperature regime (as proved in \cite{Imbrie85,BK88,DZ21})  to the entire low temperature regime as recently established in \cite{DLX22}. As a result, we will only provide a sketch of the proof with emphasis on the subtlety arising from $d=2$ (see Section~\ref{sec:2dsubtlety}). We will provide complete definitions necessary for presenting the sketch, but we refer to \cite{DLX22} for discussions on the underlying intuition of such definitions. Also, we omit details when proofs can be easily adapted.

\subsection{Coarse graining}

The coarse graining method in the context of the FK-Ising model was introduced in \cite{P96}, and in this subsection, we provide a brief review. We will employ the mathsf font to denote objects related to the coarse graining, in order for clarity of notations.
For an integer $\mathsf q\ge 1$ and a vertex $\mathsf v=(\mathsf v^1, \mathsf v^2)\in \mbb Z^2$, write $\mathsf q \mathsf v = (\mathsf q \mathsf v^1, \mathsf q \mathsf v^2)$ and  write  $\mathsf q\mbb Z^2 = \{\mathsf q \mathsf v: \mathsf v\in \mbb Z^2\}$. 
For a vertex $\mathsf v \in \mbb Z^2$,  define $\mathsf Q_{\mathsf v}$ to be the box of side length $6\mathsf q$ centered at $\mathsf q \mathsf v$, by
\begin{equation*}\label{eq-def-mathsf-Q}
\mathsf Q_{\mathsf v}=\mathsf q \mathsf v+[- 3\mathsf {q}, 3\mathsf {q}]^2\cap \mbb Z^2\,.
\end{equation*}
For  $\mathsf V\subset \mbb Z^2$, we define the following notations:
$$\mcc Q_{\mathsf V}=\{\mathsf Q_{\mathsf v}:  \mathsf v\in \mathsf V\} \mbox{ and } \mathsf Q_{\mathsf V} = \bigcup_{\mathsf v\in \mathsf V} \mathsf Q_{\mathsf v}\,.$$ Without loss of generality, for convenience of exposition we may assume that  $\mathsf q$ divides  $N$. Define $\mathsf N = \frac{N}{\mathsf q}-3$. We see that ${\mathsf Q}_{\Lambda_{\mathsf N}}=\Lambda_{N}$.
For an FK-Ising configuration $\omega$,  we say a box $\mathsf Q\in \mcc Q_{\Lambda_{\mathsf N}}$ is \textbf{good} in $\omega$ if $\omega|_{\mathsf Q}$ contains paths crossing the rightmost, leftmost, bottom and top $2\mathsf q\times 6\mathsf{q}$ (or $6\mathsf{q}\times 2\mathsf q$) boxes between the short sides. (See Figure~\ref{fig:good-box} for an illustration.) 
Here $\omega|_{\mathsf Q}$ is the restriction of $\omega$ on $\mathsf Q$ (or equivalently, the subgraph on $\mathsf Q$ induced by open edges in $\omega$). These four paths are connected in the box $\mathsf Q$, and they are contained in the same open cluster of $\omega$; we shall call it
the main cluster of the box. By planarity, it is not hard to see that the main cluster is unique for each good box. In addition, we say a vertex $\mathsf v \in \Lambda_{\mathsf N}$ is good if $\mathsf Q_{\mathsf v}$ is good.

\begin{figure}[htb]
    \centering
    \includegraphics[width=0.3\linewidth]{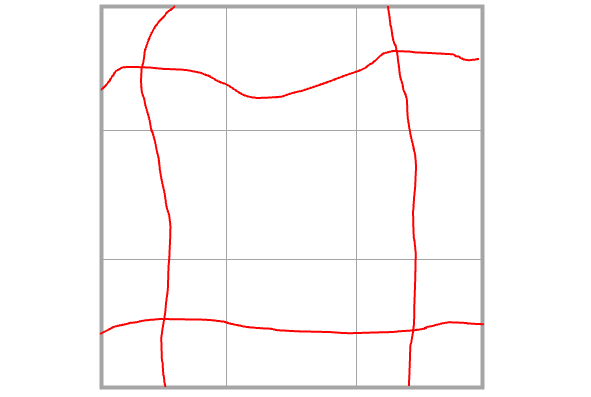}
    \caption{Good box. The red curves are open paths in $\omega|_{\mathsf Q}$.}
    \label{fig:good-box}
\end{figure}

As in \cite{DLX22}, we want to prove that a box $\mathsf{Q_v}$ is good with (conditional) probability close to 1, regardless of the configuration on edges outside a concentric box with larger and comparable size.  By the definition of good boxes, if  $\mathsf v_1, \ldots, \mathsf v_n$ is a neighboring sequence of good vertices, then the main clusters in 
$\mathsf Q_{\mathsf v_i}$ and $\mathsf Q_{\mathsf v_{i+1}}$ intersect for all $1\leq i < n$. Thus, this yields a cluster within $ \mathsf Q_{\{\mathsf v_1, \ldots, \mathsf v_n\}}$ that touches every side of every box in $\mcc Q_{\{\mathsf v_1, \ldots, \mathsf v_n\}}$.
\begin{prop}\label{prop-good box probability}
     For every $p>p_c$, there exists a constant $\mathsf c_{\mathrm g}>0$ such that for every $\mathsf q>0$ and for any FK-boundary condition $\upsilon$, 
\begin{equation*}\label{eq:goodboxprobability}
		\phi_{p, \Lambda_{4\mathsf q},{0}}^\upsilon[\Lambda_{3\mathsf q} \text{ is good }]\geq 1-e^{- \mathsf c_{\mathrm g} \mathsf q}.
	\end{equation*}
 \end{prop}
 
\begin{proof}
    For a rectangle $\cR$, denote $\cH(\cR),\cV(\cR)$ to be the event that there exists a horizontal crossing and a vertical crossing in $\cR$, respectively.
    By the FKG inequality and CBC, it suffices to prove an RSW type inequality
    \begin{equation*}
        \phi_{p, [-\mathsf q,3\mathsf q]\times[-\mathsf q, 7\mathsf q],{0}}^{\mathrm{f}}[\cV([0,2\mathsf q]\times[0,6\mathsf q])]\geq 1-e^{- \mathsf c_1 \mathsf q}.
    \end{equation*}
    In \cite{DT19}, a similar result was proved for square crossings for $p>p_c$. That is, 
    \begin{equation}\label{eq-easy crossing}
        \phi_{p, [-\mathsf q,3\mathsf q]\times[-\mathsf q, 3\mathsf q],{0}}^{\mathrm{f}}[\cV([0,2\mathsf q]\times[0,2\mathsf q])]\geq 1-e^{- \mathsf c_2 \mathsf q}.
    \end{equation}
    Applying \cite[Theorem 2]{KT23} (note here we may choose $\P^+=\P=\P^-$ to be the FK-measure induced by the free boundary condition), we can enhance \eqref{eq-easy crossing} to hard crossings, which completes the proof of Proposition~\ref{prop-good box probability}. Note that it is quite possible that this proof does not require the powerful RSW method as in \cite{KT23}, and we just applied it since it is convenient. 
\end{proof}

 \subsection{Proof of Theorem~\ref{thm-supercritical-temperature}}

 In this subsection, we define the outmost blue boundary following \cite[Section 3.1]{DLX22}. We fix a constant $\mathsf k\geq 1$ (we may simply take $\mathsf k = 4$, but we keep the notation $\mathsf k$ for conceptual clarity). We say two vertices are $\mathsf k$-neighboring with each other if their $\ell_{\infty}$-distance is at most $\mathsf k$. In addition, we say a set is $\mathsf k$-connected if it is connected with respect to the $\mathsf k$-neighboring relation, and we write $\mathrm{Ball}(\mathsf A; \mathsf k) = \{\mathsf v \in \Lambda_{\mathsf N}: d_{\infty}(\mathsf v, \mathsf A) \leq \mathsf k\}$ for $\mathsf k\geq 1$. For each set $\mathsf A$, denote by $\psi(\mathsf A)$ the collection of vertices $\mathsf v$ such that every path from $\mathsf v$ to $\partial \Lambda_{\mathsf N}$ intersects $\mathsf A$.  Furthermore, we follow a trick in \cite{P96} on controlling correlations and introduce auxiliary random variables $\mathrm{Aux} =  \{\mathrm {Aux}_{\mathsf v}: \mathsf v\in \Lambda_{\mathsf N}\}$ where $\mathrm {Aux}_{\mathsf v}$'s are i.i.d.\ Bernoulli variables so that they take value 1 with probability $\mathsf p_{\mathrm{aux}} = 1- e^{-\mathsf c_{\mathrm g} \mathsf q/{250} }$. 
 We say a vertex $\mathsf v$ is \textbf{blue} if either $\mathsf v$ is bad or $\mathrm {Aux}_{\mathsf v} = 0$; otherwise, we say $\mathsf v$ is \textbf{red}. We say $\mathsf B$ is a blue boundary if
$\mathsf B$ is blue (i.e., every vertex in $\mathsf B$ is blue), $\psi(\mathsf B)$ is a connected set containing the origin $\mathsf o$, and $\pari \psi(\mathsf B) = \mathsf B$. As an illustration, we have that $\{\mathsf o\}$ is a blue boundary if $\mathsf o$ is blue. We say a blue boundary $\mathsf B$ is the outmost blue boundary if $\psi(\mathsf B') \subset \psi(\mathsf B)$ for any other blue boundary $\mathsf B'$.  As a convention, we let the outmost blue boundary be $\emptyset$ if there is no blue boundary. For convenience, we let $\mathtt{Blue} \subset \Lambda_{\mathsf N}$ be the collection of blue vertices and define $\mathtt{Red} = \Lambda_{\mathsf N} \setminus \mathtt{Blue}$. For $\mathsf B, \mathsf B'\subset \Lambda_{\mathsf N}$, let $\Pi_{\mathsf B, \mathsf B'}$ be the collection of $\omega \times \mathrm{Aux}\in \Pi$ such that the following hold{:}
\begin{itemize}
\item $\mathsf B$ is the outmost blue boundary;
\item If $\mathsf B \neq \emptyset$, then $\mathsf B'$ is the collection of all blue vertices which can be $2\mathsf k$-connected to $\mathsf B$ via blue vertices; if $\mathsf B = \emptyset$, then $\mathsf B'$ is the collection of blue vertices which are $2\mathsf k$-connected to $\mathsf o$ in blue vertices. 
\end{itemize}
Let $\mathfrak B$ be the collection of pairs $(\mathsf B, \mathsf B')$ such that $\Pi_{\mathsf B, \mathsf B'} \neq \emptyset$, and recall that $\mathcal C_o$ is the FK-cluster of the origin. We denote by $\mathsf P_{\mathrm{aux}}$ the law for $\mathrm {Aux}$, and we let $\vfk+ = \fk+ \times\mathsf P_{\mathrm{aux}}$ and $\vfkh+ = \fkh+ \times \mathsf P_{\mathrm{aux}}$ be the product measures on configurations for $\omega \times \mathrm{Aux} \in \Pi = \Omega \times \{0, 1\}^{\Lambda_{\mathsf N}}$.

Provided with the definitions above, we next explain the main idea for proving Theorem~\ref{thm-supercritical-temperature}, which relies on a comparison between the FK-Ising measures without an external field and the FK-Ising measures with the external field. To this end, we write
\begin{equation}\label{eq-decomposition-0-epsilon-field}
    \mu^+_{T,\lamn,\eps h}(\sigma_o = -1)=\sum_{(\mathsf B, \mathsf B') \in \mathfrak B}\sum_{o\in C_o, C_o \subset \Lambda_{N-1}}\vfkh +(\Pi_{\mathsf B, \mathsf B'}\cap \{\mcc C_o=C_o\})\times \frac{\exp(-\frac{\eps}{T}h_{C_o})}{2\cosh(\frac{\eps}{T}h_{C_o})}\,.
\end{equation}
As in \cite{DLX22}, the following three results are of importance.
\begin{prop}\label{prop-compare-small-L}
For any $T<T_c$ and $\delta, \theta_1 > 0$, there exists $\mathsf q_0 = \mathsf q_0(T, \delta, \theta_1)$ such that for any $\mathsf q \geq \mathsf q_0$ the following holds. There exists $\eps_0=\eps_0(\mathsf q, \delta,\theta_1,T)$ such that for $\epsilon \leq \epsilon_0$, we have
 with $\mbb P$-probability at least $1-\delta$,
\begin{align*}
&\sum_{o\in C_o, C_o \subset \Lambda_{N-1}}\vfkh +(\Pi_{\emptyset, \emptyset}\cap \{\mcc C_o=C_o\})\frac{\exp(-\frac{\eps}{T}h_{C_o})}{2\cosh(\frac{\eps}{T}h_{C_o})} \\
&\leq e^{\theta_1}\sum_{o\in C_o, C_o \subset \Lambda_{N-1}}\vfk +(\Pi_{\emptyset, \emptyset}\cap \{\mcc C_o=C_o\})\times\frac{1}{2}\,.
\end{align*}
\end{prop}
\begin{lem}\label{lem:coarse-graining}
For any $T<T_c$, there exist $\mathsf q_0 = \mathsf q_0(T)$ such that for any $\mathsf q \geq \mathsf q_0$ the following holds for a constant $b=b(T, \mathsf q_0)>0$ \emph{not} depending on $N$ or $\mathsf q$:
    \begin{equation*}\label{eq:coarse-graining}
       \vfk+(\bigcup_{(\mathsf B, \mathsf B')\in \mathfrak B: |\mathsf B \cup \mathsf B'| = L} \Pi_{\mathsf B, \mathsf B'})\leq \exp(-b \mathsf q L). 
    \end{equation*}
\end{lem}
\begin{prop}\label{prop-large-surface-estimation}
For any $0<T<T_c$, there exist constants $c>0,\ \mathsf q_0 = \mathsf q_0(T,\delta)$ and $\eps_0=\eps_0(\mathsf q,\delta,T)$ with $\mathsf q \geq \mathsf q_0$ such that for any $\epsilon \leq \epsilon_0$ and $N\le e^{c\epsilon^{-4/3}}$, the following holds
 with $\mbb P$-probability at least $1- \delta$:
\begin{align*}
    \vfkh +(\Pi_{\mathsf B, \mathsf B'})\leq e^{200 \mathsf k^2 L}\times \vfk+(\Pi_{\mathsf B, \mathsf B'}) \mbox { for all } (\mathsf B, \mathsf B') \in \mathfrak B \mbox{ with } |\mathsf B \cup \mathsf B'| = L \geq 1\,.
\end{align*}
\end{prop}
The proofs of Proposition~\ref{prop-compare-small-L} and Lemma~\ref{lem:coarse-graining} are identical to that for \cite[Propositions 3.1 and 3.2]{DLX22}, since dimension does not play a role in these proofs. In contrast, while Proposition~\ref{prop-large-surface-estimation} is an analog of \cite[Proposition 3.3]{DLX22}, the proof in \cite{DLX22} cannot be extended trivially since dimension does play an important role here (note that we have an additional assumption of $N\leq e^{c\epsilon^{-4/3}}$). We  will discuss the additional subtleties for proving Proposition~\ref{prop-large-surface-estimation} in the next subsection.

\begin{proof}[Proof of Theorem~\ref{thm-supercritical-temperature}]
    Provided with Proposition \ref{prop-compare-small-L}, Lemma \ref{lem:coarse-graining} and Proposition \ref{prop-large-surface-estimation}, the proof of Theorem \ref{thm-supercritical-temperature} is essentially the same as that of \cite[Proposition 2.2]{DLX22}, and thus we omit further details.
\end{proof}

\subsection{Proof of Proposition~\ref{prop-large-surface-estimation}}\label{sec:2dsubtlety}

The proof of  Proposition~\ref{prop-large-surface-estimation} is an extension of that for \cite[Proposition 3.3]{DLX22}. As a crucial difference, for $d=2$ without an assumption on the relation between $N$ and $\epsilon$, the change of the log-partition function after flipping the external field on a set $A$ cannot be uniformly bounded by $O(|\intB(A)|)$. This is exactly what leads to an interesting behavior for the correlation length, and this is why we posed an additional assumption of $N\leq e^{c\epsilon^{-4/3}}$ in  Proposition~\ref{prop-large-surface-estimation}. The key ingredient for $d = 2$ is the following lemma.
\begin{lem}\label{lemma:upper partition}
    For any constant $\nc\label{6-00}>0$, there exist constants $\nc\label{6-3}=\oc{6-3}(\oc{6-00}),\nc\label{6-4}=\oc{6-4}(\oc{6-00})>0$ such that
    if $N\leq e^{\oc{6-3}\epsilon^{-4/3}}$, then with $\P$-probability at least $1 - \exp(-\frac{\oc{6-4}}{\epsilon^2})$ the following holds. For all connected set $A \in \lamn$,
    \begin{equation*}
        \mathcal{Z}_{p,\epsilon h^A}^{+,\phi}\le e^{\frac{\oc{6-00}|\intB A|}{T}}\cdot\mathcal{Z}_{p,\epsilon h}^{+,\phi}
    \end{equation*}
where $h^A$ denotes the external field obtained from $h$ by flipping the signs of external field on $A$.
\end{lem}
Lemma~\ref{lemma:upper partition} follows by adapting the proof in \cite[Section 3.3]{DZ21}. Note that Lemma~\ref{lemma:upper partition} removes a $\log(1/\epsilon)$ term (for the bound on $N$) in \cite{DZ21}; this is because the new version of \cite[Proposition 2.2]{DW20} (which follows essentially from \cite[Theorem 4.4.2]{Talagrand14}) is sharper than its older version, which was the one used in \cite{DZ21}. Provided with Lemma~\ref{lemma:upper partition}, the proof of Proposition~\ref{prop-large-surface-estimation} is a straightforward adaption for the proof of \cite[Proposition 3.3]{DLX22}. Note that in \cite[Lemma 3.6]{DLX22} we obtained and applied a bound on the ratio between the partition functions (before and after a carefully designed flipping for the disorder), where the bound is $e^{\frac{81\mathsf{k}^3 \mathsf{q}^3 \sqrt{\epsilon}}{T} |\mathsf B \cup \mathsf B'|}$. But a slightly more careful checking reveals that one does not need the factor of $\sqrt{\epsilon}$ in the exponent as in \cite[Lemma 3.6]{DLX22}; instead a bound of $e^{\oc{6-00} |\mathsf B \cup \mathsf B'|}$ for some small enough but fixed constant $\oc{6-00}>0$ as in Lemma \ref{lemma:upper partition} is sufficient. Once we proved an analog of \cite[Lemma 3.6]{DLX22} (despite the fact that the bound in the current context is weaker as commented earlier), it is straightforward to adapt the proof of \cite[Proposition 3.3]{DLX22} and obtain a proof of Proposition~\ref{prop-large-surface-estimation} as required.

\appendix
\renewcommand{\appendixname}{Appendix~\Alph{section}}
\section{Truncated Ising correlation at criticality}
In this appendix, we continue to consider $T = T_c$. In addition, all discussions in this appendix refer to the case without disorder, and thus for simplicity, we drop the subscript $0$ (which indicates an $0$-disorder) in notations.

For two points away from the boundary, we first provide a lower bound on the connectivity probability  in the FK-Ising model, and our bound holds uniformly for all boundary conditions. 
\begin{defi}
For $\mcc P, \mcc M \subset \Gamma$, recall that $\phi^{\mcc P, \mcc M}_\Gamma$ is the FK-Ising measure on $\Gamma$ with wired boundary conditions on each of $\mcc P$ and $\mcc M$ (but $\mcc P$ is not wired to $\mcc M$). In this appendix, we will also often refer to the Ising boundary condition $\xi\in \{-1, 1\}^{\mcc P\cup \mcc M}$ where $\xi_u = 1$ for $u\in \mcc P$ and $\xi_u = -1$ for $u \in \mcc M$. We remark that throughout the appendix the relation between $\xi$ and $(\mcc P, \mcc M$) is fixed as such.
\end{defi}
Recalling the definition of $\phi^{\xi}_{\Gamma}$ as in \eqref{eq-def-FK-external-field}, we can see that (since there is no disorder)
\begin{equation}\label{eq:definition for P and M}
    \phi^{\xi}_\Gamma = \phi^{\mcc P, \mcc M}_\Gamma (\cdot \mid \mcc P \centernot\longleftrightarrow \mcc M)\,.
\end{equation}

Before stating Lemma~\ref{lem: FK connecting probability with Ising boundary}, we review the dual theory for the FK-Ising model initiated in \cite{KG41}, and our presentation follows that in \cite{BDC12}. Define $(\mbb Z^2)^{\diamond}=\mbb Z^2+(\frac{1}{2},\frac{1}{2})$ to be the dual of $\mbb Z^2$, which can be viewed as a translation of $\mbb Z^2$ in $\mbb R^2$. We see that every vertex in $(\mbb Z^2)^\diamond$ is the center of a unit square in $\mbb Z^2$ and vice versa.  In other words, every edge $e\in \mbb Z^2$ intersects with a unique edge  $e^{\diamond} \in (\mbb Z^2)^{\diamond}$. For a configuration $\omega\in\{0,1\}^{\edge(\mbb Z^2)}$, we define its dual configuration $\omega^{\diamond}$ by setting
$$\omega^{\diamond}(e^\diamond)=1-\omega(e), \quad \mbox{ for all } e\in \edge(\mbb Z^2).$$
Importantly, the measure of $\omega^{\diamond}$ is also an FK-Ising measure with a dual boundary condition and dual parameter $p^{\diamond}$ (see \cite{BDC12} for more details), where by \cite{Onsager44} (see also \cite{ABF87, BDC12, BDC12b})
$$(p_c)^{\diamond}=p_c=\frac{\sqrt{2}}{1+\sqrt 2}.$$
Now for a rectangle $\cR=[a,b]\times[c,d]$, let $\mcc H^\diamond(\cR)$ denote the event that there exists an open dual path crossing $\cR$ horizontally.
 This is a slight abuse of notation since the dual path lives in $[a-\frac{1}{2},b+\frac{1}{2}]\times [c+\frac{1}{2},d-\frac{1}{2}]\subset (\mbb Z^2)^{\diamond}$. Similarly, we  define the event of vertical dual crossing $\mcc V^\diamond(\cR)$, and we  also define the event $\mcc H(\cR)$ and $\mcc V(\cR)$ which correspond to crossings by $\omega$-paths. By duality, we have $\mcc V(\cR)$ happens if and only if $\mcc H^\diamond(\cR)$ fails, and $\mcc H(\cR)$ happens if and only if $\mcc V^\diamond(\cR)$ fails.
\begin{lem}\label{lem: FK connecting probability with Ising boundary}
     Let $\cR$ be a rectangle with side lengths at least $M$ and at most $10M$. Let $\cR\subset\Gamma\subset \mathbb Z^2$. Let $\xi$ be an arbitrary Ising boundary condition on $\intB\Gamma$. For any $u,v\in \cR$ satisfying $dist(u,\intB \cR)>M/10$ and $dist(v,\intB \cR
)>M/10$, we have \begin{equation*}
        \phi^{\mcc P,\mcc M}_{\Gamma}(u\longleftrightarrow v\centernot\longleftrightarrow \mcc P\cup \mcc M\mid \mcc P\centernot\longleftrightarrow \mcc M)
        \ge \oc{A3}\cdot dist(u,v)^{-\frac{1}{4}},
    \end{equation*}for some $\nc\label{A3}>0$ not depending on the location of $u,v$ or the boundary condition $\xi$.
\end{lem}
\begin{proof}
Suppose $\cR$ is a rectangle with dimension $M_1\times M_2$.
Let $\cRp$ denote the rectangle with the same center as $\cR$ but has dimension $(M_1-0.1M)\times (M_2-{0.1}M)$.
 Let $\cR_L,\cR_B,\cR_R,\cR_T$ be the left-most, bottom, right-most, top rectangles in $\cR$ with width $0.05M$ such that $\cR^{\prime}=\cR\setminus(\cR_L\cup\cR_B\cup\cR_R\cup\cR_T)$. Let $\cB$ be the event that there is a dual path connecting each of the two short sides of $\cR_L,\cR_B,\cR_R,\cR_T$ respectively.
 See figure
 \ref{fig:TBLR}
 for illustration.

\begin{figure}[htb]
    \centering
    \begin{tikzpicture}    
    \draw [very thick] (-4,-2)rectangle (4,2);
    \draw [very thick] (-3.5,-1.5) rectangle (3.5,1.5);
    \draw [ultra thin] (-4,-1.5) --(-3.5,-1.5);
    \draw [ultra thin] (-4,1.5) --(-3.5,1.5);
    \draw [ultra thin] (4,-1.5) --(3.5,-1.5);
    \draw [ultra thin] (4,1.5) --(3.5,1.5);
    \draw [ultra thin] (-3.5,-2) --(-3.5,-1.5);
    \draw [ultra thin] (3.5,-2) --(3.5,-1.5);
    \draw [ultra thin] (3.5,2) --(3.5,1.5);
    \draw [ultra thin] (-3.5,2) --(-3.5,1.5);
    \node at (0,0) {$\cR'$};
    \node at (-3.745,0) {$\cR_L$ };
    \node at (3.745,0) {$\cR_R$ };
    \node at (0,1.75){$\cR_T$ };
    \node at (0,-1.75){$\cR_B $ };
    \filldraw[gray] (-3.1,-1) circle [radius=1pt]
    (2.85,0.3) circle [radius=1pt];
    \node at (-2.9,-0.8){$u$ };
    \node at (3.05,0.5){$v$ };
    \node [name=start] at (3.75,1) { };
    \node [name=end] at (5.5,1) {dual path };
    \draw [->] (start)--(end);
    \draw (-3.7,2) .. controls (-3.5,1) and (-4,0) .. (-3.8,-2);
    \draw (-4,-1.7) .. controls (-2,-1.99) and (1,-2) .. (4,-1.7);
    \draw (3.95,-2) .. controls (3.5,-0.5) and (4,1) .. (3.95,2);
    \draw (4,1.85) .. controls (0,1.6) .. (-4,1.8);
    \end{tikzpicture}
    \caption{An illustration for $\mcc R_T, \mcc R_B, \mcc R_L, \mcc R_R$ and $\mcc R'$.}
    \label{fig:TBLR}
\end{figure}
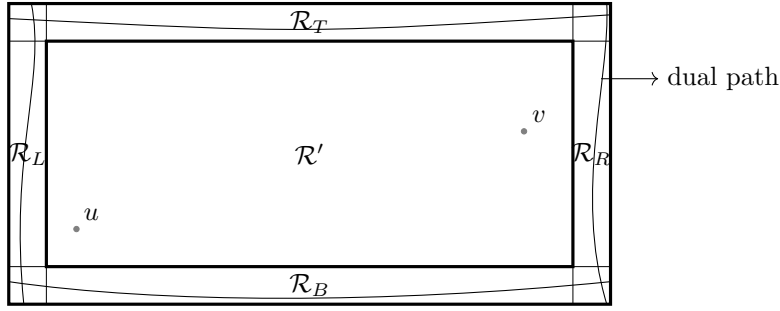

By the law of total probability, we have 
\begin{equation*}
    \begin{aligned}
        &\phi^{\mcc P,\mcc M}_{\Gamma}(u\longleftrightarrow v\centernot\longleftrightarrow \mcc P\cup \mcc M\mid\{\mcc P\centernot\longleftrightarrow \mcc M\})\\\geq~& \phi^{\mcc P,\mcc M}_{\Gamma}(u\longleftrightarrow v\centernot\longleftrightarrow \mcc P\cup \mcc M\mid \{\mcc P\centernot\longleftrightarrow \mcc M\}\cap \mcc B)\times \phi^{\mcc P,\mcc M}_{\Gamma}(\mcc B\mid \mcc P\centernot\longleftrightarrow \mcc M)\,.
    \end{aligned}
\end{equation*}
On the event $\mcc B$, there is a dual path disconnecting the boundary $\mcc P\cup \mcc M$ with $u,v$. Thus, by DMP and CBC, we obtain that
    \begin{align}
        \phi^{\mcc P,\mcc M}_{\Gamma}(u\longleftrightarrow v\centernot\longleftrightarrow \mcc P\cup \mcc M\mid \{\mcc P\centernot\longleftrightarrow \mcc M\}\cap \mcc B)&=\phi^{\mcc P,\mcc M}_{\Gamma}(u\longleftrightarrow v\mid\{\mcc P\centernot\longleftrightarrow \mcc M\}\cap\mcc B)\nonumber
        \\&\geq \phi^{\mathrm f}_{\cR'}(u\longleftrightarrow v)\geq C_2 dist(u,v)^{-\frac{1}{4}},\label{eq-otoB}
    \end{align}
where the last inequality follows from \cite[Proposition 27]{DHN11}.
Applying the RSW theory (\cite[Theorem 1.1]{DHN11})
we get that
$$\phi^{\mathrm w}_{\mss R_B}(\mcc H^{\diamond}( {\mss R}_B))\geq C_3\,, $$ and so are the other three terms on the right-hand side of \eqref{eq: dual path circuit} below. Combined with FKG, it yields that
   \begin{align}
    \phi^{\mcc P,\mcc M}_{\Gamma}(\mcc B)\geq \phi^{\mathrm w}_{\cR_B}(\mcc H^\diamond(\cR_B))\times \phi^{\mathrm w}_{\cR_T}(\mcc H^{\diamond}(\cR_T))\times \phi^{\mathrm w}_{\cR_L}(\mcc V^{\diamond}(\cR_L))\times\phi^{\mathrm w}_{\cR_R}(\mcc V^{\diamond}(\cR_R))\ge C_4.\label{eq: dual path circuit}
   \end{align}   Since $\{\mcc P\centernot\longleftrightarrow \mcc M\}$ and $\mcc B$ are decreasing events, by FKG, CBC and \eqref{eq: dual path circuit} we have \begin{equation}\label{eq-dual-path}
   \phi^{\mcc P,\mcc M}_{\Gamma}(\mcc B\mid  
   \mcc P\centernot\longleftrightarrow \mcc M)\stackrel{\mbox{FKG}}\ge\phi^{\mcc P,\mcc M}_{\Gamma}(\mcc B)\geq C_4.  
\end{equation}
Combining \eqref{eq-otoB} and \eqref{eq-dual-path} we obtain that \begin{equation*}\label{eq: first term in FK representation}
    \phi^{\mcc P,\mcc M}_{\Gamma}(u\longleftrightarrow v\mid \mcc P\centernot\longleftrightarrow \mcc M)\ge C_5dist(u,v)^{-\frac{1}{4}}.
\end{equation*}
Thus we complete the proof of Lemma~\ref{lem: FK connecting probability with Ising boundary}.
\end{proof}
\begin{lem}\label{lem: truncated correlation lower-bound uniformly for boundary condition}
Let $\cR$ be a rectangle with side lengths at least $M$ and at most $10M$. Let $\cR\subset\Gamma\subset \mathbb Z^2$. For any $u,v\in \cR$ satisfying $dist(u,\intB \cR)>M/10$ and $dist(v,\intB \cR
)>M/10$, we have (at $T_c$) that \begin{equation}\label{eq: truncated correlation lower-bound uniformly for boundary condition}
        \langle\sigma_u\sigma_v\rangle^{\xi}_{\Gamma}-\langle\sigma_u\rangle^{\xi}_{\Gamma}\cdot\langle\sigma_v\rangle^{\xi}_{\Gamma}\ge {\oc{A4}}dist(u,v)^{-\frac{1}{4}}
    \end{equation} for some $\nc\label{A4}>0$ not depending on the location of $u,v$ or the boundary condition $\xi$.
\end{lem}
\begin{proof}
    We first write the left-hand side of \eqref{eq: truncated correlation lower-bound uniformly for boundary condition} in terms of FK-configurations. By \eqref{eq:definition for P and M}, we have \begin{equation}\label{eq: spin spin correlation into FK-configurations}
    \begin{aligned}
        \langle\sigma_u \sigma_v\rangle^{\xi}_{\Gamma}=& \phi^{\mcc P,\mcc M}_{\Gamma}(u\longleftrightarrow v\longleftrightarrow \mcc P\mid \mcc P\centernot\longleftrightarrow \mcc M)+\phi^{\mcc P,\mcc M}_{\Gamma}(u\longleftrightarrow v\longleftrightarrow \mcc M\mid \mcc P\centernot\longleftrightarrow \mcc M)\\&-\phi^{\mcc P,\mcc M}_{\Gamma}(u\longleftrightarrow \mcc P, v\longleftrightarrow \mcc M\mid \mcc P\centernot\longleftrightarrow \mcc M)\\&-\phi^{\mcc P,\mcc M}_{\Gamma}(u\longleftrightarrow \mcc M, v\longleftrightarrow \mcc P\mid \mcc P\centernot\longleftrightarrow \mcc M)\\&+\phi^{\mcc P,\mcc M}_{\Gamma}(u\longleftrightarrow v\centernot\longleftrightarrow \mcc P\cup \mcc M\mid \mcc P\centernot\longleftrightarrow \mcc M).
    \end{aligned}
    \end{equation} In addition, we have \begin{equation}\label{eq: spin average into FK-configurations}
    \begin{aligned}
        \langle\sigma_u \rangle^{\xi}_{\Gamma}\cdot\langle\sigma_v \rangle^{\xi}_{\Gamma}=& \phi^{\mcc P,\mcc M}_{\Gamma}(u\longleftrightarrow \mcc P\mid \mcc P\centernot\longleftrightarrow \mcc M)\cdot\phi^{\mcc P,\mcc M}_{\Gamma}(v\longleftrightarrow \mcc P\mid \mcc P\centernot\longleftrightarrow \mcc M)\\&+\phi^{\mcc P,\mcc M}_{\Gamma}(u\longleftrightarrow \mcc M\mid \mcc P\centernot\longleftrightarrow \mcc M)\cdot\phi^{\mcc P,\mcc M}_{\Gamma}(v\longleftrightarrow \mcc M\mid \mcc P\centernot\longleftrightarrow \mcc M)\\&-\phi^{\mcc P,\mcc M}_{\Gamma}(u\longleftrightarrow \mcc P\mid \mcc P\centernot\longleftrightarrow \mcc M)\cdot\phi^{\mcc P,\mcc M}_{\Gamma}(v\longleftrightarrow \mcc M\mid \mcc P\centernot\longleftrightarrow \mcc M)\\&-\phi^{\mcc P,\mcc M}_{\Gamma}(u\longleftrightarrow \mcc M\mid \mcc P\centernot\longleftrightarrow \mcc M)\cdot\phi^{\mcc P,\mcc M}_{\Gamma}(v\longleftrightarrow \mcc P\mid \mcc P\centernot\longleftrightarrow \mcc M).
    \end{aligned}
    \end{equation}
 We apply the FKG inequality to $\bar\mu^\xi_\Gamma$ and obtain that {(recall our notations in Remark \ref{rmk:connecting events for extended Ising})}
 \begin{equation}\label{eq: FKG for FK model under Ising boundary}
    \begin{aligned}
        &\bar\mu^{\xi}_{\Gamma}(u\longleftrightarrow v\longleftrightarrow \mcc P)\ge \bar\mu^{\xi}_{\Gamma}(u\longleftrightarrow \mcc P)\cdot \bar\mu^{\xi}_{\Gamma}(v\longleftrightarrow \mcc P),\\
        &\bar\mu^{\xi}_{\Gamma}(u\longleftrightarrow v\longleftrightarrow \mcc M)\ge \bar\mu^{\xi}_{\Gamma}(u\longleftrightarrow \mcc M)\cdot \bar\mu^{\xi}_{\Gamma}(v\longleftrightarrow \mcc M),\\
        &\bar\mu^{\xi}_{\Gamma}(u\longleftrightarrow \mcc P,v\longleftrightarrow \mcc M)\le \bar\mu^{\xi}_{\Gamma}(u\longleftrightarrow \mcc P)\cdot \bar\mu^{\xi}_{\Gamma}(v\longleftrightarrow \mcc M),\\
        &\bar\mu^{\xi}_{\Gamma}(u\longleftrightarrow \mcc M,v\longleftrightarrow \mcc P)\le \bar\mu^{\xi}_{\Gamma}(u\longleftrightarrow \mcc M)\cdot \bar\mu^{\xi}_{\Gamma}(v\longleftrightarrow \mcc P).
    \end{aligned}
\end{equation}
Converting the inequalities in \eqref{eq: FKG for FK model under Ising boundary} to those for $\phi_\Gamma^{\mathcal P, \mathcal M}(\cdot \mid \mathcal P \centernot\longleftrightarrow \mathcal M)$ and combining them with \eqref{eq: spin spin correlation into FK-configurations} and \eqref{eq: spin average into FK-configurations}, we obtain that \begin{equation}\label{eq: reduction from truncated correlation function to FK connectivity}
    \langle\sigma_u\sigma_v\rangle^{\xi}_{\Gamma}-\langle\sigma_u\rangle^{\xi}_{\Gamma}\cdot\langle\sigma_v\rangle^{\xi}_{\Gamma}\ge \phi^{\mcc P,\mcc M}_{\Gamma}(u\longleftrightarrow v\centernot\longleftrightarrow \mcc P\cup \mcc M\mid \mcc P\centernot\longleftrightarrow \mcc M).
\end{equation}
Thus the desired bound follows directly from Lemma~\ref{lem: FK connecting probability with Ising boundary}.
\end{proof}

\section{Concentration inequalities for Gaussian variables}

In this appendix, we show two concentration inequalities for Gaussian variables. Our inequalities are in the same spirit of the well-studied concentration bounds for polynomials of Gaussian variables (see \cite{AL12,AW15,Lat06}). However, it seems the bounds in the literature did not provide explicit dependence on the order of the polynomials as we need in this paper. As such, we prove some concentration bounds for some specific quantities of interest with explicit control for the dependence on the order.  
\begin{lem}\label{lem: concentration for sum of products of tanh}
    Let $X_1,X_2,\cdots,X_n$ be i.i.d. Gaussian variables with mean $0$ and variance $1$. Let $Y_i=\tanh(\eps X_i)$. For any subset $I\subset \{1,2,\cdots,n\}$, define $Y_I=\prod_{i\in I}Y_i$, and let $a_I$ be some fixed real numbers only depending on $I$.
    Then there exists an absolute constant $\nc\label{appendixB1}>0$ such that for any $x>0$ and any positive integer $k$, we have \begin{equation*}\label{eq: concentration for sum of products of tanh}
        \P(\mid\sum_{|I|=k}a_IY_I\mid\ge x)\le \oc{appendixB1}^{-1}\exp(-\frac{\oc{appendixB1}x^{\frac{1}{k}}}{\eps(\sum_{|I|=k}a_I^2)^{\frac{1}{2k}}}).
    \end{equation*}
\end{lem}
\begin{proof}
We start by controlling the moments for $\mid\sum_{|I|=k}a_IY_I\mid$. For any even integer $p\ge 2$, we compute the $p$-th moment \begin{equation} \label{eq: moment expansion for general wiener chaos}
 \E \mid\sum_{|I|=k}a_IY_I\mid^p=\E \Big(\sum_{\substack{I_1,\cdots,I_p\\|I_j|=k}} \prod_{j=1}^p a_{I_j} \prod_{j=1}^p Y_{I_j}\Big)\le \sum_{\substack{I_1,\cdots,I_p\\|I_j|=k}}| \prod_{j=1}^p a_{I_j}| \cdot|\E\prod_{j=1}^p Y_{I_j}|.
\end{equation}
We denote $\mathcal I$ as the collection of $(I_1,\cdots,I_p)$ such that $|I_j|=k$ and $\sum_j\1_{i\in I_j}$ is even for all $1\le i\le n$. Then the expectation term on the right-hand side of \eqref{eq: moment expansion for general wiener chaos} is $0$ unless $(I_1, \ldots, I_p) \in \mathcal I$. Therefore, we get that (note that $\prod_{j=1}^p Y_{I_j} \geq 0$ for each $(I_1, \ldots, I_p)\in \mathcal I$) \begin{equation}\label{eq: reduction for expansion of mth moment}
    \sum_{\substack{I_1,\cdots,I_p\\|I_j|=k}}| \prod_{j=1}^p a_{I_j}| \cdot|\E\prod_{j=1}^p Y_{I_j}|=\sum_{(I_1,\cdots,I_p)\in \mathcal I} \prod_{j=1}^p |a_{I_j}| \cdot\E\prod_{j=1}^p  |Y_{I_j}|.
\end{equation} Then we obtain from the inequality $|Y_i|\le \eps |X_i|$ that \begin{equation}\label{eq: moment bound from tanh polynomials to gaussian polynomials}
\sum_{\substack{I_1,\cdots,I_p\\|I_j|=k}}| \prod_{j=1}^p a_{I_j}|\cdot |\E\prod_{j=1}^p Y_{I_j}|\le \sum_{(I_1,\cdots,I_p)\in \mathcal I} \prod_{j=1}^p |a_{I_j}| \E\Big(\prod_{j=1}^p \eps^k |X_{I_j}|\Big) =\eps^{kp}\cdot \E (\sum_{|I|=k}|a_I|X_I)^p.\end{equation}
Here, as for $Y_I$, we denote $X_I=\prod_{i\in I}X_i$. The last equality comes from the same reason as in \eqref{eq: reduction for expansion of mth moment}. We remark that the reason we have equality in \eqref{eq: moment bound from tanh polynomials to gaussian polynomials} instead of inequality in \eqref{eq: moment expansion for general wiener chaos} is that in \eqref{eq: moment bound from tanh polynomials to gaussian polynomials} we have taken absolute values for $a_I$s. Since $\sum_{|I|=k}|a_I|X_I$ is in Wiener chaos of order $k$ (for an explicit definition of Wiener chaos, see e.g. \cite[Section 1.1.1]{Nua06}), we obtain from hypercontractivity for Guassians \cite{Nel73} (see also \cite[Theorem 1.4.1]{Nua06}) that \begin{equation}\label{eq: hypercontractivity for gaussian polynomials} \E\Big|\sum_{|I|=k}|a_I|X_I\Big|^p \le (p-1)^{kp/2} \Big(\E(\sum_{|I|=k}|a_I|X_I)^2\Big)^{\frac{p}{2}}=(p-1)^{kp/2} \Big(\sum_{|I|=k}a_I^2\Big)^{\frac{p}{2}}.
\end{equation}
Combining \eqref{eq: moment expansion for general wiener chaos}, \eqref{eq: moment bound from tanh polynomials to gaussian polynomials} and \eqref{eq: hypercontractivity for gaussian polynomials}, we obtain that \begin{equation}\label{eq: moment bound for gaussian polynomials} \E \mid\sum_{|I|=k}a_IY_I\mid^p\le \eps^{kp}(p-1)^{kp/2} (\sum_{|I|=k}a_I^2)^{\frac{p}{2}}.
\end{equation}
Now we compute the exponential moment for $|\sum_{|I|=k}a_IY_I|^{\frac{1}{k}}$, via computing its moments of all orders. For $p\le 2k$, we apply the Jensen inequality and obtain that \begin{equation}\label{eq: low degree moment bound for wiener chaos}
    \E|\sum_{|I|=k}a_IY_I|^{\frac{p}{k}}\le (\E|\sum_{|I|=k}a_IY_I|^2)^{\frac{p}{2k}}\le(\sum_{|I|=k}a_I^2)^{\frac{p}{2k}}\cdot \eps^p.
\end{equation}
For $p>2k$, we assume $p=2bk+c$ with $b\ge 1$ and $0\le c<2k$. Applying H\"{o}lder inequality and \eqref{eq: moment bound for gaussian polynomials}, we get that (by viewing $|\sum_{|I| = k} a_I Y_I|^{2b+\frac{c}{k}} =  |\sum_{|I| = k} a_I Y_I|^{\frac{2b(2k-c)}{2k}}\cdot |\sum_{|I| = k} a_I Y_I|^{\frac{(2b+2)c)}{2k}}$) \begin{equation}\label{eq: high degree moment bound for wiener chaos}
    \E|\sum_{|I|=k}a_IY_I|^{2b+\frac{c}{k}}\le (\E|\sum_{|I|=k}a_IY_I|^{2b})^{\frac{2k-c}{2k}}\cdot (\E|\sum_{|I|=k}a_IY_I|^{2b+2})^{\frac{c}{2k}}\stackrel{\eqref{eq: moment bound for gaussian polynomials}}\le(\sum_{|I|=k}a_I^2)^{\frac{p}{2k}}\eps^p\cdot (2b+1)^{p/2}.
\end{equation}
Combining \eqref{eq: low degree moment bound for wiener chaos} and \eqref{eq: high degree moment bound for wiener chaos}, we obtain that for all $p\geq 1$ \begin{equation}\label{eq: full moment bound for general gaussian polynomials}
    \E|\sum_{|I|=k}a_IY_I|^{\frac{p}{k}}\le  p!\cdot (\sum_{|I|=k}a_I^2)^{\frac{p}{2k}}\eps^p
\end{equation} since $(2b+1)^{\frac{p}{2}}\le (p+1)^{\frac{p}{2}}\le p!$. For any parameter $t\in(0,\frac{1}{\eps(\sum_{|I|=k}a_I^2)^{\frac{1}{2k}}})$, we can now compute the exponential moment by applying \eqref{eq: full moment bound for general gaussian polynomials} and obtain that \begin{equation*}\label{eq: exponential moment for general wiener chaos}\begin{aligned}
    \exp(t|\sum_{|I|=k}a_IY_I|^{\frac{1}{k}}) = \sum_{p=0}^{\infty}\E|\sum_{|I|=k}a_IY_I|^{\frac{p}{k}}\cdot  t^p/p!&\le  \sum_{p=0}^{\infty} (\sum_{|I|=k}a_I^2)^{\frac{p}{2k}} (\eps t)^p\\&=\frac{1}{1-(\sum_{|I|=k}a_I^2)^{\frac{1}{2k}}\cdot \eps t}.
\end{aligned}
\end{equation*}
Applying (the exponential version of) Markov's inequality with $t=\frac{1}{2\eps(\sum_{|I|=k}a_I^2)^{\frac{1}{2k}}}$, we get that $$\P(\mid\sum_{|I|=k}a_IY_I\mid\ge x)\le \exp(-tx^{\frac{1}{k}})\frac{1}{1-(\sum_{|I|=k}a_I^2)^{\frac{1}{2k}}\cdot \eps t}=2\exp(-\frac{x^{\frac{1}{k}}}{2\eps(\sum_{|I|=k}a_I^2)^{\frac{1}{2k}}}).$$
Thus we complete the proof of Lemma~\ref{lem: concentration for sum of products of tanh}.
\end{proof}
  
As an analog to Lemma~\ref{lem: concentration for sum of products of tanh}, we can also prove the following lemma.
\begin{lem}\label{lem: concentration for sum of products of tanh2}
    We continue to use the notations in Lemma \ref{lem: concentration for sum of products of tanh}. For $I, J \subset \{1, \ldots, n\}$, let $a_{I,J}$ be some fixed real numbers only depending on $I$ and $J$. Then there exists an absolute constant $\nc\label{appendixB}>0$ such that  for any $x>0$ and non-negative integers $m,k$ such that $m+k\ge 1$, we have \begin{equation*}\label{eq: concentration for sum of products of tanh2}
        \P\Big(\mid\sum_{\substack{|I|=k,|J|=m}}a_{I,J}Y_IY_J\mid\ge x\Big)\le \oc{appendixB}^{-1}\exp\Big(-\frac{\oc{appendixB}x^{\frac{1}{k+m}}}{\eps(\max\limits_{0\le r\le \min\{m,k\}}A_r)^{\frac{1}{k+m}}}\Big)
    \end{equation*}where $A_r=n^{\frac{r}{2}}\Big[\sum_{|I|=k,|J|=m,|I\cap J|\ge r}|a_{I,J}|^2\Big]^{\frac{1}{2}}$.
\end{lem}
\begin{proof}
    The proof is highly similar to that of Lemma~\ref{lem: concentration for sum of products of tanh}. As a result, we only provide a sketch emphasizing the additional subtleties. Without loss of generality, we assume that $m\le k$ and thus $\min\{m,k\}=m$.
    
    The only difference is that $\sum_{|I|=k,|J|=m }|a_{I,J}|X_IX_J$ is no longer in Wiener chaos of order $m+k$. Nevertheless, we still have that $\sum_{|I|=k,|J|=m }|a_{I,J}|X_{I\Delta J}\cdot \prod_{i\in I\cap J}(X_i^2-1)$ is in Wiener chaos of order $m+k$. To this end, we define $Z_i=X_i^2-1$ and we expand the summation as follows: \begin{equation}
        \sum_{|I|=k,|J|=m }|a_{I,J}|X_IX_J=\sum_{|I|=k,|J|=m }|a_{I,J}|X_{I\Delta J}\cdot \prod_{i\in I\cap J}(Z_i+1)=\sum_{r=0}^k\Phi_r
    \end{equation} where $$\Phi_r=\sum_{|I|=k,|J|=m }|a_{I,J}|X_{I\Delta J}\cdot \sum_{U\subset I\cap J,|U|=|I\cap J|-r}Z_U\,.$$ Here we recall our convention that $Z_U=\prod_{i\in U}Z_i$. Since $\Phi_r$ is in Wiener chaos of order $m+k-2r$, for any even integer $p\ge 2$ we get from hypercontractivity for Guassians \cite[Theorem 1.4.1]{Nua06} \begin{equation*}
        \Big[\E(\Phi_r)^p\Big]^{\frac{1}{p}}\le (p-1)^{\frac{m+k-2r}{2}}\Big[\E(\Phi_r)^2\Big]^{\frac{1}{2}}.
    \end{equation*}
Combined with Minkowski inequaility, it yields that \begin{equation}\label{eq:1}
\Big[\E(\sum_{|I|=k,|J|=m }|a_{I,J}|X_IX_J)^p\Big]^{\frac{1}{p}}=\Big[\E(\sum_{r=0}^m\Phi_r)^p\Big]^{\frac{1}{p}}\le \sum_{r=0}^m(p-1)^{\frac{m+k-2r}{2}}\Big[\E(\Phi_r)^2\Big]^{\frac{1}{2}}.
\end{equation}
In order to get a bound similar to \eqref{eq: hypercontractivity for gaussian polynomials}, it suffices to upper-bound $\E(\Phi_r)^2$ for each $0\le r\le m$. Next we fix $0\le r\le m$. In order to make $\{U\subset I\cap J,|U|=|I\cap J|-r\}$ nonempty, we must have that $|I\cap J|\ge r$ and we denote the collection of such pairs $(I,J)$ as $\mathfrak{B}_r$. Noting that $X_i$ has mean $0$ and variance $1$ and $Z_i$ has mean $0$ and variance $2$ and $\E X_iZ_j=0$ for any $1\le i,j\le n$, we obtain that \begin{align}
    \E(\Phi_r)^2&=\sum_{(I_1,J_1),(I_2,J_2)\in \mathfrak{B}_r}\sum_{\substack{U_1\subset I_1\cap J_1,\\|U_1|=|I_1\cap J_1|-r}}\sum_{\substack{U_2\subset I_2\cap J_2,\\|U_2|=|I_2\cap J_2|-r}} |a_{I_1,J_1}||a_{I_2,J_2}|\E\big( X_{I_1\Delta J_1}X_{I_2\Delta J_2}\cdot Z_{U_1} Z_{U_2}\big)\nonumber\\&=\sum_{(I_1,J_1),(I_2,J_2)\in \mathfrak{B}_r}\sum_{\substack{U_1\subset I_1\cap J_1,\\|U_1|=|I_1\cap J_1|-r}}\sum_{\substack{U_2\subset I_2\cap J_2,\\|U_2|=|I_2\cap J_2|-r}}|a_{I_1,J_1}||a_{I_2,J_2}|\1_{\{I_1\Delta J_1=I_2\Delta J_2\}}\1_{\{U_1=U_2\}}\cdot 2^{|U_1|}\nonumber\\&\le\sum_{(I_1,J_1),(I_2,J_2)\in \mathfrak{B}_r}\sum_{\substack{U\subset I_1\cap J_1\cap I_2\cap J_2,\\|U|=|I_1\cap J_1|-r\\=|I_2\cap J_2|-r}}\frac{|a_{I_1,J_1}|^2+|a_{I_2,J_2}|^2}{2}\cdot\1_{\{I_1\Delta J_1=I_2\Delta J_2\}} 2^{|I_2\cap J_2|-r}\,.\label{eq:2}
\end{align}
Next, viewing the right-hand side of \eqref{eq:2} as linear combinations of $\{a_{I, J}^2\}$, we compute the coefficient for $|a_{I_1,J_1}|^2$. Since $I_1\Delta J_1=I_2\Delta J_2$ and $|I_1\setminus J_1|=|I_2\setminus J_2|$, the number of choices for $(I_2\setminus J_2, J_2\setminus I_2)$ is $\binom{|I_1\Delta J_1|}{|I_1\setminus J_1|}\le 2^{|I_1\Delta J_1|}$. Further, the number of choices for $U$ is $\binom{|I_1\cap J_1|}{r}\le 2^{|I_1\cap J_1|}$. For any fixed $U$, the number of choices for $I_2\cap J_2\setminus U$ is at most $n^{r}$. As a result, the coefficient for $|a_{I_1,J_1}|^2$ is at most $$2^{|I_1\Delta J_1|}\cdot 2^{|I_1\cap J_1|}\cdot n^r\cdot 2^{|I_1\cap J_1|-r}\le 4^{m+k}n^r.$$ Combined with \eqref{eq:1} and \eqref{eq:2}, it yields that \begin{align}
    \Big[\E(\sum_{r=0}^m\Phi_r)^p\Big]^{\frac{1}{p}}&\le \sum_{r=0}^m(p-1)^{\frac{m+k-2r}{2}}\Big[\sum_{(I_1,J_1)\in \mathfrak{B}_r}4^{m+k}n^r|a_{I_1,J_1}|^2\Big]^{\frac{1}{2}}\nonumber\\&= \sum_{r=0}^m(p-1)^{\frac{m+k-2r}{2}}2^{m+k}A_r\le (m+1)(p-1)^{\frac{m+k}{2}}2^{m+k}\max_{0\le r\le m} A_r.\label{eq: hypercontractivity for gaussian polynomials2}
\end{align}
With \eqref{eq: hypercontractivity for gaussian polynomials2} in place of \eqref{eq: hypercontractivity for gaussian polynomials}, it is straightforward to adapt the proof of Lemma~\ref{lem: concentration for sum of products of tanh} and obtain a proof of Lemma~\ref{lem: concentration for sum of products of tanh2} as required.
\end{proof}
\medskip

\noindent {\bf Acknowledgement.} We thank Rongfeng Sun and Hao Wu for helpful discussions, and in particular, Rongfeng's explanation of \cite{BS22} was very inspiring to us. We thank Ron Peled for helpful discussions and in particular for pointing out a mistake in an earlier version of the manuscript (which is now incorporated in Remark~\ref{rmk-ron}). J. Ding is partially supported by NSFC Tianyuan Key Program Project No. 12226001, NSFC Key Program Project No. 12231002 and an Explorer Prize.

\end{document}